\documentclass[10pt]{amsart}
\usepackage{amsmath, amscd, amsthm} 
\usepackage{amssymb, amsfonts} 
\usepackage[all]{xy} 
\subjclass[2010]{14F42, 19E15, 55P42, 55P91}
\keywords{Equivariant motivic homotopy, Bredon motivic cohomology, Betti realization}
\usepackage{aliascnt}
\usepackage[svgnames]{xcolor} 
\usepackage{hyperref}
  \definecolor{dark-red}{rgb}{0.4,0.15,0.15}
  \hypersetup{
    colorlinks, linkcolor=dark-red,
    citecolor=DarkBlue, urlcolor=MediumBlue
}  
\newcommand{\C}{\mathbb{C}} 
\newcommand{\A}{\mathbb{A}}
\renewcommand{\P}{\mathbb{P}}
\newcommand{\Z}{\mathbb{Z}}
\newcommand{\R}{\mathbb{R}}
\newcommand{\G}{\mathbb{G}}
\renewcommand{\L}{\mathbb{L}}

\newcommand{\MM}{\mathcal{M}}
\renewcommand{\SS}{\mathbb{S}}

\newcommand{\RRe}{\mathrm{Re}}
 
\newcommand{\iso}{\cong}
\newcommand{\wkeq}{\simeq}

\newcommand{\Sch}{\mathrm{Sch}}
\newcommand{\Sm}{\mathrm{Sm}}
\newcommand{\sch}{\mathrm{Sch}}
\newcommand{\sm}{\mathrm{Sm}}

\newcommand{\spc}{\mathrm{Spc}}
\newcommand{\topp}{\mathrm{Top}}

\newcommand{\Ab}{\mathrm{Ab}}

\newcommand{\spre}{\mathrm{sPre}}

\newcommand{\sspt}{\mathrm{Spt}^{\Sigma}}
\newcommand{\GMS}{G\mathrm{Spc}_{\bullet}}
\newcommand{\MS}{\mathrm{Spc}}
\newcommand{\bMS}{\mathrm{Spc}_{\bullet}}
\newcommand{\bGH}{\mathrm{H}_{G,\bullet}}
\newcommand{\HH}{\mathrm{H}_{\bullet}}

\newcommand{\SH}{\mathrm{SH}}

\newcommand{\rH}{\widetilde{H}}

\newcommand{\rG}{\rho_{G}}
\newcommand{\EG}{\mathbf{E}}
\newcommand{\BG}{\mathbf{B}}
\newcommand{\EGt}{\widetilde{\EG} }
\newcommand{\Es}{\mathrm{E}_{\bullet}}

\newcommand{\Th}{\mathrm{Th}}

\newcommand{\id}{\mathrm{id}}

\newcommand{\wt}[1]{\widetilde{#1}}

\newcommand{\mcal}[1]{\mathcal{#1}}
\newcommand{\ul}[1]{\underline{\smash{#1}}}
\newcommand{\MZ}{\mathbf{M}\ul{\Z}}
\newcommand{\MA}{\mathbf{M}\ul{A}}
\newcommand{\M}[1]{\mathbf{M}\ul{#1}}

\DeclareMathOperator*{\colim}{\mathrm{colim}}

\DeclareMathOperator{\diag}{\mathrm{diag}}
\DeclareMathOperator{\spec}{\mathrm{Spec}}
\DeclareMathOperator{\proj}{\mathrm{Proj}}

\DeclareMathOperator{\cone}{\mathrm{cone}}

\newcommand{\Sym}{\mathrm{Sym}}

\DeclareMathOperator{\coker}{\mathrm{coker}}
\DeclareMathOperator{\codim}{\mathrm{codim}}

\DeclareMathOperator{\Hom}{Hom}
\DeclareMathOperator{\Ext}{Ext}

\DeclareMathOperator{\Cor}{Cor}
\DeclareMathOperator{\Aut}{Aut}
\DeclareMathOperator{\ch}{char}

\DeclareMathOperator{\Tot}{Tot}

\DeclareMathOperator{\msing}{\mathcal{S}\mathrm{ing}}

\newcommand{\shom}{\underline{\mathrm{Hom}}}
\newcommand{\ihom}{\mathbf{hom}}

\newcommand{\cd}{\smash\cdot}

\numberwithin{equation}{section} 

\theoremstyle{plain} 
 
\newaliascnt{theorem}{equation}  
\newtheorem{theorem}[theorem]{Theorem}  
\aliascntresetthe{theorem}

\newaliascnt{proposition}{equation}  
\newtheorem{proposition}[proposition]{Proposition}
\aliascntresetthe{proposition}

\newaliascnt{lemma}{equation}    
\newtheorem{lemma}[lemma]{Lemma}
\aliascntresetthe{lemma}

\newaliascnt{corollary}{equation}  
\newtheorem{corollary}[corollary]{Corollary}
\aliascntresetthe{corollary}

\newaliascnt{claim}{equation}  

\aliascntresetthe{claim}

 \theoremstyle{definition}

\newaliascnt{definition}{equation}  
\newtheorem{definition}[definition]{Definition}
\aliascntresetthe{definition}

\newaliascnt{example}{equation}  

\aliascntresetthe{example}

\newaliascnt{remark}{equation}   
\newtheorem{remark}[remark]{Remark}
\aliascntresetthe{remark}

\newaliascnt{condition}{equation}

\aliascntresetthe{condition}

\newaliascnt{notationconvention}{equation}

\aliascntresetthe{notationconvention}

\newcommand{\aref}[1]{\autoref{#1}}

\begin{document}
\title[Topological comparison theorems for Bredon motivic cohomology]{Topological comparison theorems for \\ Bredon motivic cohomology}

\author{J. Heller}
\email{jeremiahheller.math@gmail.com}
\address{Department of Mathematics, University of Illinois, Urbana-Champaign}
\author{M. Voineagu}
\email{m.voineagu@unsw.edu.au}
\address{UNSW Sydney, NSW 2052 Australia}
\author{P. A. {\O}stv{\ae}r}
\email{paularne@math.uio.no}
\address{Department of Mathematics, University of Oslo, Norway}

\begin{abstract}
We prove equivariant versions of the Beilinson-Lichtenbaum conjecture for Bredon motivic cohomology of smooth complex and real varieties 
with an action of the group of order two.
This identifies equivariant motivic and topological invariants in a large range of degrees.
\end{abstract} 
\maketitle
\tableofcontents

\section{Introduction}

A major achievement of motivic homotopy theory is the proof of the Bloch-Kato conjecture relating Milnor $K$-theory to Galois cohomology \cite{Voev:miln}, 
\cite{Voev:BK},  
and consequently the solution of the Beilinson-Lichtenbaum conjecture by the results in \cite[\S7]{SV:BK}.
With finite mod-$n$ coefficients and $X$ a smooth scheme of finite type over a field $k$ of characteristic coprime to $n$, 
one form of the Beilinson-Lichtenbaum conjecture asserts that the comparison map between motivic and \'etale cohomology 
\begin{equation}
\label{equation:BLconjecturemap}
H^{p,q}_{\mcal{M}}(X,\Z/n)\to H^{p}_{et}(X,\mu_{n}^{\otimes q})
\end{equation}
is an isomorphism when $p\leq q$
\cite[Conjecture 6.8]{SV:BK} and \cite[Chapter 1]{OrangeBook}.
This ``\'etale descent'' property identifies a large range of motivic cohomology groups, 
a.k.a.~higher Chow groups, 
with more computable \'etale cohomology groups.
For a smooth complex variety $X$, 
the \'etale cohomology groups in (\ref{equation:BLconjecturemap}) agree with the singular cohomology groups of its corresponding analytic space $X(\C)$.
Thus the Beilinson-Lichtenbaum conjecture provides a powerful link between algebro-geometric and topological invariants.
This can further be enhanced to prove the Quillen-Lichtenbaum conjectures comparing the algebraic and hermitian $K$-theories of $X$ with their analytic or \'etale counterparts, 
see e.g., 
\cite[Theorem 4.7]{SuslinICM94}, \cite[Theorem 7.10]{Voev:miln} and \cite[Theorem 5.1]{BKSO}.

A framework for invariants of smooth varieties equipped with a group scheme action has recently been organized in the subject of equivariant motivic homotopy theory
\cite{Deligne:V}, \cite{HKO}, \cite{HKO:EMHT}, \cite{Herrmann}, \cite{Hoyois:six}, \cite{CJ}.
Of particular interest are actions by the group $C_{2}$ of order two governing the examples of hermitian $K$-theory and motivic Real cobordism \cite{HKO}.
Bredon motivic cohomology introduced in \cite[\S5]{HVO:cancellation} is an equivariant generalization of motivic cohomology for finite group actions.
This theory is now amenable to a homotopical analysis on account of the equivariant cancellation theorem shown in \cite[Theorem 9.7]{HVO:cancellation}.

We employ this setup to prove an equivariant Beilinson-Lichtenbaum comparison theorem for smooth complex (and real) varieties. Before stating our main theorem, we review the players involved. If $M$ is a topological space with $C_2$-action and $A$ is an abelian group, the Bredon cohomology groups $H^{a+p\sigma}_{C_2}(M, \ul{A})$ are an
 equivariant analog of singular cohomology groups. Here $a,p$ are integers and $\sigma$ stands for the sign representation, i.e., topological Bredon cohomology is graded by virtual representations of $C_2$. 
  Now if $X$ is a smooth complex $C_{2}$-scheme of finite type, the Bredon motivic cohomology groups 
 $H^{a+p\sigma, b+q\sigma}_{C_{2}}(X,\ul{A})$ have a grading by a $4$-tuple of integers. This $4$-tuple is a pair of virtual representations $V= a+p\sigma$ and $W=b+q\sigma$ which are respectively the ``cohomological degree" and the ``weight" of the grading.

In \aref{app:EMHT} we construct a natural comparison map between Bredon motivic cohomology and its equivariant topological counterpart
\begin{equation}
\label{equation:equivariantBLconjecturemap}
H^{a+p\sigma, b+q\sigma}_{C_{2}}(X,\ul{A})
\to 
H^{a+p\sigma}_{C_{2}}(X(\C),\ul{A}).
\end{equation}

The map in (\ref{equation:equivariantBLconjecturemap}) is induced by a Betti realization functor for the stable $C_{2}$-equivariant motivic homotopy category.
In fact, 
the groups on the left hand side of (\ref{equation:equivariantBLconjecturemap}) are represented by a $C_2$-equivariant Bredon motivic cohomology spectrum $\M{A}$, 
whose Betti realization agrees with the $C_2$-equivariant Eilenberg-MacLane spectrum for the constant Mackey functor $\ul{A}$, 
cf.~\aref{definition:MBC}, \aref{subsection:topologicalrealization(stable)}, and \aref{thm:realbred}.

A virtual $C_2$-representation $U=a+p\sigma$ has the dimension given by $\dim(U)=a+p$ and $\dim (U^{C _2})=a$. The Beilinson-Lichtenbaum comparison theorem says that the comparison map \eqref{equation:BLconjecturemap} is an isomorphism (resp.~ injection) for cohomological degrees less than or equal than the weight (resp.~ weight plus one). 
Our main result, which appears as \aref{theorem:equivariantBLconjecture} below, 
asserts that the condition of isomorphism depends on the dimensions of the representations and their fixed points.  A precise statement is the following. 
\begin{theorem} 
 Let $X$ be a smooth complex scheme of finite type with $C_2$-action. 
Write $V=a+p\sigma$ and $W=b+q\sigma$.  
 Then, the comparison map
 (\ref{equation:equivariantBLconjecturemap}) is an isomorphism for $a+p\leq b+q$ and $a\leq \min\{b,b-q\}$ and an injection when  $a+p\leq b+q+1$, and  $a\leq \min\{b,b-q\}+1$ for any finite abelian group $A$. 
\end{theorem}

The condition $a\leq b-q$ cannot be improved as we can see from  \aref{rem}. From this theorem it also results that in the above  range of degrees it follows that $H^{a+p\sigma,b+q\sigma}_{C_{2}}(X,\ul{A})$ is a finite abelian group.

We also establish an equivariant version of the Beilinson-Lichtenbaum comparison theorem in the case of smooth real $C_{2}$-schemes of finite type, see \aref{equivBLconjecture}. In this case, if $X$ is a finite-type $C_2$-scheme over $\R$, the space of complex points $X(\C)$ has a $C_2\times C_2$-action. One copy of $C_2$ acts algebraically through the action on $X$ and the other acts via complex conjugation. Thus \aref{equivBLconjecture} is a comparison result which relates    $C_2$-equivariant Bredon motivic cohomology of $X$ with $C_2\times C_2$-equivariant topological Bredon cohomology of $X(\C)$.

We expect that over other fields, there is an analogous comparison between Bredon motivic cohomology and an appropriate equivariant \'etale analog of the topological Bredon cohomology. However it doesn't appear that this has been developed yet. Doing so here would lead us too far astray and so we leave this for a future paper.

The nonequivariant Beilinson-Lichtenbaum comparison theorem is a fundamental ingredient in the proof of \aref{theorem:equivariantBLconjecture}. We leverage this nonequivariant comparison result to an equivariant comparison by means of a motivic version of the isotropy separation cofiber sequence together with some computations in Borel motivic cohomology. We emphasize that these arguments rely both on the representability of 
Bredon motivic cohomology in the stable equivariant motivic homotopy category as well as that it can be computed as sheaf hypercohomology. That the representable theory coincides with sheaf hypercohomology is basically a recasting of the homotopy invariance and equivariant cancellation theorems, established in \cite{HVO:cancellation}, and is proved in \aref{thm:motivicrep} below.

The equivariant Beilinson-Lichtenbaum comparison theorem is a first important step towards understanding the Bredon motivic cohomology ring 
$H^{\star,\star}_{C_{2}}(k,\ul{\Z/2})$ of a field $k$.
As in (\ref{equation:equivariantBLconjecturemap}), 
the gradings are sums of $C_{2}$-representations.
These invariants are fundamental for understanding key features of $C_2$-equivariant motivic homotopy theory, 
e.g., 
$H^{\star,\star}_{C_{2}}(k,\ul{\Z/2})$ forms part of the largely unknown $C_2$-equivariant motivic Steenrod algebra of cohomology operations.
Another fundamental aspect of Bredon motivic cohomology is that the zeroth slice of the $C_2$-equivariant motivic sphere spectrum turns out 
to be the highly structured $C_2$-equivariant Bredon motivic cohomology spectrum $\M{\Z}$ introduced in \aref{definition:MBC},
cf.~\cite{HO:zeroslices}. 

This paper is structured as follows.
Sections 2-7 are devoted to the proof of our main result for the comparison map \eqref{equation:equivariantBLconjecturemap}.
The proof rests on techniques from equivariant motivic homotopy theory: 
roughly speaking, 
taking complex or real points induces a Betti realization functor which in turn gives rise to the comparison map for Bredon motivic cohomology.
The details of these constructions are found in \aref{app:EMHT}.
In \aref{section:emht} we develop the homotopical techniques that we need in the sequel;
in particular, 
the motivic isotropy separation cofiber sequence plays a central role in our approach.
The computational core of the paper lies in Sections 3-7. 
First, 
in \aref{section:Bredonmotivi9ccohomology},  
we establish that Bredon motivic cohomology is represented by $\M{\Z}$ and show that it affords Thom classes for a certain class of equivariant vector bundles as well as Gysin sequences. In \aref{sub:genborel} we 
compare Bredon motivic cohomology with Edidin-Graham's equivariant higher Chow groups.
In \aref{section:periodicityandBmc} we show that the generalized ``geometric'' Borel motivic cohomology ring 
$H^{\star,\star}_{C_2}(\EG C_2,\ul{\Z/2})$ is periodic with period $(2\sigma-2,\sigma-1)$. 
These preliminary results are used to prove the complex and real comparison theorems for Bredon motivic cohomology in \aref{s:C} and \aref{s:R}, 
respectively.

\subsection*{Notation and Conventions:}
Throughout $k$ is a perfect field of characteristic $\ch(k)\neq 2$. 
Unless said otherwise, a scheme is always assumed to be separated. 
For a finite group $G$, let $G\sch/k$ denote the category of separated schemes of finite type over $k$ with left $G$-actions and equivariant maps. Similarly, $G\sm/k$ is the category of smooth schemes of finite type over $k$ with left $G$-actions and equivariant maps. We use the term 
{\em $k$-variety } synonymous with separated, finite type, scheme over $k$.

We write $\A(V) = \spec (\Sym (V^{\vee}))$ for the affine scheme associated to a vector space $V$ over $k$, 
and $\P(V) = \proj(\Sym (V^{\vee}))$ for the associated projective scheme.

The construction of the stable $G$-equivariant motivic homotopy category $\SH_{G}(k)$ is recalled in Appendix \ref{app:EMHT}. 
We write $[-,-]_{G}$ for maps in $\SH_{G}(k)$.
We distinguish four sphere objects in the  $C_{2}$-equivariant motivic homotopy category, see \aref{sub:C2}. 
These are denoted $S^{1}$, $S^{\sigma}$, $S^{1}_{t}$, and $S^{\sigma}_{t}$.
The sphere $S^{1}$ is the usual simplicial sphere and $S^{\sigma}$ the simplicial sign representation sphere. 
The algebro-geometric sphere $S^{1}_{t}$ is the pointed scheme $(\G_{m},1)$ equipped with trivial action and $S^{\sigma}_{t}$ is the pointed scheme $(\G_{m},1)$ 
equipped with the inversion action, $x\mapsto x^{-1}$.
We write $V=a+p\sigma$ for the $C_2$-representation which is the sum of $a$-copies of the trivial representation and $p$-copies of the sign representation, 
and define 
\begin{equation}
\label{eqn:funnyspheres}
S^{a+p\sigma, b + q\sigma}
:=
S^{a-b}\wedge S^{(p-q)\sigma}\wedge S^{b}_{t}\wedge S^{q\sigma}_{t}.
\end{equation}
We adopt the convention that $*$ refers to an integer grading of homotopy or cohomology groups while $\star$ refers to grading by representations.

\section{Background}\label{section:emht}
Equivariant motivic homotopy theory was introduced by Voevodsky \cite{Deligne:V} as a tool for understanding symmetric products and motivic Eilenberg-MacLane spaces. 
Stable equivariant motivic homotopy category was introduced by Hu-Kriz-Ormsby \cite{HKO} as part of their study of the homotopy limit problem for hermitian $K$-theory of fields. 
In this section we recall definitions and basic results about equivariant motivic homotopy theory. 
Technical details and a fairly complete, self-contained discussion are given in \aref{app:EMHT}.

\subsection{Equivariant Nisnevich topology}
\label{EquivariantNisnevichtopology}
The equivariant Nisnevich topology was introduced by Voevodsky \cite[\S3]{Deligne:V}. 
See \cite[\S3]{HVO:cancellation} and \cite[\S2]{HKO:EMHT} for more details concerning the equivariant Nisnevich topology.
\begin{definition}
\label{definition:eNt}
An \emph{equivariant distinguished square}
\begin{equation}\label{eqn:edist}
\xymatrix{
 W\ar[r]\ar[d] & Y \ar[d]^{p} \\
 U \ar@{^{(}->}[r]^{i} & X
}
\end{equation}
is a cartesian square in $G\Sch/k$ such that $p$ is \'etale, $i:U\subseteq X$ is an open embedding, and $p$ induces an isomorphism of reduced schemes
$(Y\setminus W)_{red} \iso (X\setminus U)_{red}$. An elementary Nisnevich cover is the cover $\{U\to X, Y\to X\}$ associated to an equivariant distinguished square.
The \emph{equivariant Nisnevich topology} on  $G\Sm/k$ is the smallest Grothendieck topology containing the elementary Nisnevich covers. 
\end{definition}

Recall that the set-theoretic stabilizer $S_{x}$ of a point $x\in X$ is the subgroup $S_{x}:=\{g\in G\mid gx =x\}$. 
By \cite[Proposition 3.5]{HVO:cancellation}, 
an equivariant \'etale map $f:Y\to X$ is an equivariant Nisnevich cover if and only if for every $x\in X$ there is $y\in Y$ such that $f$ induces an 
isomorphism $k(y)\iso k(x)$ on residue fields as well as an isomorphism $S_{x}\iso S_{y}$ on set-theoretic stabilizers. 

\begin{remark}\label{rem:localaffine}
Using this characterization of Nisnevich covers we see that every smooth $G$-scheme is locally affine. 
Every point $x\in X$ has an $S_{x}$-invariant affine neighborhood (take any affine neighborhood and consider the intersection of the translates by elements of $S_{x}$).
Let $U_{x}$ be such a neighborhood. 
Then $G\times^{S_{x}}U_{x}\to X$,
where for a subgroup $H\subseteq G$ and $H$-scheme $Z$ we write $G\times^{H}Z$ for $(G\times Z)/H$,  
is an equivariant Nisnevich neighborhood of $G\cd x$ in the sense of \cite[Section 2]{HKO:EMHT}. 
The collection $\{G\times^{S_{x}}U_{x}\to X\}$ is an ``infinite'' equivariant Nisnevich cover of $X$ by smooth affine $G$-schemes admitting a finite subcover by 
\cite[Remark 2.18]{HKO:EMHT}. 
\end{remark}

\subsection{Motivic {$G$}-spaces and spectra}

A \emph{motivic $G$-space} over $k$ is a presheaf of simplicial sets on $G\sm/k$. 
We write $G\MS(k)$ and $G\bMS(k)$ respectively for the categories of motivic $G$-spaces and pointed motivic $G$-spaces over $k$. 
The unstable equivariant motivic homotopy category is constructed following a pattern familiar from ordinary motivic homotopy theory; 
technical details are found in \aref{app:EMHT}. 
In brief, 
it is the homotopy category of a model structure which is constructed so that the following two relations hold.
\begin{enumerate}
 \item[(i)] Any equivariant distinguished square (\ref{eqn:edist}) is a homotopy cocartesian square.
 \item[(ii)] The projection $X\times\A^{1}\to X$ is an equivariant motivic weak equivalence for any $X$ in $G\sm/k$.
\end{enumerate}
These relations have non-obvious consequences.
For example,   
by the Whitehead theorem, 
a map inducing isomorphisms on equivariant Nisnevich sheaves of homotopy groups is an equivariant motivic weak equivalence, 
cf.~\cite[\S3.2, Proposition 2.14]{MV:A1}.
Moreover,
every $G$-equivariant vector bundle is an equivariant motivic weak equivalence \cite[Proposition 4.10]{HKO:EMHT}.

Let $V$ be a representation of $G$, 
e.g., 
the regular representation $\rho_{G}=k[G]$.
The associated \emph{motivic representation sphere} is defined to be the pointed motivic $G$-space 
$$
T^{V}:=\P(V\oplus 1)/\P(V).
$$
For an integer $n\geq 0$ we use the smash product in $G\bMS(k)$ to define  
$$
T^{nV}:=(T^{V})^{\wedge n}.
$$

Consider the equivariant distinguished square
$$
\xymatrix{
\A(V)\setminus\{0\}\ar[r] \ar[d] & \A(V) \ar[d] \\
\P(V\oplus 1)\setminus \P(1) \ar[r] & \P(V\oplus 1).
}
$$ 
It is in particular a homotopy cocartesian square. 
The inclusion of $\P(V\oplus 1)\setminus\P(1)$ into $\P(V\oplus 1)$ is equivariantly $\A^{1}$-homotopic to the inclusion $\P(V)\subseteq \P(V\oplus 1)$; 
the requisite deformation is given by  $([x_{0}:\cdots:x_{n+1}], t)\mapsto [x_{0}:\cdots:x_{n}:tx_{n+1}]$. 
We conclude there is an equivariant motivic weak equivalence
$$
T^{V}\wkeq \A(V)/\A(V)\setminus\{0\}.
$$
Also note that $T^{nV} \wkeq T^{V^{\oplus n}}$ and more generally $T^{V}\wedge T^{W}\wkeq T^{V\oplus W}$.

The stable equivariant motivic homotopy category $\SH_{G}(k)$ is the stabilization of $G\bMS(k)$ with respect to the sphere $T^{\rho_{G}}$. 
We use symmetric $T^{\rho_{G}}$-spectra as a model for $\SH_{G}(k)$, see
\aref{subsection:Stableequivariantmotivichomotopytheory}. It is a tensor triangulated category with unit the sphere spectrum $\mathbf{1} = \Sigma^{\infty}_{T^{\rho_{G}}}S^{0}$. 
Here, $S^{0}=\spec(k)_{+}$ is the unit for the smash product in $G\bMS(k)$.
If $X$ is an unbased motivic $G$-space, e.g.,  a smooth $G$-scheme, we have an associated based motivic $G$-space $X_+$, by adding a disjoint basepoint, 
and an associated suspension spectrum $\Sigma^{\infty}_{T^{\rho_{G}}}X_+$. 
When no confusion should arise, 
we sometimes simply write $X$ for $\Sigma^{\infty}_{T^{\rho_{G}}}X_+$ and $S^{0}$ for the sphere spectrum.

\subsection{{$C_2$}-equivariant spheres}\label{sub:C2}
When $G= C_2$ there are only two representations, 
the trivial representation and the sign representation which we write as $\sigma$. It is convenient  to introduce the following sphere objects.  
The \emph{sign Tate sphere} $S^{1}_{t}$ is the pointed $C_2$-scheme $(\G_{m},1)$ where $\G_{m}$ is equipped with the action $x\mapsto x^{-1}$.
The \emph{simplicial sign representation sphere} $S^{\sigma}$ is defined to be the unreduced suspension of $C_2$, 
i.e., 
it is the homotopy cofiber of $C_{2\,+}\to S^{0}$.
We have as well the usual simplicial sphere $S^1$ and the Tate sphere $S^{1}_{t}$  which is the pointed scheme $(\G_{m},1)$ considered with trivial action. 
As observed in \cite[\S4.1]{HKO}, 
there is an equivariant motivic weak equivalence $S^{\sigma}\wedge S^{\sigma}_{t}\wkeq T^{\sigma}$. 

The indexing convention 
\begin{equation*}
S^{a+p\sigma, b + q\sigma}
:=
S^{a-b}\wedge S^{(p-q)\sigma}\wedge S^{b}_{t}\wedge S^{q\sigma}_{t}
\end{equation*}
in (\ref{eqn:funnyspheres}) is a mixture between the convention standardly used in motivic homotopy theory and that used in classical equivariant homotopy theory.  
The translation between the convention of indexing here and the one in 
\cite[\S4.1]{HKO} is given by
$$
S^{a+p\sigma, b + q\sigma} = S^{(a-b) +(p-q)\gamma + b \alpha + q\gamma\alpha}.
$$
The convention we use in this paper has the feature that the effect of the complex Betti realization functor (constructed in \aref{sub:top}) is the first entry of the index, 
$$
\RRe_{\C}(S^{a+p\sigma, b + q\sigma}) = S^{a+p\sigma}.
$$ 
A real Betti realization (taking value in $\SH_{C_2\times \Sigma_2}$) is constructed in \aref{sub:topR}. In this case, 
$$
\RRe_{\C,\,\Sigma_2}(S^{a+p\sigma, b + q\sigma}) = 
S^{a-b+ (p-q)\sigma +b\epsilon +q\sigma\otimes\epsilon},
$$
where $\sigma$ is the sign representation corresponding to the factor $C_2$ and $\epsilon$ is the sign representation corresponding to $\Sigma_2$.

The following two fundamental homotopy cofiber sequences of pointed motivic $C_2$-spaces are useful for computations, 
\begin{equation}\label{eqn:fun1}
C_{2\,+} \to S^{0} \to S^{\sigma, 0}
\end{equation}
and
\begin{equation}\label{eqn:fun2}
(\A(n\sigma)\setminus \{0\})_+ \to S^{0} \to S^{2n\sigma, n\sigma}.
\end{equation}
Here, 
the first maps in these sequences are induced respectively by the projections $C_2\to \spec(k)$ and $\A(n\sigma)\setminus \{0\}\to \spec(k)$.

\subsection{Motivic isotropy separation}
The isotropy separation cofiber sequence is a fundamental tool for analyzing equivariant homotopy types in classical equivariant homotopy theory. 
Our proof of the main Theorems \ref{theorem:equivariantBLconjecture} and \ref{equivBLconjecture} makes use of an appropriate motivic version. 
We shall restrict our attention to the group $C_2$. 
See \cite{GH:S0} for a general discussion of motivic isotropy separation. 

Recall that the classical topological isotropy separation cofiber sequence is 
$$
\Es C_{2\,+} \to S^{0} \to \wt{\rm E}_{\bullet} C_{2},
$$
where the last term is defined by this sequence. 
A check of fixed points shows that a model for $\wt{\rm E}_{\bullet}C_2$ is given by
\begin{equation}\label{eqn:colim}
\wt{\rm E}_{\bullet} C_{2}\wkeq \colim_{n}S^{n\sigma}. 
\end{equation}

The \emph{geometric classifying space} $\BG C_2$ for $C_2$ is defined as the quotient of 
$$
\EG C_{2} := \colim_{n} \A(n\sigma)\setminus\{0\}
$$
by the free $C_{2}$-action.
This space plays an important role in nonequivariant motivic homotopy theory because it is a geometric model for the \'etale classifying space \cite[\S4.2]{MV:A1}. 
A similar interpretation of the geometric classifying space is also true in equivariant motivic homotopy theory as well, 
see \cite{GH:S0}.

The \emph{motivic isotropy separation cofiber sequence} is the cofiber sequence
\begin{equation}
\label{eqn:motisotropy}
\EG C_{2\,+}\to S^{0}\to \wt{\EG}C_{2},
\end{equation}
where the space $\wt{\EG}C_{2}$ is defined by this cofiber sequence. 
Because of the definition of the geometric classifying space we have an equivariant motivic equivalence
\begin{equation}
\label{eqn:colim2}
\EGt C_{2} \wkeq \colim_{n} S^{2n\sigma, n\sigma}.
\end{equation}

\begin{proposition}\label{prop:invariance}
The maps $S^{0}\to S^{2\sigma, \sigma}$ and $S^{0}\to S^{\sigma, 0}$ induce equivariant motivic equivalences $\EGt  C_{2}\wkeq S^{2\sigma,\sigma}\wedge \EGt  C_{2}$
and $\EGt  C_{2}\wkeq S^{\sigma, 0}\wedge\EGt  C_{2}$.
\end{proposition}
\begin{proof}
The first equivalence follows from (\ref{eqn:colim2}) together with the fact that the cyclic permutation on $T^{3\sigma}$ is the identity by \aref{lem:cyclicistrivial} below. 
Using \aref{lem:trick}, and that $\EG C_{2,+}$ is nonequivariantly contractible \cite[Proposition 4.2.3]{MV:A1},
we have  equivariant motivic equivalences 
$$
\EG C _{2+}\wedge C _{2\,+}\simeq (\EG C_{2\,+})^{e}\wedge C _{2+}
\simeq C _{2\,+}.
$$
It follows that $(C_{2})_{+}\wedge \EGt  C_{2}  \wkeq *$. 
Thus the second equivalence follows from the cofiber sequence (\ref{eqn:fun1}). 
\end{proof}

\begin{lemma}\label{lem:cyclicistrivial}
	Let $W$ be a representation. Let $\gamma\in\Sigma_k$ be an even permutation. Then the induced map on $T^{kW}$ is the identity map in 
	${\rm H}_{G,\bullet}(k)$.
\end{lemma}
\begin{proof}
	Since $W^{\oplus k} = 1^{\oplus k}\otimes W$, any automorphism of $1^{\oplus k}$ induces an equivariant automorphism of $W^{\oplus k}$ and thus a pairing  ${\rm GL}_k\times \A(W^{\oplus k})\to \A(W^{\oplus k})$ in $G\Sm_k$ (where ${\rm GL}_k$ has trivial action). This induces in turn a pairing $({\rm GL}_k)_+\wedge T^{kW}\to T^{kW}$ in $G\bMS(S)$. 
	An even permutation $\gamma$ is the product of row multiplication and addition of elementary matrices. We therefore have a map $\A^1\to {\rm GL}_{k}$ connecting the identity and $\gamma$. This implies that the map on $T^{kW}$ induced by an even permutations is equivariantly $\A^1$-homotopic to the identity.  
\end{proof}

\section{Bredon motivic cohomology}
\label{section:Bredonmotivi9ccohomology}
In \cite[\S5]{HVO:cancellation} we introduced a Bredon motivic cohomology theory on $G\sm/k$.   
Here we define an equivariant motivic spectrum which is a representing object for the Bredon motivic cohomology groups of loc.~cit., 
and record some of its basic properties. 
To keep the exposition streamlined, 
we restrict our discussion to the group $C_2$.

\subsection{Motivic complexes}\label{sub:motiviccomplex}
For $X$ and $Y$ smooth $k$-schemes of finite type we write $\Cor_{k}(X,Y)$ for the group of finite correspondences \cite[\S3.1]{SV:BK}. 
If $X$ and $Y$ have an action by $C_2$ then $C_2$ acts on $\Cor_{k}(X,Y)$ as well.
The category of equivariant correspondences $C_2\Cor_{k}$ has the same objects as $C_2\Sm/k$ and maps  $\Cor_{k}(X,Y)^{C_2}$. 
A \emph{presheaf with equivariant transfers} on $C_2\Sm/k$ is an additive presheaf on $C_2\Cor_k$. 
 
If $Y$ is a smooth $C_2$-scheme over $k$, write
$\Z_{tr,C_2}(Y)$ for the free presheaf with equivariant transfers, 
$$
\Z_{tr,C_2}(Y)(X)
:=
\Cor_{k}(X,Y)^{C_2}.
$$ 
More generally, 
if $A$ is an abelian group, 
we write 
$A_{tr,C_2}(Y)=\Z_{tr,C_2}(Y)\otimes A$.
It is useful to extend the definition of $A_{tr,C_2}(-)$ to quotients of $G$-schemes. If $\mathcal{X}=\colim_{i} X_i$, where $X_i$ are smooth $C_2$-schemes
over $k$ and the colimit is in the category of presheaves,
then we define $A_{tr,C_2}(\mathcal{X}) := \colim_{i} A_{tr,C_2}(X_{i})$, where the colimit is computed in the category of presheaves of abelian groups.

Defining $\Z_{tr, C_2}(X)\otimes^{tr}\Z_{tr,C_2}(Y) := \Z_{tr, C_2}(X\times Y)$ determines a symmetric monoidal product $\otimes^{tr}$ on the category of presheaves with equivariant transfers. 

If $W$ is a finite $C_2$-set, viewed as a smooth $C_2$-scheme over $k$, we have isomorphisms as presheaves on $C_2\Sm/k$, 
$\Z_{tr, C_2}(W) \iso \Z(W):=\Z \Hom _k(-,W)^{C_2}$. 
Moreover, if $F$ is a presheaf with equivariant transfers, 
then we also have an isomorphism $F\otimes^{tr}\Z_{tr,C_2}(W)\iso F\otimes \Z(W)$ of presheaves on $C_2\Sm/k$.

If $F$ is a presheaf of abelian groups we write $C_{*}F(X)$ for the simplicial abelian group $F(X\times\Delta^{\bullet}_{k})$. 
We use the same notation for the associated cochain complex.
Here, 
$\Delta^{\bullet}_{k}$ is the standard cosimplicial object in $\Sm/k$.

Write $D^{-}(C_2\Cor_k)$ for the derived category of bounded above chain complexes of equivariant Nisnevich sheaves with equivariant transfers on $C_2\Sm/k$.
According to \cite[Lemma 5.12]{HVO:cancellation} the cone $\Z_{top}(\sigma)$ of the map $\Z_{tr,C_2}(C_2)\to \Z$ is invertible in $(D^{-}(C_2\Cor_k),\otimes^{tr})$.
If $F_\bullet$ is a cochain complex of presheaves with equivariant transfers, write $F_{\bullet}[\sigma]= F_{\bullet}\otimes^{tr}\Z_{top}(\sigma)$.

Let $V=a+b\sigma$ be a representation of $C_2$ and $A$ an abelian group. 
Define the \emph{motivic Bredon complex}
$$
\ul{A}(V) := a_{{\rm Nis}}(C_*A_{tr,C_2}(T^{V}))[-2a-2b\sigma],
$$
where $a_{{\rm Nis}}$ denotes sheafification in the equivariant Nisnevich topology, 
see \aref{definition:eNt}. 
There are quasi-isomorphisms $a_{{\rm Nis}}(C_*A_{tr,C_2}(T))\wkeq a_{{\rm Nis}}(C_*A_{tr,C_2}(S^1_t))[1]$ and $a_{{\rm Nis}}(C_*A_{tr,C_2}(T^{\sigma}))\wkeq a_{{\rm Nis}}(C_*A_{tr,C_2}(S^\sigma_t))[\sigma]$, 
see \cite[p.328]{HVO:cancellation}. 
It follows that there is a quasi-isomorphism 
\begin{equation}\label{eqn:ab}
a_{{\rm Nis}}(C_*A_{tr,C_2}(T^{V}))\wkeq a_{{\rm Nis}}(C_*A_{tr,C_2}(S^{a}_t\wedge S^{b\sigma}_t))[a+b\sigma].
\end{equation}
 
Now let $A$ be a commutative ring. 
We construct a product pairing 
$$
\ul{A}(V)\otimes \ul{A}(W)\to \ul{A}(V\oplus W).
$$ 
First, we have an associative pairing 
$A_{tr,C_2}(T^{V})\otimes A_{tr,C_2}(T^{W}) \to A_{tr,C_2}(T^{V}\wedge T^{W})$ of presheaves which is induced by the pairing
\begin{align*}
A_{tr,C_2}(\P(V\oplus 1))(U)\otimes & A_{tr,C_2}(\P(W\oplus 1))(U) \\
& \xrightarrow{\times}
A_{tr,C_2}(\P((V\oplus 1)\times \P(W\oplus 1))(U \times  U) 
\\
& \xrightarrow{\Delta^{*}} A_{tr,C_2}(\P(V\oplus 1)\times \P(W\oplus 1))(U).
\end{align*}
Let $A_{*,*}$ be a bisimplicial abelian group. We write $A_{*,*}$ as well for the associated cochain complex.
By the Eilenberg-Zilber theorem \cite[Theorem IV.2.4]{GoerssJardine}, 
taking totalizations and diagonals yields a natural quasi-isomorphism of chain complexes $\Tot(A_{*,*})\to \diag(A_{*,*})$.
We thus obtain the natural pairing 
\begin{align*}
\Tot \big( A_{tr,C_2}&(T^{V})(U\times \Delta^{\bullet}_{k}) \otimes A_{tr,C_2}(T^{W})(U\times\Delta^{\bullet}_{k}) \big) 
\\
&\to 
\diag\big(A_{tr,C_2}(T^{V})(U\times \Delta^{\bullet}_{k})\otimes A_{tr,C_2}(T^{W})(U\times\Delta^{\bullet}_{k}) \big)
\\
& \to A_{tr,C_2}(T^{V}\wedge T^{W})(U\times \Delta^{\bullet}_{k})\xleftarrow{\iso} 
A_{tr,C_2}(T^{V\oplus W})(U\times \Delta^{\bullet}_{k}).
\end{align*}
This induces our desired pairing upon sheafification. 
Now the quasi-isomorphism obtained from the Eilenberg-Zilber theorem is homotopy associative, 
so the pairing is also homotopy associative.

\subsection{Stable representability}\label{sub:stablerep}
Let $A$ be an abelian group. Let $\mathcal{F}$ be a presheaf of sets. 
we may consider $A_{tr,C_2}(\mathcal{F})$ as a presheaf of sets and therefore as a based motivic $C_2$-space, the basepoint is $0$.  
We have a canonical map $\gamma: \mathcal{F}\to A_{tr,C_2}(\mathcal{F})$ of motivic $C_2$-spaces and there is a pairing  
$\mu:A_{tr,C_2}(X)\wedge \Z_{tr,C_2}(Y) \to A_{tr,C_2}(X \times Y)$ of motivic $C_2$-spaces.

\begin{definition}
\label{definition:MBC} 
Let $A$ be an abelian group. The \emph{motivic Bredon cohomology spectrum} $\MA$,  
is defined by letting $\MA_{n} := A_{tr,C_2}(T^{n\rho_{C_2}})$ with structure maps
\begin{align*}
A_{tr,C_2}(T^{ n\rho_{C_2}}) \wedge T^{\rho_{C_2}} 
\xrightarrow{\id\wedge \gamma} &
A_{tr,C_2}(T^{ n\rho_{C_2}})\wedge \Z_{tr,C_2}(T^{\rho_{C_2}})
\\
\xrightarrow{\mu} & A_{tr,C_2}(T^{ (n+1)\rho_{C_2}}).
\end{align*}
\end{definition}

The symmetric group $\Sigma_{n}$ acts on  $\MA_{n}$ by permuting the factors of $T^{\rho_{C_2}}$. 
The iterated structure maps  $\MA_{n}\wedge T^{ k\rho_{C_2}} \to \MA_{n+k}$ are $(\Sigma_{n}\times \Sigma_{k})$-equivariant. 
This means that $\MA$ is a symmetric motivic $C_2$-spectrum, 
cf.~\aref{def:spectra}. Moreover, if $A$ is a commutative ring, there are pairings $\MA_{n}\wedge \MA_{k}\to \MA_{n+k}$ which give $\MA$ the structure of a commutative ring spectrum 
(i.e., a commutative monoid in the category of equivariant symmetric motivic spectra $\sspt_{C_2}(k)$).

Finally, we note that the spectrum $\MA$ is stably equivalent to $\MZ\wedge \SS A$ where $\SS A$ is a Moore spectrum associated to $A$. This can be seen by noting that the assignment $A\mapsto \MA$ defines a functor $\Ab \to \sspt_{C_2}(k)$
which has the following properties.
\begin{enumerate}
\item[(i)] If $A = \colim_{i} A_{i}$ is a filtered colimit of abelian groups, then $\colim_{i} \MA_i\wkeq \MA$.
\item[(ii)] If $A = \oplus_{j} A_j$, then $\vee_{j} \MA_j\wkeq \MA$.
\item[(iii)] If $0\to A_1\to A_2\to  A_3\to 0$ is an exact sequence of abelian groups, then $\MA_1 \to \MA_2 \to \MA_3$ is a homotopy cofiber sequence of spectra.
\end{enumerate}
Moore spectra satisfy similar properties and so it suffices to check the statement when $A=\Z$, where it is trivial. Of course, an important advantage of the model $\MA$ is that it is a commutative ring ring spectrum when $A$ is a commutative ring while  this is not necessarily a priori clear for $\MZ\wedge \SS A$.

\begin{definition}

The \emph{Bredon motivic cohomology} of a motivic $C_2$-spectrum $\mathsf{E}$, 
with coefficients in the abelian group $A$, 
is defined by

$$
\rH^{a+p\sigma, b+q\sigma}_{C_{2}}(\mathsf{E}, \ul{A}):= [\mathsf{E}, S^{a+p\sigma, b+q\sigma}\wedge\MA]_{\SH_{C_2}(k)}.
$$
If $X$ is a smooth $C_2$-scheme, 
then its unreduced Bredon motivic cohomology is defined via its suspension spectrum by setting
$$
H^{a+p\sigma, b+q\sigma}_{C_2}(X, \ul{A}):= 
\rH^{a+p\sigma, b+q\sigma}_{C_2}(\Sigma_{T^{\rho_{C_2}}}^{\infty} X_+,\ul{A}).
$$
\end{definition}

Next we verify that the  definition of Bredon motivic cohomology which we have just given agrees with the one given in \cite{HVO:cancellation}. 
(Note, however, that the indexing we use in the present paper is slightly different than in loc.~cit.). 
The fact that Bredon motivic cohomology defined in the stable equivariant motivic homotopy category is equal to the hypercohomology groups of the motivic complexes, 
plays a crucial role in our arguments in the later sections. 
This fact relies on the homotopy invariance and equivariant cancellation theorems for presheaves with equivariant transfers proved in 
\cite[Theorem 8.12, Theorem 9.7]{HVO:cancellation}.

If $\mcal{X} = \colim X_i$ is a colimit (in presheaves) of smooth $C_2$-schemes over $k$,  write $\Z(\mcal{X}) := \colim_i\Z(X_i)$ and 
 $$
 H_{C_{2}Nis}^{a}(\mcal{X}, \ul{A}(V)[p\sigma]) := 
 \Ext_{C_{2}Nis}^{a}(\Z(\mcal{X}), \ul{A}(V)[p\sigma]).
 $$ 
\begin{theorem}\label{thm:motivicrep}
Let $V= b+q\sigma$ be a virtual representation of $C_2$, $W=c+d\alpha$ a representation such that $V\oplus W$ is a representation, 
$X$ a smooth $C_2$-scheme of finite type over $k$, and $A$ an abelian group. 
Then there is a natural isomorphism
$$
H^{a+p\sigma, b+q\sigma}_{C_2}(X,\ul{A})
\cong
H_{C_{2}Nis}^{a}(T^{W}\wedge X_+, \ul{A}(W\oplus V)[2c+(2d+p)\sigma]).
$$
\end{theorem}
\begin{proof}
We assume that $V=b +q\sigma$ is an actual representation; 
the more general case of a virtual representation is similar.
Let $\MA\to \MA'$ be a levelwise motivic fibrant replacement, 
i.e., 
for each $n$,  
$\MA'_n$ is motivic fibrant and $\MA_n\to \MA'_n$  is a motivic weak equivalence 
(see \aref{subsection:Stableequivariantmotivichomotopytheory} for details on the motivic model structure).
We claim that $\MA'$ is already a fibrant motivic $C_2$-spectrum. 
Indeed, 
using \aref{thm:rcoh}, 
the map 
$$
\pi_{i}(\MA'_{n})(X)\to \pi_{i}(\Omega_{T^{\rho_{C_2}}}\MA'_{n+1})(X)
$$ 
is naturally identified with the  map 
$$
H^{-i}_{C_2Nis}(X, \ul{A}(n\rho_{C_2}))\to H^{-i}_{C_2Nis}(T^{\rho_{C_2}}\wedge X_+, \ul{A}((n+1)\rho_{C_2})).
$$ 
This is the map of the equivariant Cancellation Theorem \cite[Theorem 9.8]{HVO:cancellation}, 
which is an isomorphism for all $i$. 
Therefore,
$\MA'_{n}\to \Omega_{T^{\rho_{C_2}}}\MA'_{n+1}$ is a weak equivalence of motivic $C_2$-spaces, 
which implies that $\MA'$ is an $\Omega_{T^{\rho_{C_2}}}$-spectrum and so is a fibrant motivic $C_2$-spectrum, 
cf.~\aref{subsection:Stableequivariantmotivichomotopytheory}.

Let $i,j,k,l$ be nonnegative integers such that
$$
S^{(i+k) +(j+l)\sigma, k+l\sigma}\wedge S^{a+p\sigma, b+q\sigma} \wkeq T^{m\rho_{C_2}}
$$
for some $m\geq 0$.  
In particular, 
$m = i+a-b = j+p-q = k+b = l +q$.
We have $H^{a+p\sigma, b+q\sigma}_{C_2}(X,\ul{A})=[X_+,S^{a+p\sigma, b+q\sigma}\wedge \MA]_{\SH_{C_2}(k)}$ and
\begin{align*}
[ X_+, &   S^{a+p\sigma, b+q\sigma}\wedge \MA]_{\SH_{C_2}(k)} \\
& \iso
[S^{(i+k) +(j+l)\sigma, k+l\sigma}\wedge X_+,  T^{m\rho_{C_2}}\wedge \MA]_{\SH_{C_2}(k)}  
\\
& \iso
[S^{i}\wedge S^{j\sigma}\wedge S^{k}_{t}\wedge S^{l\sigma}_{t}\wedge X_+,  T^{m\rho_{C_2}}\wedge \MA']_{\SH_{C_2}(k)} \\
& \iso
[S^{i}\wedge S^{j\sigma}\wedge S^{k}_{t}\wedge S^{l\sigma}_{t}\wedge X_+, (\MA')_m]_{{\rm H}_{\bullet,C_2}(k)} 
\\
& \iso H^{-i}_{C_{2}Nis}(S^{j\sigma}\wedge S^{k}_{t}\wedge S^{l\sigma}_{t}\wedge X_+, C_*A_{tr,C_2}(T^{m})) 
\\
& \iso H^{-i}_{C_{2}Nis}(S^{j\sigma}\wedge S^{k}_{t}\wedge S^{l\sigma}_{t}\wedge X_+, C_*A_{tr,C_2}(S^{m}_t\wedge S^{m\sigma}_t)[m +m\sigma])
 \\
 & \iso H^{0}_{C_{2}Nis}( X, C_*A_{tr,C_2}(S^{m-k}_t\wedge S^{(m-l)\sigma}_t)[(m-i) +(m-j)\sigma])
 \\
 & \iso
 H^{a}_{C_{2}Nis}(X, \ul{A}(V)[p\sigma]).
\end{align*}

The first two isomorphisms are immediate. 
The third follows from \cite[Theorem 8.10]{Hovey:spt} and the standard adjunction relating  ${\rm H}_{\bullet,C_2}(k)$ and $\SH_{C_2}(k)$, 
the fourth from \aref{thm:rcoh}, the remaining isomorphisms  follow from (\ref{eqn:ab}) and equivariant cancellation \cite[Theorem 9.8]{HVO:cancellation}.

\end{proof}  

\begin{remark}
Under the isomorphisms above, the product structures arising from the pairing of spectra $\MA\wedge \MA \to \MA$ agrees with that arising from the pairing of complexes 
$\ul{A}(V)\otimes \ul{A}(W)\to \ul{A}(V\oplus W)$.
\end{remark}

\subsection{Basic properties}
We record some of the basic properties of Bredon motivic cohomology.

\subsubsection{Cancellation}
Let $V= s+t\sigma$ be a virtual representation. We defined Bredon motivic cohomology via a representing spectrum in $\SH_{C_2}(k)$. Immediate from this definition  we have natural isomorphisms
\begin{equation}\label{eqn:suspiso}
\rH^{a+p\sigma, b+q\sigma}_{C_{2}}(T^{V}\wedge \mathsf{E}, \ul{A}) \iso \rH^{a-2s+(p-2t)\sigma, b-s+(q-t)\sigma}_{C_{2}}(\mathsf{E}, \ul{A}).
\end{equation}
\subsubsection{Equivariant transfers}\label{subsub:transfers}
Recall we write $D^{-}(C_2\Cor_k)$ for the derived category of bounded above chain complexes of equivariant Nisnevich sheaves with equivariant transfers on $C_2\Sm/k$.
\begin{proposition}
	Let $X$ be a smooth $C_2$-scheme over $k$ and $K_{\bullet}$ a cochain complex of presheaves with equivariant transfers. Then there is a natural isomorphism
	$\Ext^{n}_{D^{-}(C_2\Cor_k)}(\Z_{tr,C_2}(X), K_{\bullet})\iso H^{n}_{C_2Nis}(X, K_{\bullet})$.
\end{proposition}
\begin{proof} 
	The argument is as in \cite[Proposition 3.1.9]{Voevodsky:triangulated}, using that
	smooth $C_2$-schemes have finite equivariant Nisnevich cohomological dimension \cite[Corollary 3.9]{HVO:cancellation} together with \cite[Theorem 4.15]{HVO:cancellation}. 
\end{proof}
  
\begin{corollary}
If $K_{\bullet}$ is a cochain complex of sheaves with equivariant transfers then $H^{n}_{C_2Nis}(-, K_{\bullet})$ is a presheaf with equivariant transfers.
\end{corollary}

\aref{thm:motivicrep} shows that for any $a,b,p,q$,  
$H^{a+p\sigma, b+q\sigma}_{C_2}(X,\ul{A})$ is naturally identified with equivariant Nisnevich hypercohomology groups with coefficients in a (bounded above) 
cochain complex of sheaves with equivariant transfers. 
It is therefore a presheaf with equivariant transfers. 

\subsubsection{Mayer-Vietoris sequences} 
Since Bredon motivic cohomology is a representable theory, 
associated to an equivariant distinguished square (\ref{eqn:edist}) is a Mayer-Vietoris long exact sequence
\begin{align*}
\cdots \to H^{a+p\sigma, b+q\sigma}_{C_2}(X,\ul{A}) \to &
H^{a+p\sigma,b+q\sigma}_{C_2}(U,\ul{A})\oplus H^{a+p\sigma, b+q\sigma}_{C_2}(Y,\ul{A})
\\
\to & H^{a+p\sigma, b+q\sigma}_{C_2}(W,\ul{A}) \to H^{a+1+p\sigma, b+q\sigma}_{C_2}(X,\ul{A}) \to \cdots .
\end{align*}

\subsubsection{Ring structure}
When $A$ is a commutative ring,
then $\MA$ is a commutative ring spectrum. 
We thus have a cup product pairing for smooth $C_2$-varieties
$$
H^{a+p\sigma, b+q\sigma}_{C_2}(X, \ul{A})\times H^{c+s\sigma, d+t\sigma}_{C_2}(X,\ul{A}) \xrightarrow{\cup} 
H^{(a+c)+(p+s)\sigma, (b+d) + (q+t)\sigma}_{C_2}(X, \ul{A}) 
$$
given by the usual formula 
$$
x\cup y := \Delta^*(x \boxtimes y),
$$ 
where $\Delta:X\to X\times X$ is the diagonal and $\boxtimes$ is the external product. 
The cup product makes $H^{\ast,\ast}_{C_2}(X,\ul{A})$ into a $\Z^{4}$-graded ring. 
Since it is a representable theory, 
it is  a module over $\pi_{0}^{C_2}(\mathbf{1}):={\rm End}_{\SH_{C_2}(k)}(\mathbf{1})$ and in fact an algebra over this ring.

Consider the following endomorphisms of $\mathbf{1}$ in $\SH_{C_2}(k)$
\begin{align}
\nonumber
\epsilon = & \Sigma^{-2}_{S^1_t}\left(S^1_t\wedge S^1_t\xrightarrow{{\rm twist}} S^1_t\wedge S^1_t\right), \\
\label{eqn:elements}
\epsilon'= & \Sigma^{-2}_{S^{\sigma}_t}\left(S^{\sigma}_t \wedge S^{\sigma}_t \xrightarrow{{\rm twist}} S^{\sigma}_t\wedge S^{\sigma}_t\right), \\
u=& \Sigma^{-2}_{S^{\sigma}}\left(S^{\sigma}\wedge S^{\sigma} \xrightarrow{{\rm twist}} S^{\sigma}\wedge S^{\sigma}\right).
\nonumber
\end{align}
 
We also write $\epsilon, \epsilon', u\in H^{0,0}_{C_2}(k,\ul{A})$ for the respective elements in cohomology, 
induced by these endomorphisms. 
 
\begin{proposition}\label{prop:skewcomm}
Let $x\in H^{a+p\sigma, b+q\sigma}_{C_2}(X, \ul{A})$ and $y\in H^{c+s\sigma, d+t\sigma}_{C_2}(X, \ul{A})$. 
Then 
$$
x\cup y = (-1)^{ac}(u)^{ps}(\epsilon)^{bd}(\epsilon')^{qt} (y\cup x).
$$
\end{proposition}
\begin{proof}
This follows by the same argument as \cite[Proposition 6.13]{dugger:coherence} (see also Remark 6.14 of loc.~cit.). 
The point is to carefully analyze the endomorphism of $S^{\alpha_1+\alpha_2, \beta_1 +\beta_2}$ arising from the twist 
$S^{\alpha_1, \beta_1} \wedge S^{\alpha_2,  \beta_2} \to S^{\alpha_2, \beta_2} \wedge S^{\alpha_1, \beta_1}$, 
where $(\alpha_1, \beta_1) = (a+p\sigma, b+q\sigma)$ and $(\alpha_2, \beta_2) = (c+s\sigma, d+t\sigma)$.
\end{proof}

\begin{remark}
We will see in \aref{prop:product} below 
that $\epsilon = 1$, $\epsilon'=1$, 
and $u= -1$ in $H^{\star,\star}_{C_2}(k,\ul{A})$ and therefore $ x\cup y = (-1)^{ac+ps} (y\cup x)$.
\end{remark}

\subsubsection{Topological realization}
Let $k=\C$. 
The functor $C_2\Sm/\C \to C_2{\rm Top}$, $X\mapsto X(\C)$ extends to a functor $\RRe_{\C}:\SH_{C_2}(\C)\to \SH_{C_2}$ from the stable equivariant motivic homotopy category over $\C$ 
to the classical stable equivariant homotopy category, 
see \aref{sub:top}. 
For a topological space $M$ with $C_2$-action we write $H^{a+p\sigma}_{C_2}(M,\ul{A})$ for the topological Bredon cohomology theory, 
and $\rH^{a+p\sigma,b+q\sigma}_{C_2}(\mcal{X}, \ul{A})$ for the reduced Bredon cohomology of a pointed motivic $C_2$-space $\mcal{X}$. 
By \aref{thm:realbred}, 
the Betti realization $\RRe_{\C}(\MA)$ represents topological Bredon cohomology with $\ul{A}$-coefficients. 
This yields the comparison functor
$$
\RRe_{\C}:H^{a+p\sigma, b+q\sigma}_{C_2}(X, \ul{A}) \to H^{a+p\sigma}_{C_2}(X(\C), \ul{A}).
$$
Moreover, 
this yields a ring homomorphism since $\RRe_{\C}:\SH_{C_2}(\C)\to \SH_{C_2}$ is a symmetric monoidal functor, 
see \aref{theorem:stableBettirealization}.

\subsubsection{Change of groups}
Consider the adjunction $C_{2}\times - : \Sm/k \rightleftarrows C_2\Sm/k:(-)^{e}$,
where $X^{e}$ is the underlying smooth scheme of a smooth scheme with $C_2$-action. 
This extends to an adjunction on pointed motivic spaces and by \aref{prop:Gstadj} it stabilizes to yield the adjunction 
\begin{equation*}
C_{2\,+}\wedge - :\SH(k)\rightleftarrows \SH_{C_2}(k) : (-)^{e}
\end{equation*}
on stable homotopy categories. 
The functor $(-)^{e}$ should be thought of as forgetting the action. 
Similarly, the adjunction $(-)^{triv}:\Sm/k\rightleftarrows C_2\Sm/k:(-)^{C_2}$ induces an adjunction on stable homotopy categories
 $$
 (-)^{triv}:\SH(k)\rightleftarrows \SH_{C_2}(k):(-)^{C_2}.
 $$ 
Here  $(-)^{triv}:\Sm/k\to C_2\Sm/k$ sends a smooth scheme $X$ to the $C_2$-scheme consisting of $X$ with the trivial action. 
If $\mathsf{E}$ is a motivic spectrum then $C_{2\,+}\wedge \mathsf{E}$ agrees with $C_{2\,+}\wedge (\mathsf{E})^{triv}$ 
(the latter being the smash product of two equivariant motivic spectra). 
Typically we simply write $\mathsf{E}$ again for the spectrum $\mathsf{E}^{triv}$.

Write $H^{\ast, \ast}_{\mcal{M}}(-, A)$ for motivic cohomology theory and ${\bf M}A$ for the representing motivic spectrum. 
There is a natural map 
\begin{equation}\label{eqn:phimap}
\phi:\rH^{a+p\sigma, b+q\sigma}_{C_2}(\mcal{X}, \ul{A}) \to \rH^{a+p, b+q}_{\mcal{M}}(\mcal{X}^{e}, A)
\end{equation}
obtained from 
$(-)^{e}:[\mcal{X}, S^{a+b\sigma, b+q\sigma}\wedge \MA]_{\SH_{C_2}(k)}
\to [\mcal{X}^{e}, (S^{a+b\sigma, b+q\sigma}\wedge \MA)^{e}]_{\SH(k)}$ and the facts that 
$(S^{a+b\sigma, b+q\sigma}\wedge \MA)^{e} = S^{a+b, b+q}\wedge (\MA)^{e}$ and $(\MA)^{e}$ is a $T^2$-spectrum representing motivic cohomology.
When $k=\C$, 
by \aref{prop:stablecommrel} and \aref{thm:realbred} we have a commutative square
\begin{equation}
\label{eqn:esquare}
\xymatrix{
\rH^{a+p\sigma, b+q\sigma}_{C_2}(\mcal{X}, \ul{A}) \ar[r]^{\phi}\ar[d] & \rH^{a+p, b+q}_{\mcal{M}}(\mcal{X}^{e}, A) \ar[d] \\
\rH^{a+p\sigma}_{C_2}(\RRe_{\C}(\mcal{X}), \ul{A}) \ar[r]^{\phi} & \rH^{a+p}_{sing}(\RRe_{\C}(\mcal{X})^{e}, A).
}
\end{equation}

\begin{proposition}\label{prop:C2e}
Let $\mcal{X}$ be a pointed motivic $C_2$-space. 
For integers $a, b, p, q$, there is a natural isomorphism
$$
\rH^{a+p\sigma, b+q\sigma}_{C_2}(C_{2\,+}\wedge \mcal{X}, \ul{A}) \xrightarrow{\iso} \rH^{a+p, b+q}_{\mcal{M}}( \mcal{X}^{e}, A).
$$
Moreover, 
when $k=\C$, 
there is a commutative square
$$
\xymatrix{
\rH^{a+p\sigma, b+q\sigma}_{C_2}(C_{2\,+}\wedge \mcal{X}, \ul{A}) \ar[r]^-{\iso}\ar[d]_-{\RRe_{\C}} & \rH^{a+p, b+q}_{\mcal{M}}( \mcal{X}^{e}, A)\ar[d]^-{\RRe_{\C}}\\
\rH^{a+p\sigma}_{C_2}(C_{2\,+}\wedge \RRe_{\C}(\mcal{X}), \ul{A}) \ar[r]^-{\iso} & \rH^{a+p }_{sing}( \RRe_{\C}(\mcal{X})^{e}, A).
}
$$
\end{proposition}
\begin{proof}
By \aref{lem:trick}, 
the adjoint of the inclusion $\mcal{X}^{e}\to (C_{2\,+}\wedge \mcal{X})^{e} = \mcal{X}^e\coprod \mcal{X}^e$ as the first summand, 
is an isomorphism $i:C_{2\,+}\wedge \mcal{X}^{e}\iso C_{2\,+}\wedge \mcal{X}$ in the equivariant motivic homotopy category. 
Together with the adjunction isomorphism and the observation that $(\MA)^{e}$ represents motivic cohomology, 
we obtain isomorphisms
\begin{align*}
[C_{2\,+}\wedge \mcal{X}, \Sigma^{a+p\sigma, b+q\sigma}\MA]_{\SH_{C_2}(k)} & \iso 
[C_{2\,+}\wedge (\mcal{X})^{e}, \Sigma^{a+p\sigma, b+q\sigma}\MA]_{\SH_{C_2}(k)} \\
& \iso [\mcal{X}^{e}, \Sigma^{a+p, b+q}{\bf M} A]_{\SH(k)},
\end{align*}
yielding the first part of the proposition.
 
We have $\RRe_{\C}(\mcal{X})^{e}\iso\RRe_{\C}(\mcal{X}^{e})$ and $C_{2\,+}\wedge \RRe_{\C}(\mcal{X})\iso\RRe_{\C}(C_{2\,+}\wedge \mcal{X})$ by \aref{prop:commrel}. 
The map $\RRe_{\C}(i)$ is an isomorphism $C_{2\,+}\wedge \RRe_{\C}(\mcal{X})^{e}\iso C_{2\,+}\wedge\RRe_{\C}(\mcal{X})$. 
Together with the isomorphism $\RRe_{\C}(\MA) \cong {\rm H}\ul{A}$ in \aref{thm:realbred} this implies the second part of the proposition.
\end{proof}

\begin{proposition}
\label{prop:forgetiso}
Let $X$ be a smooth $k$-scheme considered as a $C_2$-scheme with trivial action. 
For integers $m,n$,  
the map (\ref{eqn:phimap}) is an isomorphism
$$
\phi:H^{m,n}_{C_2}(X, \ul{A}) \xrightarrow{\iso} H^{m,n}_{\mcal{M}}(X, A).
$$
\end{proposition}
\begin{proof}
The functor $(-)^{e}:C_2\Sm/k\to \Sm/k$ takes equivariant Nisnevich covers to Nisnevich covers and so induces a morphism of sites
$(\Sm/k)_{Nis} \to (C_2\Sm/k)_{C_2Nis}$. This induces a map on cohomology $\phi:H^{m}_{C_2Nis}(X,\ul{A}(n)) \to H^m_{Nis}(X, A(n))$. 
Here $A(n)$ is the Nisnevich sheafification of $C_*A_{tr}(T^{\wedge n})[-2n]$, 
i.e., it is the ``usual'' weight-$n$ motivic complex.
Under the isomorphism of \aref{thm:motivicrep} this map $\phi$ is identified with the map 
$\phi$ in the statement of the proposition.

There is also a morphism of sites 
$t:(C_2\Sm/k)_{C_2Nis}\to (\Sm/k)_{Nis}$, induced by  
$(-)^{triv}:\Sm/k\to C_2\Sm/k$. The functor $t^*$ is exact and  
$t^*A(n) = \ul{A}(n)$,  see 
\cite[Lemma 3.19]{HVO:cancellation}. We thus have an induced isomorphism
$\psi:H^m_{Nis}(X,A(n))\iso H^m_{C_2Nis}(X, \ul{A}(n))$; this is an inverse to $\phi$.
\end{proof}

\begin{proposition}
\label{lem:freequot}
Let $X$ be a smooth quasi-projective $C_2$-scheme over $k$. 
Suppose that $X$ has free action.  
For integers $m,n$,  
there are isomorphisms  
$$
\xymatrix{
H^{m,n}_{\mcal{M}}(X/C_2, A)&
H^{m,n}_{C_2}(X/C_2, \ul{A}) \ar[r]^-{\iso}_-{\pi^*}\ar[l]_-{\iso}^-{\phi}& 
H^{m,n}_{C_{2}}(X,\ul{A}),
}
$$
where $\pi:X\to X/C_2$ is the quotient.  
Moreover, 
when $k=\C$, 
there are commutative squares
$$
\xymatrix{
H^{m,n}_{\mcal{M}}(X/C_2, A) \ar[d]_{\RRe_{\C}}&
H^{m,n}_{C_2}(X/C_2, \ul{A}) \ar[r]^-{\iso}_-{\pi^*}\ar[d]_{\RRe_{\C}}\ar[l]_-{\iso}^-{\phi}& 
H^{m,n}_{C_{2}}(X,\ul{A}) \ar[d]_{\RRe_{\C}}\\
H^{m}_{sing}(X(\C)/C_2, A) &
H^{m}_{C_2}(X(\C)/C_2, \ul{A}) \ar[r]^-{\iso}_-{\pi^*}\ar[l]_-{\iso}^-{\phi}&
H^{m}_{C_{2}}(X(\C),\ul{A}).
}
$$
\end{proposition}
\begin{proof}
The first arrow is an isomorphism by \aref{prop:forgetiso}. 
By \aref{thm:motivicrep}, 
the second map is identified with the map between equivariant Nisnevich hypercohomology groups $\pi^*:H^{m}_{C_2Nis}(X/C_2, \ul{A}(n))\to H^{m}_{C_2Nis}(X, \ul{A}(n))$.
Since $X/C_2$ has trivial action  
we have the isomorphism
$$
H^{m}_{C_2Nis}(X/C_2, \ul{A}(n)) \iso H^{m}_{Nis}(X/C_2, A(n)).
$$ 
Under this identification, 
the isomorphism 
$$
H^{m}_{Nis}(X/C_2, A(n)) \iso H^{m}_{C_2Nis}(X, \ul{A}(n))
$$ 
of \cite[Lemma 3.19]{HVO:cancellation} is $\pi^*:H^{m}_{C_2Nis}(X/C_2, \ul{A}(n))\to H^{m}_{C_2Nis}(X, \ul{A}(n))$. 
 
When $k=\C$, we have that $X(\C)/C_2 = (X/C_2)(\C)$. 
The commutativity of the first square is a specialization of (\ref{eqn:esquare}). The commutativity of the second is immediate.
\end{proof}

We embed $C_2\subseteq \A(\sigma)$ via $C_2=\{\pm 1\}$. 
By the equivariant homotopical purity theorem \cite[Theorem 7.6]{HKO:EMHT}, \cite[Theorem 3.23]{Hoyois:six}, 
there is an equivariant motivic weak equivalence 
$$
C_{2+}\wedge T^{\sigma}\wkeq \P(\sigma\oplus 1)/\P(\sigma\oplus 1)\setminus C_2.
$$ 
Since $\P(\sigma\oplus 1)\setminus\A(\sigma)\subseteq \P(\sigma\oplus 1)\setminus C_2$ and $T^{\sigma} \wkeq \P(\sigma\oplus 1)/\P(\sigma\oplus 1)\setminus\A(\sigma)$ we have a map 
$\tau':T^{\sigma}\to C_{2+}\wedge T^{\sigma}$, 
and hence a stable map $\tau:S^{0}\to C_{2+}$. 
 
\begin{remark}
Presumably topological Bredon cohomology is a presheaf with equivariant transfers (in the sense of \cite[\S4]{HVO:cancellation}), but establishing this would require a lengthy digression. For this reason we use the map $\tau$ in the following proposition rather than transfer maps coming from the theory of  presheaves with equivariant transfers.
\end{remark}

\begin{proposition}\label{prop:mtwo}
Let $\mathsf{E}$ be a motivic $C_2$-spectrum over $\C$. Write $\pi:C_{2+}\wedge \mathsf{E}\to \mathsf{E}$ for the projection. 
Then the diagram below commutes and $\tau^*\pi^*=2$,
$$
\xymatrix{
\rH^{\star,\star}_{C_2}(\mathsf{E}, \ul{A})\ar[r]^-{\pi^*}\ar[d]_{\RRe_{\C}} & 
\rH^{\star,\star}_{C_2}(C_{2+}\wedge \mathsf{E},\ul{A}) \ar[r]^-{\tau^*}\ar[d]_{\RRe_{\C}} &
\rH^{\star,\star}_{C_2}(\mathsf{E},\ul{A})\ar[d]_{\RRe_{\C}}
\\
\rH^{\star}_{C_2}(\RRe_{\C}(\mathsf{E}), \ul{A}) \ar[r]^-{\pi^*}&
\rH^{\star}_{C_2}(C_{2+}\wedge \RRe_{\C}(\mathsf{E}),\ul{A}) \ar[r]^-{\tau^*}&
\rH^{\star}_{C_2}(\RRe_{\C}(\mathsf{E}),\ul{A}).
}
$$
\end{proposition}
\begin{proof}
The commutativity of the diagram is immediate.
It suffices to treat the case $\mathsf{E}=S^{0}$ and to see that $\tau^*(1) = 2$.
The topological realization of $\tau$ is the Spanier-Whitehead dual of the projection 
$C_{2+}\to S^{0}$. In particular $\tau^*(1) = 2$ in $H^{0}_{C_2}({\rm pt.},\ul{A})$.
It remains to show that $\tau^*(1)=2\in H^{0,0}_{C_2}(k,\ul{A}) = A$. 
This follows from the commutative diagram
$$
\xymatrix{
A=\rH^{0,0}_{C_2}(C_{2+},\ul{A}) \ar[r]^-{\tau^*}\ar[d]_{\iso}^{\phi}&
\rH^{0,0}_{C_2}(S^{0},\ul{A})=A\ar[d]^{\iso}_{\phi}
\\
A=\rH^{0,0}_{\mcal{M}}(S^{0}\vee S^{0},\ul{A}) \ar[r]^-{\tau^*}&
\rH^{0,0}_{\mcal{M}}(S^{0},\ul{A})=A
}
$$
and that the bottom arrow sends $1$ to $2$. 
\end{proof}

\subsection{Thom isomorphisms}
Let $\mathsf{R}$ be a $C_2$-equivariant motivic commutative ring spectrum, 
$X$ a smooth $C_2$-scheme over $k$, 
and $E\to X$ a $C_2$-equivariant vector bundle.

\begin{definition}
An \emph{$\mathsf{R}$-Thom class} (or simply \emph{Thom class}, when $\mathsf{R}$ is understood) for $E$ is a class $u\in \mathsf{R}^{\star,\star}_{C_2}(\Th(E))$ 
with the property that for any equivariant map $f:Y\to X$ of smooth $C_{2}$-schemes over $k$, 
the composition
$$
\mathsf{R}^{\star, \star}_{C_{2}}(Y_+) 
\xrightarrow{\id\otimes f^*u} 
\mathsf{R}^{\star, \star}_{C_{2}}(Y_+)\otimes \mathsf{R}^{\star, \star}_{C_{2}}(\Th(f^*E))
\xrightarrow{\Delta^*} 
\mathsf{R}^{\star, \star}_{C_{2}}(\Th(f^*E))
$$ 
is an isomorphism. 
Here, 
$\Delta: \Th(f^*E)\to Y_+\wedge \Th(f^*E)$ is the Thom diagonal.
\end{definition}

\begin{proposition}\label{prop:thomtriv}
Let $V=a +b \sigma$. 
There are classes
$u_{V}\in \rH^{2a+2b\sigma, a+b\sigma}_{C_2}(T^{V},\ul{\Z})$
such that
$$
\rH^{\star,\star}_{C_2}(X_+, \ul{\Z})\xrightarrow{- \cup (1_X\times u_{V})} \rH^{\star+2a+2b\sigma, \star+a+b\sigma}_{C_2}(X_+\wedge T^{V},\ul{\Z})
$$
is an isomorphism, for any $C_2$-variety $X$. Moreover, if $\phi$ is an automorphism of the  $C_2$-equivariant vector bundle $X\times \A(V)\to X$, then 
\begin{equation}\label{eqn:gp}
\phi^*(1_{X}\times u_{V}) = 1_X\times u_{V}.
\end{equation}
 
\end{proposition}
\begin{proof}

Let $u_1 \in \rH^{2,1}_{C_2}(T,\ul{\Z})$ and $u_{\sigma} \in \rH^{2\sigma,\sigma}_{C_2}(T^\sigma,\ul{\Z})$ be elements corresponding to the unit under the suspension isomorphism \eqref{eqn:suspiso}. 
For a representation $V=a + b\sigma$ define $u_{V} = (u_1)^{a}(u_{\sigma})^{b}$. This element satisfies the first condition and it remains to check that (\ref{eqn:gp}) holds for any equivariant bundle automorphism $\phi$ of $X\times\A(V)$. We will proceed by induction on the dimension of $V$.

We first consider the case  $V=\sigma$. 
Write $\alpha = a+ p\sigma$, $\beta = b+q\sigma$ and  write
$$
H^{\alpha,\beta}_{C_2}(X,\ul{\Z})_{(-\sigma)} = 
\coker\left(H^{\alpha,\beta}_{C_2}(X\times \A(\sigma),\ul{\Z})\xrightarrow{i^*} H^{\alpha,\beta}_{C_2}(X\times (\A(\sigma)\setminus\{0\}), \ul{\Z})\right).
$$
Since $H^{\alpha,\beta}_{C_2}(-,\ul{\Z})$ is a presheaf with equivariant transfers (see \aref{subsub:transfers}), if $X$ is affine then by 
\cite[Proposition 8.3]{HVO:cancellation}
the map $i^*$ has a retraction and we have a natural splitting
\begin{equation}
\label{eqn:spl}
H^{\alpha,\beta}_{C_2}(X\times (\A(\sigma)\setminus\{0\}), \ul{\Z}) = 
H^{\alpha,\beta}_{C_2}(X\times \A(\sigma), \ul{\Z})
\oplus H^{\alpha,\beta}_{C_2}(X,\ul{\Z})_{(-\sigma)}.
\end{equation}

Consider the cofiber sequence $(\A(\sigma)\setminus\{0\})_+\to \A(\sigma)_+\to T^{\sigma}$. 
The induced long exact sequences break into (split) short exact sequences
$$
0\to H^{\alpha, \beta}_{C_2}(\A(\sigma), \ul{\Z}) \xrightarrow{i^*} 
H^{\alpha, \beta}_{C_2}(\A(\sigma)\setminus\{0\},\ul{\Z}) \xrightarrow{\delta} 
\rH^{\alpha + 1,\beta}_{C_2}(T^{\sigma},\ul{\Z}) \to 0.
$$
The element $u_{\sigma}$ lifts to an element $u_{\sigma}'\in H^{2\sigma-1,\sigma}_{C_2}(\A(\sigma)\setminus\{0\}, \ul{\Z})$.
We choose $u'_{\sigma}$ so that under the  splitting (\ref{eqn:spl}), we have $u'_{\sigma}\in H^{2\sigma-1,\sigma}_{C_2}(\spec(k),\ul{\Z})_{(-\sigma)}$.
Let $\phi$ be an automorphism of the equivariant vector bundle $X\times \A(\sigma)$ over $X$. By naturality, 
to show that $\phi^*(1_{X}\times u_{\sigma}) = 1_{X}\times u_{\sigma}$ for a smooth affine $C_2$-variety $X$, 
we are reduced to showing
\begin{equation}\label{eqn:enuf}
\phi_{0}^*(1_{X}\times u'_{\sigma}) = (1_{X}\times u'_{\sigma}) + \beta,
\end{equation}
where $\beta$ is some element in $\ker(\delta)$ 
and $\phi_{0}$ is the restriction of $\phi$ to  $X\times \A(\sigma) \setminus \{0\}$. 
For any $X$ (not necessarily affine), 
the group of equivariant linear automorphisms of $X\times \A(\sigma)$ over $X$ is  $(\mcal{O}_{X}^{*})^{C_2}$, 
i.e., 
an equivariant automorphism is given by  multiplication with an invariant unit. 
An invariant unit is specified (uniquely) by an equivariant map $X\to \G_{m}$. 
By naturality, 
to verify the relation (\ref{eqn:gp}) for $1_{X}\times u_{\sigma}$ and any automorphism $\phi$ of $X\times\A(\sigma)$, 
it suffices to verify it for $X=\G_{m}$ and $\phi$ the automorphism of $\G_{m}\times \A(\sigma)\setminus \{0\}\to \G_{m}$ given by multiplication with the canonical unit $t$ of $\G_{m}$. 
In this case,  
$\phi_{0}$ is the map $\langle pr_1, \mu \rangle$, 
where $pr_1$ is the projection to the first factor and $\mu:\G_{m}\times (\A(\sigma)\setminus\{0\})\to \A(\sigma)\setminus \{0\}$ is the multiplication. 
We have that $\phi^*_{0}(1_{\G_{m}}\times u'_{\sigma}) =  1\cup \mu^*(u'_{\sigma}) = \mu^*(u'_{\sigma})$.

From the naturality of the decomposition (\ref{eqn:spl}) we see that $\mu^*$ restricts to a map  
$$
\mu^*: H^{2\sigma-1,\sigma}(k,\ul{\Z})_{(-\sigma)} \to  H^{2\sigma-1,\sigma}(\G_{m},\ul{\Z})_{(-\sigma)}. 
$$
Moreover, 
this map has a splitting induced by $e:\spec(k)\to \G_{m}$, 
the inclusion at $1\in \G_{m}$.

It follows that $\mu^*(1_{\G_{m}}\times u'_{\sigma}) = 1_{\G_{m}}\times u_{\sigma}' + \gamma$, 
where $\gamma \in \ker(e^*)$. 
An easy diagram chase shows that $\delta(\ker(e^*)) =0$, 
and so (\ref{eqn:enuf}) holds. 
It follows that (\ref{eqn:gp}) holds for $1_{X}\times u_{\sigma}$ for any $X$ and any automorphism of $X\times \A(\sigma)$. 
A similar (but easier) argument establishes (\ref{eqn:gp}) for $1_{X}\times u_{1}$ as well. 
This takes care of the case when $V$ has dimension one.

Now we proceed by induction on the dimension of $V$. Let $V = a + b\sigma$ be an $n$-dimensional representation. Let $\ul{\Aut}^{C_2}(V)$ be the presheaf whose value on $X$ is the group of equivariant bundle automorphisms of $X\times \A(V)$. It is represented by a group scheme in $C_2\Sm/k$, in particular it is an equivariant Nisnevich sheaf. Under the isomorphism \eqref{eqn:suspiso}, the assignment $1_X\times u_V\mapsto \phi^*(1_X\times u_V)$ is an automorphism of $H^{0,0}_{C_2}(X,\ul{\Z})=H^0_{C_2Nis}(X,\Z)$. This assignment is natural and defines a morphism of sheaves
$\ul{\Aut}^{C_2}(V)\to\Z^{\times}$. To check that the image of this map is $\{1\}$, we may assume that $X$ is a point of the equivariant Nisnevich topology. Recall that for a finite group $G$, the points of the equivariant Nisnevich topology are of the form $G\times_{H} \spec(R)$ where $H\subseteq G$ is a subgroup and $R$ is an essentially smooth Henselian local $k$-algebra with $H$-action, see \cite[Theorem 3.14]{HVO:cancellation}.   
For $G=C_2$, there are two possibilities,  
$X = C_{2}\times \spec(R)$ or $X=\spec(R)$, 
where $R$ is a smooth local ring with $C_2$-action. 
In the first case the claim follows from \aref{prop:C2e} and that the $u_V$ are nonequivariant Thom classes. 
Now we consider the case that $X=\spec(R)$, 
where $R$ is a smooth local ring with $C_2$-action. 
The equivariant automorphisms of $X\times\A(V)$ are $(\Aut_{R}(V_{R}))^{C_2}$. 
If $A$ is a matrix with entries in $R$, 
write $A^{\sigma}$ for the matrix obtained by applying the involution $\sigma$ to its entries. 
Under the identification ${\rm End}_{R}(V_{R}) = {\rm Mat}_{n\times n}(R)$ the $C_2$-action is given by
$$
\left[
\begin{array}{cc}
A & B \\
C & D
\end{array}
\right]
\mapsto
\left[
\begin{array}{cc}
A^{\sigma} & -(B^{\sigma}) \\
-(C^{\sigma}) & D^{\sigma}
\end{array}
\right].
$$
Here, 
$A$ is an $a\times a$-matrix and $B$ is an $b\times b$-matrix, and $a+b=n$. 
It follows that this matrix is in ${\rm Mat}_{n\times n}(R)^{C_2}$ if and only if $A$ and $D$ have coefficients in $R^{C_2}$ and $\sigma$ acts by $-1$ on the coefficients of $B$, $C$.

Let $x\in R$ and for $i\neq j$ write $E_{ij}(x)$ for the elementary matrix corresponding to adding $x$ times row $j$ to row $i$, 
i.e.,  
it has $x$ in position $(i,j)$ and is the same as the identity matrix in all other entries. 
Then $E_{ij}(x)$ is in ${\rm GL}_{n}(R)^{C_2}$ if
\begin{enumerate}
\item[(i)] either $i,j\leq a$ or $a<i,j$ and $x\in R$ is invariant, or
\item[(ii)] $i\leq a$, $j>a$ or $i> a$, $j\leq a$ and $\sigma x = -x$.
\end{enumerate}
The matrix $E_{ij}(x)$ is equivariantly homotopic to the identity via the equivariant $\A^1$-homotopy $t\mapsto E_{ij}(tx)$.

Let $T_{ij}$ be the elementary matrix corresponding to switching the $i$th and $j$th rows. 
If $i,j\leq a$ then $T_{ij}$ is in ${\rm Mat}_{n\times n}(R)^{C_2}$. 
Moreover, 
in this case there is an algebraic map $\A^1\to {\rm GL}_{a}(R^{C_2})\subseteq  {\rm GL}_{n}(R)^{C_2}$ joining $T_{ij}$  and the diagonal matrix 
$\langle -1,1\ldots, 1\rangle$.  
Thus there is an equivariant $\A^1$-homotopy joining $T_{ij}$ and $\langle -1,1,\ldots, 1\rangle$ in ${\rm GL}_{a+b}(R)$. 
Similarly, 
$T_{ij}$ is equivariantly $\A^1$-homotopic to the identity for $i,j>a$.

Let $M=(m_{ij})$ be an invertible matrix in ${\rm Mat}_{n\times n}(R)^{C_2}$. 
We show that there is an $\A^1$-homotopy 
$$
M\wkeq_{\A^1}
\left[
\begin{array}{cc}
u & 0 \\
0 & M'
\end{array}
\right]
\;\;\;\text{or} \;\;\; \left[
\begin{array}{cc}
M' & 0 \\
0 & u
\end{array}
\right].
$$
Here, 
$u\in R$ is an invariant unit and $M'$ is an invertible $(n-1)\times (n-1)$-matrix in ${\rm Mat}_{(n-1)\times(n-1)}(R)^{C_2}$.
First we assume the entry $m_{11}$ of $M$ is a unit. 
In this case, 
the claim follows by multiplying with the elementary matrices $E_{i1}(-m_{i1}/m_{11})$, $i>1$ on the left and with $E_{1j}(-m_{1j}/m_{11})$, 
$j>1$ on the right.  
If $m_{11}$ is not a unit but some $m_{ij}$ is a unit for $i,j\leq a$, 
then by multiplying by $T_{1i}$ on the left and $T_{j1}$ on the right, 
we are reduced to the previous case. 
If all entries $m_{ij}$ of $A$ (i.e., $i,j\leq a$) are non-units, 
we can repeat the previous arguments for the elements of $D$, 
if at least one of its entries is a unit. 
The only remaining case is when all entries of $A$ and of $D$ are in the maximal ideal of $R$. 
For $t\in \A^1$ we write $M_{t}$ for the matrix which agrees with $M$ in all positions except $(M_{t})_{11} = t+m_{11}$. 
This gives a map $\A^1\to {\rm Mat}_{n\times n}(R)^{C_2}$. 
Write $M_{t}'$ for the $a\times a$-matrix consisting of $(M_{t})_{ij}$ for $i,j\leq a$. 
The reduction of $M_{t}$ modulo the maximal ideal of $R$ is
$$
\overline{M_{t}} = 
\left[
\begin{array}{cc}
 \overline{M'_t} & \overline{B} \\
 \overline{C} & 0
\end{array}
\right].
$$
Since $M=M_{0}$ is invertible so is $\overline{M}_{t}$. 
It follows that $M_{t}$ is invertible for all $t$. 
Now $M_{1}$ is equivariantly $\A^1$-homotopic to $M$ and the previous case applies to $M_1$, so we are done.
\end{proof}

Recall (\ref{eqn:elements}) that we have elements $\epsilon, \epsilon', u \in H^{0,0}_{C_2}(k,\ul{\Z})$ which determine the commutativity properties
of the ring $H^{\star,\star}_{C_2}(X,\ul{\Z})$. These elements are 
$\epsilon = \Sigma^{-2}_{T}\tau^*_{T}(\Sigma^2_T1)$, $\epsilon' = \Sigma^{-2}_{T^{\sigma}}\tau_{T^{\sigma}}^*(\Sigma^2_{T^\sigma}1)$, 
and $u=\Sigma^{-2}_{S^{\sigma}}\tau_{S^{\sigma}}^*(\Sigma^2_{S^{\sigma}}1)$, 
where $\tau_{\mathsf{E}}:\mathsf{E}\wedge \mathsf{E}\to \mathsf{E}\wedge \mathsf{E}$ is the twist endomorphism. 

\begin{proposition}
\label{prop:product}
In $H^{0,0}_{C_2}(k, \ul{\Z})$,  $\epsilon = 1$, $\epsilon' =1$, and $u = -1$. 
In particular, 
if $x\in H^{a+p\sigma, b+q\sigma}_{C_2}(X, \ul{\Z})$ and $y\in H^{c+s\sigma, d+t\sigma}_{C_2}(X, \ul{\Z})$, 
then
$$
x\cup y = (-1)^{ac+ps} (y\cup x).
$$
\end{proposition}
\begin{proof}
We have that 
$\rH^{4,2}_{C_2}(T^{2},\ul{\Z})$ is a free $H^{0,0}_{C_2}(k,\ul{\Z})$-module with basis $u_{2}$. 
The map $\tau_T$ is induced by the twist automorphism of the vector bundle $\A^2\to \spec(k)$ and so by \aref{prop:thomtriv}, $\tau_{T}^*(u_{2}) = u_{2}$. 
In particular, 
$\tau_{T}^{*} = \id$ and it follows that 
$\epsilon = \Sigma^{-2}_{T}\tau_{T}^{\ast}(\Sigma^2_T1)=1$, as claimed. 
The argument for $\epsilon'$ is entirely similar.

Now we compute $u$. Recall from \aref{sub:motiviccomplex} the complex $\Z_{top}(\sigma)$ of presheaves with transfers. Under the identification of \aref{thm:motivicrep} the element $u$ corresponds to the twist isomorphism $\Z_{top}(\sigma)\otimes^{tr,\L}\Z_{top}(\sigma)\to \Z_{top}(\sigma)\otimes^{tr,\L}\Z_{top}(\sigma)$ in 
$D^{-}(C_2\Cor_k)$. 
The twist isomorphism of complexes $C\otimes D \to D\otimes C$  is given  componentwise by $(-1)^{pq}$ times the twist $C^{p}\otimes D^{q} \to D^{q}\otimes C^{p}$.
The complex $\Z_{top}(\sigma)\otimes^{tr,\L}\Z_{top}(\sigma)$ is isomorphic to 
$$
\Z_{tr,C_2}(C_2\times C_2) \xrightarrow{\langle q_2, -q_1 \rangle} \Z_{tr,C_2}(C_2)\oplus \Z_{tr,C_2}(C_2) \xrightarrow{p_*\oplus p_*} \Z,  
$$
where $\Z$ is in degree $0$, $p:C_2\to \spec(k)$ is the projection, and $q_i:C_2\times C_2\to C_2$ is the projection to the $i$th factor.
Under this isomorphism, the twist isomorphism is given respectively in degrees $-2$, $-1$, and $0$, by $-\tau_*$, $\tau$, and $\id$.

A chain homotopy between the twist map and $-\id$ is given by $\{s_{i}\}$, $s_i = 0$, $i\neq 0,-1$, and $s_0 = \langle 0 , p^{t} \rangle$, $s_{-1} = \Delta'_* \oplus \Delta'_*$. 
Here $\Delta':C_2\to C_2\times C_2$ is given by $e\mapsto \sigma\times e$, $\sigma\mapsto e\times \sigma$ and $p^t$ is the transpose of $p$. 
Indeed, it is quickly checked that 
$s_{-1}\langle q_2 , -q_1\rangle = \Delta'_*q_2 - \Delta'_*q_1 =-\tau_*+\id$ and
$(p_*\oplus p_*)s_0 = 2\cd\id$. For the remaining chain homotopy relation, we have that $\langle q_2,-q_1\rangle s_1 + s_0(p_*\oplus p_*) $ is equal to
$$
   \begin{pmatrix}
q_2 \Delta'_* &  q_2 \Delta'_*\\ 
p^tp_*-q_1 \Delta'_* & p^tp_*-q_1 \Delta'_*
\end{pmatrix} 
= \begin{pmatrix}
\id  &  \id \\ 
(\id + m_*) - m_* & (\id +m_*)-m_*
\end{pmatrix} 
= \tau+\id,
$$
where $m:C_2\to C_2$ is the nontrivial involution.
\end{proof}

\begin{definition}
Let $V$ be a $C_2$-representation and $E\to X$ a $C_2$-equivariant vector bundle. 
Say that $E$ is \emph{type $V$} if every point $x\in X$ is contained in an invariant open neighborhood $U\subseteq X$ such that $E|_{U}$ is isomorphic, 
as a $C_2$-equivariant vector bundle, 
to the product bundle $U\times \A(V)\to U$. 
\end{definition}

\begin{theorem}\label{thm:thomisos}
Let $X$ be a smooth $C_2$-scheme over $k$, $V= a + b\sigma$, and $E\to X$ a $C_2$-equivariant vector bundle of type $V$. 
Then there are Thom classes $$
{\rm th}(E)\in \rH_{C_2}^{2a+2b\sigma,a+b\sigma}({\rm Th(E)}, \ul{A}).
$$  
\end{theorem}
\begin{proof}
By assumption, there is a  cover $X = U_{1}\cup \cdots \cup U_{n}$ by open invariant subschemes such that
$E|_{U_{i}} \iso U_{i}\times \A(V)$. 
Proceeding by induction on $n$,  
we can use the Mayer-Vietoris long exact sequence to patch the elements $1_{U_{i}}\times u_{V}$ constructed in the previous proposition. 
The condition of (\ref{eqn:gp}) guarantees that they patch together. 
\end{proof}

In the following, $A$ denotes an abelian group.

\begin{corollary}
\label{cor:gysin}
Let $i:Z\hookrightarrow X$ be a closed immersion of smooth $C_{2}$-schemes over $k$, 
with open complement $j:U\hookrightarrow X$ and normal bundle $\mcal{N}_{i}$.
Suppose that $Z=\coprod Z_{r}$, 
with each $Z_{r}$ invariant, and $\mcal{N}_{i}|_{Z_{r}}$ is of type $a_r + b_r\sigma$. 
Then there is a Gysin long exact sequence 
$$
\cdots \to \oplus_{r} H_{C_2}^{\star-2\alpha_r,\star-\alpha_{r}}(Z_{r},\ul{A})\to H_{C_2}^{\star,\star}(X, \ul{A})\xrightarrow{j^*} H_{C_2}^{\star, \star}(U,\ul{A})  \to \cdots .
$$
\end{corollary}
\begin{proof}
By equivariant homotopical purity \cite[Theorem 7.6]{HKO:EMHT}, 
we have a cofiber sequence of motivic $C_2$-spaces
\begin{equation}
\label{eqn:gysin}
U\to X \to \Th(\mcal{N}_{i}).
\end{equation} 
This induces a long exact sequence
$$
\cdots \to \rH_{C_2}^{\star,\star}(\Th(\mcal{N}_{i}),\ul{A})\to H_{C_2}^{\star,\star}(X, \ul{A})\xrightarrow{j^*} 
H_{C_2}^{\star, \star}(U,\ul{A})  \to 
\rH_{C_2}^{\star +1,\star}(\Th(\mcal{N}_{i}),\ul{A}) \to \cdots .
$$
Note that  $\Th(\mcal{N}_{i}) = \vee_{r} \Th(\mcal{N}_{i}|_{Z_{r}})$. 
Applying the previous theorem to each $\Th(\mcal{N}_{i}|_{Z_{\alpha}})$ identifies the long exact sequence induced by (\ref*{eqn:gysin}) with the desired Gysin sequence.
\end{proof} 

\begin{remark}\label{rem:fixtype}
An important case is the following.
Let $X$ be a smooth $C_2$-scheme and $Z$ a connected component of the fixed point subscheme $X^{C_2}\subseteq X$. 
Then the fibers of the normal bundle of $Z\subseteq X$ are of type $\codim_{X}(Z)\sigma$. 
Indeed, 
 we have $(T_{z}X)^{C_2} = T_{z}(X^{C_2}) = T_{z}(Z)$ for any $z\in Z$, by e.g., \cite[Lemma 8.10]{HVO:cancellation}.
\end{remark}

When $k=\C$, the same construction as in the proof of \aref{thm:thomisos} applies to topological Bredon cohomology. 
Moreover, 
this construction is compatible with the Betti realization functor 
$\RRe_{\C}: \rH^{\star,\star}_{C_{2}}(\Th(E),\ul{A}) \to  \rH^{\star}_{C_{2}}(\Th(E)(\C),\ul{A})$.

\begin{proposition}
\label{prop:thomcompat}
Let $k=\C$ and $E\to X$ be a $C_2$-equivariant vector bundle and suppose ${\rm th}(E)$ is a Thom class.   
Then $\RRe_{\C}({\rm th}(E))\in \rH^{\star}_{C_{2}}(\Th(E)(\C),\ul{A})$ is a Thom class. 
\end{proposition}
\begin{proof}
To show that $\RRe_{\C}({\rm th}(E))$ is a Thom class it suffices to show that $i^*({\rm th}(E))$ is a generator of the free $H^{*}(C_2/H, \ul{A})$-module $\rH^*( \Th(i^*E), \ul{A})$, where $H\subseteq C_2$ is a subgroup and $i:C_2/H\to X(\C)$ is an equivariant map \cite[XVI.9]{May:CBMS}. Write $V = a+ b\sigma$, where $i^*E = C_2/H\times \A(V)$. Then, $i^*({\rm th}(E)) =  a (\Sigma_{T^{V}}1)$ in $\rH^{2a+2b\sigma, a+b\sigma}_{C_2}(Th(i^*E),\ul{A})$, for some $a\in H^{0,0}_{C_2}(\spec(\C), \ul{A})=A$. It follows that $i^*\RRe_{\C}({\rm th}(E)) = a (\Sigma_{S^{V(\C)}}1)$, 
which is a generator of $\rH^{2a+2b\sigma}_{C_2}(\Th(i^*E(\C)),\ul{A})$, 
and so $\RRe_{\C}({\rm th}(E))$ is a Thom class.
\end{proof}

\section{Bredon cohomology and equivariant higher Chow groups}\label{sub:genborel}
In \cite[Theorem 5.19]{HVO:cancellation} we constructed a natural comparison map between the  Bredon motivic cohomology groups and Edidin-Graham's equivariant higher Chow groups. 
In this section we elaborate on the comparison between these two constructions of equivariant motivic cohomology. 
Throughout, $A$ denotes an abelian group.

\begin{proposition}\label{prop:higherchow}
Let $X$ be a smooth quasi-projective $C_2$-scheme over $k$. There is a natural isomorphism
$$
CH_{C_2}^{b}(X, 2b-a, A) \iso 
H^{a,b}_{C_{2}}(X\times\EG C_{2}, \ul{A}).
$$ 
\end{proposition}
\begin{proof}
By definition, 
$CH_{C_2}^{b}(X, 2b-a, A) = CH^{b}(X\times_{C_2} (\A(n\sigma)\setminus \{0\}), 2b-a, A)$ for $n$ sufficiently large, 
see \cite[p.~599, 605]{EG:int}. 
In particular, 
the value of this latter group is constant for $n \gg 0$. Write $U_n = \A(n\sigma)\setminus\{0\}$.
Using the isomorphism between higher Chow groups and motivic cohomology 
\cite[Corollary 2]{Voevodsky:chow}
together with \aref{lem:freequot} we obtain the natural isomorphisms
\begin{align*}
CH_{C_2}^{b}(X, 2b-a, A) & \iso 
{\rm lim}_{n} H^{a,b}_{\mcal{M}}(X\times_{C_2}U_n,A) 
\\
&\iso {\rm lim}_{n} H^{a,b}_{C_2}(X\times U_n,\ul{A}) \\
& \iso
 H^{a,b}_{C_{2}}(X\times\EG C_{2}, \ul{A}).
\end{align*}
For the last isomorphism we have used the Milnor exact sequence
$$
0\to {\rm lim}_{n}^{1} H^{a,b}_{C_2}(X\times U_n) \to  
H^{a,b}_{C_{2}}(X\times\EG C_{2}) \to {\rm lim}_{n} H^{a,b}_{C_2}(X\times U_n) \to 0, 
$$
and the fact that the ${\rm lim}^1$-term vanishes.
\end{proof}

The groups $H^{\star,\star}_{C_2}(X\times \EG C_2, \ul{A})$ define the Borel motivic cohomology of $X$.
In light of the identification above, 
we view $H^{\star,\star}_{C_2}(X\times \EG C_2, \ul{A})$ as a generalized version of equivariant higher Chow groups (in which the grading is by representations instead of just integers).
By the motivic isotropy separation cofiber sequence (\ref{eqn:motisotropy}), 
the projection map  $\EG C_{2\,+}\to S^{0}$ induces the comparison map between the Bredon motivic cohomology and the Borel motivic cohomology theories. Their difference is measured by $\tilde{\EG} C_{2}$.

\begin{lemma}
\label{lem:EGtilde}
Let $X$ be a smooth, 
quasi-projective $C_{2}$-scheme. 
Suppose that either (i) $X$ has free action, 
(ii) $b<0$, 
or (iii) $a \leq 1$ and $A$ is finite, 
then  
$$
\rH^{a+p\sigma, b+q\sigma}_{C_{2}}(X_{+}\wedge
\EGt  C_{2}, \ul{A}) = 0.
$$  
\end{lemma}
\begin{proof} 
By \aref{prop:invariance} we have isomorphisms 
$$
\rH^{a+p\sigma, b+q\sigma}_{C_{2}}(X_{+}\wedge
\EGt  C_{2},\ul{A}) \iso \rH^{a, b}_{C_{2}}(X_{+}\wedge\EGt  C_{2},\ul{A}).
$$ 

First we consider the case when $X$ has free action. 
Then  $X/G$ is smooth and applying \aref{lem:freequot}, 
the map $H^{a,b}_{C_2}(X,\ul{A}) \to H^{a,b}_{C_2}(X\times \EG C_2,\ul{A})$ becomes identified with the natural isomorphism
$$
H^{a, b}_{\mcal{M}}(X/G,A) \xrightarrow{\iso}  
H^{a, b}_{\mcal{M}}(X\times^{G}\EG C_{2},A)
$$ 
for all $a,b$. 
This establishes the vanishing in case (i).

Now consider the case when $X$ has trivial action. 
In this case we have
$$
H^{a,b}_{C_2}(X\times\EG C_2,\ul{A}) = H^{a,b}_{\mcal{M}}(X\times \BG C_2,A). 
$$
The projection map $X\times \BG C_{2}\to X$ affords a section and the long exact sequence associated to the motivic isotropy sequence breaks up into short exact sequences 
$$
0\to H^{a,b}_{C_2}(X,\ul{A}) \to H^{a, b}_{\mcal{M}}(X\times \BG C_2,A)\to
\rH^{a+1, b}_{C_{2}}(X_{+}\wedge\EGt  C_{2},\ul{A})\to 
0.
$$
The two left groups vanish whenever $b<0$ and so does the third, 
which establishes case (ii). 
The space $\BG C_2$ is the complement of the zero section of the line bundle $\mcal{O}(-2)$ on $\P^{\infty}$, 
see e.g., \cite[Lemma 6.4]{Voevodsky:reduced}. 
It thus fits into a cofiber sequence
$$
\BG C_{2} \to \P^{\infty} \to \Th(\mcal{O}(-2))
$$
of motivic spaces. 
Now we suppose that $A$ is finite. 
We have, 
as a consequence of the Bloch-Kato conjectures, 
that $H^{a,b}_{\mcal{M}}(X, A) =0$ for $a<0$.
Combined with the sequence in motivic cohomology resulting from the above cofiber sequence, 
the Thom isomorphism, 
and the projective bundle theorem, 
we find an isomorphism $H^{a, b}_{\mcal{M}}(X,A) \to H^{a, b}_{\mcal{M}}(X\times \BG C_{2},A)$ for $a\leq 0$ and $A$ finite.

Now consider an arbitrary smooth $C_{2}$-scheme $X$. 
The fixed points $X^{C_{2}}\subseteq X$ are smooth and its open complement  $U=X\setminus X^{C_{2}}$ has free action.  
The previous paragraphs applied to $X^{C_2}$ and $U$ together with the Gysin sequence of \aref{cor:gysin} and \aref{rem:fixtype} yields the result. 
\end{proof}

\begin{theorem}\label{thm:BLChow}
Let $X$ be a smooth quasi-projective $C_{2}$-scheme over $k$. 
Then the natural map
$$
H^{a+p\sigma, b+q\sigma}_{C_{2}}(X,\ul{A}) \to H^{a+p\sigma, b+q\sigma}_{C_{2}}(X\times \EG C_{2},\ul{A})
$$
is an isomorphism if either (i) $X$ has free action, 
(ii) $b<0$, 
or (iii) and $A$ is finite $a\leq 0$. 
In case (iii) the map is injective if $a=1$.
\end{theorem}
\begin{proof} 
This follows immediately from the previous lemma and the motivic isotropy separation cofiber sequence (\ref{eqn:motisotropy}).
\end{proof}

\begin{remark}
 Combining the identification of \aref{prop:higherchow} and the periodicity of \aref{cor:period} we find that there is a natural map
 \begin{align*}
H^{a+p\sigma, b+q\sigma}_{C_{2}}(X,\ul{\Z/2})& \rightarrow
H^{a+p, b+q}_{\mcal{M}}(X\times^{C_2} \EG C_{2}, \Z/2) \\
& \iso CH^{b+q}(X, 2(b+q)-a-p, \Z/2)
\end{align*}
which is an isomorphism if (i) $X$ has free action, (ii) $b<0$, or (iii) $a\leq 0$. Moreover in case (iii) the map is an injection for $a =1$.
\end{remark}

\section{Periodicity and Borel motivic cohomology}
\label{section:periodicityandBmc}
In this section we show that the ring $H^{\star,\star}_{C_2}(\EG C_2,\ul{\Z/2})$ is periodic with period $(2\sigma-2,\sigma-1)$.
It follows that the groups $H^{\star,\star}_{C_{2}}(X\times \EG C_2, \ul{\Z/2})$ are also periodic.  
These form a generalized "geometric" Borel motivic cohomology theory. 
The integer graded portion of these groups is isomorphic to equivariant higher Chow groups, 
see \aref{sub:genborel}.
This periodicity of the generalized Borel motivic cohomology will play an important role in our comparison theorem between motivic and topological Bredon cohomology.

We remind the reader that our convention is that $(*,*)$ stands for an integer bigrading and $(\star,\star)$ stands for the bigrading determined by representations, 
see (\ref{eqn:funnyspheres}).
We begin with the cohomology of $C_2$.

\begin{lemma}
\label{lem:ofC2}
Let $A$ be an abelian group. There is an $H^{*,*}_{\mcal{M}}(k,A)$-algebra isomorphism
$$ 
H^{\star,\star}_{C_{2}}(C_2,\ul{A}) \iso H^{*,*}_{\mcal{M}}(k,A)[s^{\pm 1},t^{\pm 1}],
$$ 
where $s\in H^{\sigma -1, 0}_{C_{2}}(C_2,\ul{A})$ and $t\in H^{\sigma -1,\sigma -1}_{C_{2}}(C_2,\ul{A})$.
\end{lemma}
\begin{proof}
Write $(\alpha, \beta) = (a+p\sigma, b+ q\sigma)$ and $|a +p\sigma| = a +p$. 
We have isomorphisms
\begin{align*}
H^{\star,\star}_{C_2}(C_{2}, \ul{A}) \xrightarrow{\iso} 
\rH^{\star +\alpha,\star+\beta}_{C_2}(S^{\alpha, \beta}\wedge C_{2\,+}, \ul{A})
& \xrightarrow{\iso} \rH^{\star+\alpha,\star+\beta}_{C_2}(S^{|\alpha|, |\beta|}\wedge C_{2\,+}, \ul{A}) \\
& \xleftarrow{\iso} H^{\star +\alpha - |\alpha|, \star+ \beta - |\beta|}_{C_2}(C_2, \ul{A});  
\end{align*}
the first and last isomorphisms are instances of the suspension isomorphism, and the middle isomorphism follows from \aref{lem:trick}. 
This is an isomorphism of free one-dimensional $H^{\star,\star}_{C_2}(C_{2},\ul{A})$-modules and thus is given by multiplication with an invertible element 
$x_{\alpha, \beta} \in H^{\alpha - |\alpha|, \beta - |\beta|}_{C_2}(C_{2}, \ul{A})$. 
Taking $s= x_{\sigma-1,0}$ and $t = x_{\sigma -1, \sigma-1}$, 
we get an $H^{*,*}_{\mcal{M}}(k,A)$-algebra map
$$ 
H^{*,*}_{\mcal{M}}(k,A)[s^{\pm 1},t^{\pm 1}]\to H^{\star,\star}_{C_{2}}(C_2,\ul{A}),
$$ 
which is an isomorphism.
\end{proof}

\begin{lemma}\label{lem:sigma1}
The map $C_{2+}\to S^{0}$ induces an isomorphism  
\[
H^{2\sigma -2,\sigma-1}_{C_2}(k,\ul{\Z/2})\xrightarrow{\iso} H^{2\sigma-2, \sigma-1}_{C_2}(C_2,\ul{\Z/2}).
\]
\end{lemma}
\begin{proof}
By \aref{thm:motivicrep}, this map is identified with the map
\[
H^{0}_{C_2Nis}(T,\ul{\Z/2}(\sigma)[2\sigma])\to 
H^{0}_{C_2Nis}(C_{2+}\wedge T, \ul{\Z/2}(\sigma)[2\sigma]).
\]
We have that $\ul{\Z/2}(\sigma)[2\sigma] = C_{*}\Z_{tr,C_2}(T^\sigma)\otimes^{\L}\Z/2$ and by \cite[Proposition 5.14]{HVO:cancellation}, 
\[
C_{*}\Z_{tr,C_2}(T^\sigma) \wkeq \cone((\mathcal{O}^*)^{C_2}\oplus \Z \to \Z). 
\]
We now see that 
\[
H^i_{C_2Nis}(T,C_{*}\Z_{tr,C_2}(T^\sigma)) \xrightarrow{\iso} H^{i}_{C_2Nis}(C_{2+}\wedge T, C_{*}\Z_{tr,C_2}(T^\sigma)) 
\]
is an isomorphism for all $i$ (for example, by applying \cite[Lemma 3.19]{HVO:cancellation}).

\end{proof}

\begin{lemma}
\label{lem:EG00}
Let $A$ be an abelian group. There are isomorphisms 
$$
H^{0,0}_{C_2}(\EG C_2,\ul{A}) \cong \lim_nH^{0,0}_{C_2}(\A(n\sigma)\setminus \{0\},\ul{A}) \cong A.
$$
\end{lemma}
\begin{proof}
We have $H^{0,0}_{C_{2}}(\A(n\sigma)\setminus \{0\},\ul{A}) \cong H^{0,0}_{\mcal{M}}((\A(n\sigma)\setminus \{0\})/C_{2},A)= A$ by \aref{lem:freequot} and the maps 
$$
H^{0,0}_{\mcal{M}}((\A(n\sigma)\setminus \{0\})/C_{2},A)\to H^{0,0}_{\mcal{M}}((\A((n+1)\sigma)\setminus \{0\})/C_{2},A)
$$ 
are isomorphisms.
\end{proof}

Write $v=st\in H^{2\sigma-2,\sigma -1}_{C_2}(k,\ul{\Z/2})$ 
for the element obtained from \aref{lem:ofC2} and \aref{lem:sigma1}.
We also write $v$ for the corresponding image in $H^{\star, \star}_{C_{2}}(\EG C_{2},\ul{\Z/2})$.

\begin{theorem}\label{thm:periodicity}
The element $v\in H^{2(\sigma-1),\sigma -1}_{C_2}(\EG C_2,\ul{\Z/2})$ is invertible.
\end{theorem}
\begin{proof}
Consider the equivariant embedding $i_{n}:C_{2}\subseteq \A(n\sigma)\setminus  \{0\} $ given by including at $\{\pm 1\}$.
We show that $i_{n}$ induces an isomorphism
\[
i_n^*:H^{2-2\sigma, 1-\sigma}_{C_2}(\A(n\sigma)\setminus  \{0\} , \ul{\Z/2}) \xrightarrow{\iso} H^{2-2\sigma, 1-\sigma}_{C_{2}}(C_{2}, \ul{\Z/2}).
\]

This will imply the theorem as follows.
For each $n$, 
the elements $v^{-1}$ in $H^{\star,\star}_{C_2}(C_{2}, \ul{\Z/2})$ lift uniquely to  elements $u_{n}$  in $H^{\star,\star}_{C_2}(\A(n\sigma)\setminus \{0\},\ul{\Z/2})$. 
The uniqueness of these lifts implies that $\{u_{n}\}$   determine elements $(u_{n})$, in $\lim_{n} H^{\star,\star}_{C_{2}}(\A(n\sigma)\setminus \{0\},\ul{\Z/2})$.
These in turn lift to an element $\overline{u}$,  in $H^{\star,\star}_{C_{2}}(\EG C_{2}, \ul{\Z/2})$.
We now find that $v\cup \overline{u} \in H^{0,0}_{C_2}(\EG C_{2}, \ul{\Z/2}) =\Z/2$ must be equal to $1$ since it maps to $1 \in H^{0,0}_{C_{2}}(C_{2},\ul{\Z/2})=\Z/2$.

For typographical simplicity we will suppress the coefficients $\Z/2$ of the cohomology groups  from the notation. 
We will also write $U_{n}:=\A(n\sigma)\setminus \{0\}$. 
Consider the comparison of exact sequences, induced by (\ref{eqn:fun2})
\begin{equation*}
\xymatrix{
 H^{1,1}_{C_2}(U_{n})\ar[r]\ar[d] &  H^{1,1}_{C_2}(U_1\times U_n) \ar[r]\ar[d]
&
\rH^{2,1}_{C_2}( T^\sigma\wedge U_{n+})\ar[r]\ar[d] 
& 0 \\
H^{1,1}_{C_2}(C_{2})\ar[r] &  H^{1,1}_{C_2}( U_{1}\times C_{2}) 
\ar[r] &
\rH^{2,1}_{C_2}(T^{\sigma}\wedge C_{2+}) \ar[r] 
& 0 .
}
\end{equation*}

The quotient $U_1\times_{C_2}U_n$ is the complement of the zero section of the line bundle $L:=\A(\sigma)\times_{C_2}U_{n}$ on $U_n/C_2$. 
By \aref{lem:freequot}, the left hand square of the above diagram is identified with the left hand square of the commutative diagram
\begin{equation*}
\xymatrix{
	H^{1,1}_{\mcal{M}}(U_{n}/C_2)\ar[r]\ar[d] &  H^{1,1}_{\mcal{M}}(U_1\times_{C_2} U_n) \ar[r]\ar[d]
	&
	\rH^{2,1}_{\mcal{M}}( \Th(L))\ar[r]\ar[d] 
	& 0 \\
	H^{1,1}_{\mcal{M}}(\spec(k))\ar[r] &  H^{1,1}_{\mcal{M}}(\A^{1}\setminus\{0\}) 
	\ar[r] &
	\rH^{2,1}_{\mcal{M}}(T) \ar[r] 
	& 0.
}
\end{equation*}
The map $\rH^{2,1}_{\mcal{M}}( \Th(L))\to\rH^{2,1}_{\mcal{M}}(T)$ sends the Thom class of $L$ to a generator and so this map is an isomorphism. In particular 
$\rH^{2,1}_{C_2}( T^\sigma\wedge U_{n+})\to \rH^{2,1}_{C_2}( T^\sigma\wedge C_{2+})$ is an isomorphism, as desired.

\end{proof}

\begin{corollary}
\label{cor:period}
Multiplication by $v^{-q}\in H^{2q(1-\sigma), q(1-\sigma)}_{C_2}(\EG C_2,\ul{\Z/2})$ induces a natural isomorphism
$$
\rH^{a+p\sigma,b+q\sigma}_{C_{2}}(\mathsf{E}\wedge (\EG C_{2})_{+},\ul{\Z/2}) \iso 
\rH^{(a+2q) +(p-2q)\sigma,b+q}_{C_{2}}(\mathsf{E}\wedge (\EG C_{2})_+,\ul{\Z/2})
$$
of $H^{\star,\star}_{C_2}(\EG C_2,\ul{\Z/2})$-modules for any motivic $C_2$-spectrum $\mathsf{E}$.
\end{corollary}

\begin{remark} \label{rem}
This  $(2\sigma -2, \sigma -1)$-periodicity does not decompose
into a  $(\sigma -1, 0)$ and a $(\sigma -1, \sigma -1)$-periodicity (contrary to an erroneous claim in a previous version of this paper). 
We also note that the failure of $(\sigma -1, \sigma -1)$-periodicity shows that the condition $a\leq b-q$ in \aref{prop:free} is unavoidable.

To see the failure of $(\sigma -1, 0)$-periodicity, consider the exact sequence,  where $U_n = \A(n\sigma)\setminus 0$,
$$
H^{0,0}_{C_2}(U_n,\ul{\Z/2}) \to H^{0,0}_{C_2}(C_2\times U_n,\ul{\Z/2}) \to \rH^{1,0}_{C_2}(S^{\sigma}\wedge U_{n+},\ul{\Z/2}) \to H^{1,0}_{C_2}(U_n,\ul{\Z/2}).
$$ 
Since $H^{1,0}_{C_2}(U_n,\ul{\Z/2}) \cong H^{1,0}_{\mathcal{M}}(U_n/C_2,\ul{\Z/2}) = 0$ and the
lefthand map is an isomorphism, we have 
$H^{1-\sigma,0}_{C_2}(U_n,\ul{\Z/2}) = 0$. 
Therefore 
$H^{1-\sigma,0}_{C_2}(\EG C_2,\ul{\Z/2}) = 0$;
in particular,  $H^{\star,\star}_{C_2}(\EG C_2,\ul{\Z/2})$ is not  $(\sigma -1,0)$-periodic. Consequently it is also not $(\sigma -1, \sigma -1)$-periodic.

\end{remark}

We note as well that \aref{prop:invariance} immediately implies that the cohomology of any $X_{+}\wedge \EGt C_{2}$ is periodic in the following sense.

\begin{proposition}\label{cor:degshift}
Let $X$ be a smooth $C_2$-scheme over $k$ and $A$ be a commutative ring.
For all integers $a$,$b$,$p$,$q$, 
there are $H^{\star,\star}_{C_{2}}(X,\ul{A})$-module isomorphisms
$$
\rH^{a+p\sigma, b+q\sigma}_{C_{2}}(X_{+}\wedge
\EGt  C_{2},\ul{A}) \iso \rH^{a+(p-1)\sigma, b+q\sigma}_{C_{2}}(X_{+}\wedge\EGt  C_{2},\ul{A})
$$ 
and 
$$
\rH^{a+p\sigma, b+q\sigma}_{C_{2}}(X_{+}\wedge
\EGt  C_{2},\ul{A}) \iso \rH^{a+(p-2)\sigma, b+(q-1)\sigma}_{C_{2}}(X_{+}\wedge\EGt  C_{2},\ul{A} ).
$$ 
\end{proposition}

\section{Comparing motivic and topological Bredon cohomology over \texorpdfstring{$\C$}{C}}\label{s:C}
Let $X$ be a smooth variety over a field. 
The Beilinson-Lichtenbaum conjecture \cite[Conjecture 6.8]{SV:BK} is the assertion that the map
\begin{equation}\label{eqn:BLconj}
H^{p,q}_{\mcal{M}}(X,\Z/n)\to H^{p}_{et}(X,\mu_{n}^{\otimes q})
\end{equation}
is an isomorphism when $p\leq q$ and is an injection when $p = q+1$. 
The validity of this conjecture is a consequence of the Milnor and Bloch-Kato conjectures \cite{Voev:miln, Voev:BK} together with \cite[Theorem 7.4]{SV:BK}.
Now if $X$ is a complex variety this can be rephrased using singular cohomology;
topological realization
\begin{equation}
\label{eqn:BLconjC}
H^{p,q}_{\mcal{M}}(X,\Z/n)\to H^{p}_{sing}(X(\C),\Z/n)\;\;
\text{is } \begin{cases}
\text{an isomorphism} & \text{if } p\leq q, \\
\text{a monomorphism} & \text{if } p = q+1.                                                                                                                                         
\end{cases}
\end{equation}

In this section, 
we establish a $C_2$-equivariant generalization of the Beilinson-Lichtenbaum conjecture for smooth complex varieties $X$ with involution. 
We begin with a consideration of the Borel part of the Bredon cohomologies.

\begin{lemma}\label{lem:freeinp}
	Let $U$ be a smooth quasi-projective complex $C_{2}$-variety with free action. Then
	 $$
	 \RRe_{\C}: H^{m+i\sigma, n}_{C_{2}}(U, \ul{\Z/2}) \to H^{m+i\sigma}_{C_{2}}(U(\C), \ul{\Z/2}),
	 $$
	 is an isomorphism if $m\leq n$ and $m+i\leq n$. It is 
a monomorphism if  $m \leq n+1$ and  $m+i\leq n+1$. 
\end{lemma}
\begin{proof}
When $i=0$ the result holds by \eqref{eqn:BLconjC} together with \aref{lem:freequot}. The result follows in general using induction on $i$ and considering the comparison of exact sequences obtained from \eqref{eqn:fun1},
\[
\xymatrix@C-1.3pc{
 \ar[r] & H^{m+j\sigma,n}_{C_2}(U) \ar[r]\ar[d] & H^{m+(j+1)\sigma,n}_{C_2}(U) \ar[r]\ar[d] & H^{m+j+1,n}_{\mcal{M}}(U) \ar[r]\ar[d] & H^{m+1+j\sigma,n}_{C_2}(U)\ar[r]\ar[d] & \\
\ar[r] & H^{m+j\sigma}_{C_2}(U(\C)) \ar[r] & H^{m+(j+1)\sigma}_{C_2}(U(\C)) \ar[r] & H^{m+j+1}_{sing}(U(\C))\ar[r]& 
H^{m+1+j\sigma}_{C_2}(U(\C))\ar[r] & . 
}
\]
 
\end{proof}

\begin{proposition}\label{prop:free}
	Let $X$ be a smooth complex $C_{2}$-variety and $A$ a finite abelian group. Then
	 $$
	 \RRe_{\C}: H^{a+p\sigma, b+q\sigma}_{C_{2}}(X\times \EG C_{2}, \ul{A}) \to H^{a+p\sigma}_{C_{2}}(X(\C)\times \EG C_{2}(\C), \ul{A}),
	 $$
	 is 
\begin{enumerate}
\item[(i)] an isomorphism if $a\leq b-q$ and $a+p\leq b+q$, and 
\item[(ii)] a monomorphism if $a \leq b-q+1$ and $a+p\leq b+q+1$. 
\end{enumerate}
\end{proposition}
\begin{proof}
	We first assume that $X$ is quasi-projective.
It suffices to assume that $A=\Z/\ell^i$ where $\ell$ is a prime.
Suppose that $2$ is invertible in $A$. Consider the commutative diagram coming from  and \aref{prop:C2e} and \aref{prop:mtwo}, where the coefficient group $A$ has been suppressed
$$
\xymatrix{
	H^{a+p\sigma,b+q\sigma}_{C_2}(X\times \EG C_2)
	\ar@{^{(}->}[r]^-{\pi^*}\ar[d]_{\RRe_{\C}} & 
	H^{a+p,b+q}_{\mcal{M}}(X\times \EG C_2) \ar@{->>}[r]^-{\tau^*}\ar[d]_{\RRe_{\C}} &
	H^{a+p\sigma,b+q\sigma}_{C_2}(X\times \EG C_2)\ar[d]_{\RRe_{\C}}
	\\
	H^{a+p\sigma}_{C_2}((X\times \EG C_2)(\C)) \ar@{^{(}->}[r]^-{\pi^*}&
	H^{a+p}_{sing}((X\times \EG C_2)(\C)) \ar@{->>}[r]^-{\tau^*}&
	H^{a+p\sigma}_{C_2}((X\times \EG C_2)(\C)).
}
$$
The horizontal compositions are multiplication by $2$ and so are isomorphisms. 
By the Beilinson-Lichtenbaum conjecture (\ref{eqn:BLconjC}), 
the middle map is an isomorphism for $a+p\leq b+q$ and an injection for $a+p=b+q+1$.  
This implies that the same is thus true of the outer vertical arrows. 

It remains to consider the case $A = \Z/2^{i}$. 
By comparing the exact sequences arising from the short exact sequence $0\to \Z/2^{i-1}\to\Z/2^{i}\to \Z/2\to 0$ and induction, 
we are reduced to the case $A=\Z/2$. 

We now assume that $A=\Z/2$ (and continue suppressing the coefficients as needed).
Using the periodicity from \aref{cor:period} we may replace 
$(a+p\sigma, b+q\sigma)$ by $((a+2q) + (p-2q)\sigma, b+q)$ 
and we write $(m+i\sigma,b)$ for this new bidegree. The hypothesis on $(a+p\sigma, b+q\sigma)$ implies that  $(m+i\sigma,b)$ satisfies the hypothesis of the previous lemma. 
The result now follows in this case from the previous lemma together with the comparison of Milnor exact sequences
\begin{equation*}
\xymatrix@-1.3pc{
	0 \ar[r] & 
	\lim_{n}^{1}H^{m-1+i\sigma, b}_{C_{2}}(U_n)
	\ar[r]\ar[d] &
	H^{m+i\sigma, b}_{C_{2}}(X\times\EG C_{2})
	\ar[r]\ar[d] &
	\lim_{n}H^{m+i\sigma, b}_{C_{2}}(U_n)
	\ar[r]\ar[d] & 0
	\\
	0 \ar[r] & 
	\lim_{n}^{1}H^{m-1+i\sigma}_{C_{2}}(U_n(\C))
	\ar[r] &
	H^{m+i\sigma}_{C_{2}}((X\times\EG C_{2})(
	\C))
	\ar[r] &
	\lim_{n}H^{m+i\sigma}_{C_{2}}(U_n(\C))
	\ar[r] & 0
}
\end{equation*}
where $U_{n}:=X\times(\A(n\sigma)\setminus  \{0\})$, since $U_{n}$ is a smooth quasi-projective variety with free action.

To deduce the proposition for a general smooth $X$ from the quasi-projective case,  
we use that $X$ is locally affine in the equivariant Nisnevich topology, 
see \aref{rem:localaffine}. 
There are several ways to turn this observation into a formal argument; 
we proceed directly as follows. 
Suppose that we have a cartesian square of smooth $C_2$-complex varieties
$$
\xymatrix{
	W \ar@{^{(}->}[r]\ar[d] & Y \ar[d]^{\phi}\\
	U \ar@{^{(}->}[r] & X,
}
$$
where $Y$ is quasi-projective, 
$\phi$ is equivariant \'etale, 
$U$ is an invariant open, 
and the restriction $\phi|_{Y\setminus W}$ has an equivariant section. 
Then, 
if the proposition is true for $U$ it is also true for $X$. 
Indeed, 
this square leads to a distinguished equivariant Nisnevich square, via standard techniques,
$$
\xymatrix{
	W' \ar@{^{(}->}[r]\ar[d] & Y' \ar[d]\\
	U \ar@{^{(}->}[r] & X,
}
$$
where $Y'$ is open in $Y$;
in particular, it is quasi-projective. 
Comparing the resulting Mayer-Vietoris long exact sequences then shows that under these assumptions, 
the proposition holds for $X$. 
	
Now we let $A\subseteq X$ be a dense invariant affine open and $Y$ any quasi-projective equivariant Nisnevich cover of $X$. 
Let $\emptyset = Z_{n+1}\subseteq Z_{n}\subseteq \cdots Z_{1}\subseteq Z_{0}:=X\setminus A$ be an equivariant splitting sequence for $Y|_{Z_{0}}$. 
Set $X_{i} = X\setminus Z_{i}$ and $Y_{i} = Y|_{X_{i}}$. 
The cartesian square
$$
\xymatrix{
	Y_{i} \ar@{^{(}->}[r]\ar[d] & Y_{i+1} \ar[d]\\
	X_{i} \ar@{^{(}->}[r] & X_{i+1}
}
$$
satisfies the conditions of the previous paragraph, and so proceeding by induction we find that the proposition holds for each $X_{i}$.
\end{proof}

Next we consider the isotropic part of the Bredon cohomologies.

\begin{proposition}\label{prop:isotrop}
Let $X$ be a smooth complex $C_{2}$-variety and $A$ a finite abelian group. 
Then
$$
\RRe_{\C}: \rH^{a+p\sigma, b+q\sigma}_{C_{2}}(X_+\wedge \wt\EG C_{2}, \ul{A}) \to \rH^{a+p\sigma}_{C_{2}}(X(\C)_+\wedge \wt\EG C_{2}(\C), \ul{A}),
$$
is an isomorphism if $a\leq b$ and an injection if $a = b+1$.
\end{proposition}
\begin{proof}
First we consider the special case when $X$ has trivial action. 
By the periodicities supplied by \aref{cor:degshift} and the corresponding ones in topological Bredon cohomology, 
we may assume that $p=q=0$. 
	
Since $X$ has trivial action, 
the map $H^{a,b}_{C_{2}}(X, \ul{A})\to H^{a}_{C_2}(X(\C), \ul{A})$ is naturally isomorphic to  the map $H^{a,b}_{\mcal{M}}(X,A)\to H^{a}_{sing}(X(\C),A)$,  by \aref{prop:forgetiso}. 
In particular,  
by the Beilinson-Lichtenbaum conjecture, 
it is an isomorphism if $a\leq b$ and an injection if $a=b+1$. 
The map $H^{a, b}_{C_{2}}(X\times \EG C_{2}, \ul{A}) \to H^{a}_{C_{2}}(X(\C)\times \EG C_{2}(\C), \ul{A})$ is an isomorphism for $a\leq b$ and an injection for $a=b+1$, 
by the previous proposition. 
The proposition thus follows for $X$ with trivial action by comparing the long exact sequences associated to the  motivic isotropy cofiber sequence
\begin{equation}\label{eqn:isotropya}
X_+\wedge\EG C_{2\,+} \to X_+ \to X_+\wedge \wt\EG C_2.
\end{equation}
	  
We now treat the general case. 
The fixed point scheme $X^{C_{2}}$ is smooth and so by the equivariant homotopical purity theorem \cite[Theorem 7.6]{HKO:EMHT}, 
we have a cofiber sequence 
$$
X \setminus  X^{C_{2}}\to X \to \Th(\mcal{N}),
$$
where $\mcal{N}$ is the normal bundle of the inclusion $X^{C_{2}}\subseteq X$. 
By \aref{lem:EGtilde}, 
the map $X\to \Th(\mcal{N})$ induces an isomorphism  
\begin{equation}\label{eqn:isoN}
\rH^{\star,\star}_{C_{2}}(\Th(\mcal{N})\wedge \EGt C_{2},\ul{A})\xrightarrow{\iso} \rH^{\star,\star}_{C_{2}}(X_{+}\wedge \EGt C_{2},\ul{A})
\end{equation}
and similarly for the topological Bredon cohomology.  
Note that  $X^{C_{2}}$ is a disjoint union $X^{C_{2}} = \coprod_{r}Z_{r}$ of connected smooth varieties with trivial action. 
We may apply \aref{thm:thomisos} and \aref{prop:thomcompat}, together with \aref{rem:fixtype} to each $\mcal{N}|_{Z_{r}}$ in order to obtain the commutative square
$$
\xymatrix@-1pc{
	\rH^{a+p\sigma, b+q\sigma}_{C_{2}}(\Th(\mcal{N})\wedge \EGt C_{2}) \ar[r]^-{\iso}\ar[d] &
	\bigoplus^{r}\rH^{a+(p-2p_{r})\sigma,\, b + (q-p_{r})\sigma}_{C_2}(Z_{r\,+}\wedge \EGt C_{2}) \ar[d] 
	\\
	\rH^{a+p\sigma}_{C_{2}}(\Th(\mcal{N}(\C))\wedge \EGt C_{2}(\C)) \ar[r]^-{\iso} &
	\bigoplus^{r}\rH^{a+(p-2p_{r})\sigma}_{C_2}(Z_{r}(\C)_+\wedge \EGt C_{2}(\C)),
} 
$$
where $p_r = \codim_{X}(Z_r)$. 
The proposition holds for the right hand map and thus it holds for the left hand map as well.    
Applying the isomorphism (\ref{eqn:isoN}) yields the conclusion of the proposition.
\end{proof}

Combining the previous two results now implies our equivariant generalization of the Beilinson-Lichtenbaum conjectures over $\C$.

\begin{theorem}
\label{theorem:equivariantBLconjecture} 
Let $X$ be a smooth complex $C_{2}$-variety and $A$ a finite abelian group. 
The comparison map
$$
\RRe_{\C}:H^{a+p\sigma, b+q\sigma}_{C_{2}}(X,\ul{A})\to 
H^{a+p\sigma}_{C_{2}}(X(\C),\ul{A})
$$
is  
\begin{enumerate}
	\item[(i)] an isomorphism if both $a+p\leq b+q$ and $a\leq \min\{b-q, b\}$,
	\item[(ii)] an injection if both $a+p\leq b+q+1$ and $a\leq\min\{b-q,b\} + 1$.
\end{enumerate}
\end{theorem}
\begin{proof}
This follows by comparing the long exact sequences induced by the motivic isotropy cofiber sequence $X_+\wedge\EG C_{2\,+} \to X_+ \to X_+\wedge \EGt C_2$ together with \aref{prop:free} 
and \aref{prop:isotrop}.
\end{proof}
Notice that in the case $p=q=0$ and $X$ complex variety with trivial $\Z/2-$action the above theorem reduces to the usual Beilinson-Lichtenbaum conjecture for complex varieties (which we actually used in the proof).
(See also \eqref{eqn:BLconjC}.)

\section{Comparing motivic and topological Bredon cohomology over \texorpdfstring{$\R$}{R}}\label{s:R}

Let $X$ be a smooth real variety and write $\Sigma_2 = {\rm Gal}(\C/\R)$. 
The space $X(\C)$ has an $\Sigma_2$-action. 
This extends to the topological realization functor $\RRe_{\C,\Sigma_2}:\SH(\R)\to \SH_{\Sigma_2}$, 
see \cite[Proposition 4.8]{HO:galois}. 
We have $\RRe_{\C,\Sigma_2}(\M A) = {\rm H}\ul{A}$ \cite[Theorem 4.17]{HO:galois} and thus a comparison map relating motivic cohomology and Bredon cohomology. 
By \cite[Corollary 5.11]{HV}, the Beilinson-Lichtenbaum conjecture \eqref{eqn:BLconj} for real varieties can be reinterpreted as the statement that the map
\begin{equation}\label{eqn:blreal}
H^{a,b}_{\mcal{M}}(X,A) \to H^{a-b+b\sigma}_{\Sigma_2}(X(\C), \ul{A}) \;\;\text{is } 
\begin{cases}
\text{an isomorphism} & \text{if } a\leq b, \\
\text{a monomorphism} & \text{if } a = b+1.                                                                                                                                          
\end{cases}
\end{equation}

We now consider a smooth real variety $X$ with a $C_2$-action. 
The space of complex points $X(\C)$ has two involutions, 
one coming from complex conjugation and the other coming from action on the variety $X$. 
To avoid confusing these actions, 
we write $C_2$ for the group acting algebraically on the real variety $X$ while we write $\Sigma_2 = {\rm Gal}(\C/\R)$ for the Galois group and its action is by complex conjugation. 
Thus $X(\C)$ has an $C_2\times \Sigma_2$-action.  
This extends to the topological realization functor, 
see \aref{thm:stableR},
$$
\RRe_{\C,\,\Sigma_2}:\SH_{C_2}(\R) \to \SH_{C_2\times\Sigma_2}.
$$

Write $\tau_1$ (resp.~$\tau_2$) for the nontrivial element of $C_2$ (resp.~of $\Sigma$). 
Write $\sigma$ for the $C_2\times \Sigma_2$-representation which is defined by letting $\tau_1$ act by $-1$ and $\tau_2$ by the identity. 
Write $\epsilon$ for the representation which is defined by letting $\tau_1$ act by the identity and $\tau_2$ by $-1$. 
The four representations of the Klein group $C_2\times \Sigma _2$ are $1$, $\sigma$, $\epsilon$, and $\sigma\otimes\epsilon$. 

The effect of the topological realization functor on spheres is as follows.
\begin{lemma}
We have 
$\RRe_{\C,\,\Sigma_2}(S^1)\wkeq S^1$,   
$\RRe_{\C,\,\Sigma_2}(S^\sigma)\wkeq S^\sigma$, 
$\RRe_{\C,\,\Sigma_2}(S^1 _t)\wkeq S^\epsilon$, 
and 
$\RRe_{\C,\,\Sigma_2}(S^\sigma _t) \wkeq S^{\sigma\otimes\epsilon}$. 
Thus
\begin{align*}
\RRe_{\C,\,\Sigma_2}(S^{a+p\sigma, b+q\sigma}) \wkeq S^{(a-b) + (p-q)\sigma + b\epsilon +q \sigma\otimes\epsilon}.
\end{align*}
\end{lemma}
\begin{proof}
The first relation is obvious. 
The second follows from the cofiber sequence $C_{2\,+}\to S^0 \to S^{\sigma}$. 
We have that $S^1_t(\C)$ and $S^\sigma_t(\C)$ are equivariantly homotopic to the unit circle in $\C^*$. 
In the first case, 
$\tau_1(z)= z$ and $\tau_2(z) = \overline{z}$, 
and in the second case $\tau_1(z) = 1/z$ and $\tau_2(z) = \overline{z}$. 
The displayed equalities follow immediately from these formulae.
\end{proof}

Since  $\RRe_{\C,\Sigma_2}(\MA) \iso {\rm H}\ul{A}$ in $\SH_{C_2\times \Sigma_2}$, see \aref{thm:realbred}, there exists a comparison map
\begin{equation}\label{eqn:cyclemap}
\RRe_{\C,\Sigma_2}: H^{a+p\sigma,b+q\sigma}_{C_2}(X, \ul{A}) \to H^{a-b+(p-q)\sigma + b\epsilon +q\sigma\otimes\epsilon}_{C_2\times\Sigma_2}(X(\C), \ul{A}).
\end{equation}

Observe that $\RRe_{\C,\Sigma_2}(\A(\sigma)) = \sigma +\sigma\otimes \epsilon$.
Applying $\RRe_{\C,\Sigma_2}$ to the motivic isotropy cofiber sequence \eqref{eqn:motisotropy}
yields the cofiber sequence of $C_2\times \Sigma_2$-spaces
$$
{\rm E}_{\Sigma_2}C_{2 +}\rightarrow S^0\rightarrow \wt{\rm E}_{\Sigma_2}C_2.
$$
Here ${\rm E}_{\Sigma_2}C_{2}$ is the $\Sigma_2$-equivariant universal $C_2$-space, see \cite[VII.1]{May:CBMS} or \cite[Definition B.108]{HHR}.
We have 
$\wt{\rm E}_{\Sigma_2}C _2=\colim _{n}S^{n\sigma\otimes\epsilon}\wedge S^{n\sigma}$. For any representation $V$, the cyclic permutation on $S^{3V}$ is equivariantly homotopic to the identity (e.g., a similar argument as in \aref{lem:cyclicistrivial} works topologically) and so we conclude the following. 

\begin{proposition} \label{prop:tildeRperiod}
The unit maps $S^0\rightarrow S^\sigma$, $S^0\rightarrow S^{\sigma\otimes\epsilon}$ induce $C_2\times\Sigma _2$-equivariant homotopy equivalences 
$\wt{\rm E}_{\Sigma_2}C _2\simeq \wt{\rm E}_{\Sigma_2}C _2\wedge S^\sigma$ and $\wt{\rm E}_{\Sigma_2}C _2\simeq \wt{\rm E}_{\Sigma_2}C _2\wedge S^{\sigma\otimes\epsilon} $. 
In particular, 
$\rH^{\star}_{C_2\times\Sigma_2}( \mathsf{E}\wedge \wt{\rm E}_{\Sigma_2}C _2, \ul{A})$ is $\sigma$-periodic as well as $\sigma\otimes\epsilon$-periodic for any 
$C _2\times\Sigma _2$-equivariant spectrum $\mathsf{E}$.
\end{proposition}

Since \eqref{eqn:cyclemap} is a ring map, 
\aref{thm:periodicity} implies there is an invertible element $v\in H^{\star}_{C _2\times \Sigma_2}({\rm E}_{\Sigma_2}C_2,\ul{\mathbb{Z}/2})$, 
the degree of $v$ is $\sigma -1+\sigma\otimes\epsilon -\epsilon$. This immediately implies the following.

\begin{proposition}\label{prop:topperiod}
Multiplication by $v^{-q}\in 
H^{q-q\sigma+ q\epsilon-q\sigma\otimes\epsilon}_{C_2\times\Sigma_2}({\rm E}_{\Sigma_{2}}C_2,\ul{\Z/2})$ induces a natural isomorphism
$$
\rH^{a+p\sigma + c\epsilon +q\sigma\otimes \epsilon}_{C_{2}\times\Sigma_2}(\mathsf{E}\wedge ({\rm E}_{\Sigma_{2}}C_2)_{+},\ul{\Z/2}) \iso 
\rH^{(a+q) + (p-q)\sigma + (c+q)\epsilon}_{C_{2}\times\Sigma_2}(\mathsf{E}\wedge ({\rm E}_{\Sigma_{2}}C_2)_+,\ul{\Z/2})
$$
of $H^{\star}_{C_2\times\Sigma_2}({\rm E}_{C_2}\Sigma_2,\ul{\Z/2})$-modules for any  $C_2\times\Sigma_2$-spectrum $\mathsf{E}$.	
\end{proposition}

We proceed, as in the previous section, towards our equivariant version of the Beilinson-Lichtenbaum conjecture over the reals by establishing the following case.  
\begin{lemma}\label{lem:freeinpR}
	Let $U$ be a smooth quasi-projective real $C_{2}$-variety with free action. Then
	 $$
	 \RRe_{\C}: H^{m+i\sigma, n}_{C_{2}}(U, \ul{\Z/2}) \to H^{m-n+i\sigma+n\epsilon}_{C_{2}\times\Sigma _2}(U(\C), \ul{\Z/2}),
	 $$
	 is an isomorphism if $m\leq n$ and $m+i\leq n$. It is 
a monomorphism if  $m \leq n+1$ and  $m+i\leq n+1$. 
\end{lemma}
\begin{proof}
 If $Y$ is a $C_2\times\Sigma_2$-space, then we have $H^{s +t\epsilon}_{C_{2}\times \Sigma_2}(Y, \ul{A}) \iso 	H^{s +t\sigma}_{\Sigma_{2}}(Y/C_2, \ul{A})$. 
Moreover, if $Y=X(\C)$, where $X$ is a real variety with free $C_2$-action, arguing as in \aref{lem:freequot}, we have a commutative square
$$
\xymatrix{
H^{m,n}_{C_{2}}(X,\ul{A}) \ar[d]_{\RRe_{\C,\,\Sigma_2}}  \ar[r]^{\iso}
&
H^{m,n}_{\mcal{M}}(X/C_2, A) \ar[d]^{\RRe_{\C,\,\Sigma_2}}
\\
H^{m-n +n\epsilon}_{C_{2}\times \Sigma_2}(X(\C), \ul{A}) \ar[r]^{\iso} & 	
H^{m-n +n\sigma}_{\Sigma_{2}}(X(\C)/C_2, \ul{A}).
}
$$
When $i=0$ the result thus holds by \eqref{eqn:blreal}.  The result follows in general using induction and considering the comparison of exact sequences obtained from \eqref{eqn:fun1},
\[
\xymatrix@C-1.3pc{
 \ar[r] & H^{m+(j-1)\sigma,n}_{C_2}(U) \ar[r]\ar[d] & H^{m+j\sigma,n}_{C_2}(U) \ar[r]\ar[d] & H^{m+j,n}_{\mcal{M}}(U) \ar[r]\ar[d] &  \\
\ar[r] & H^{m-n+n\epsilon+(j-1)\sigma}_{C_2\times\Sigma _2}(U(\C)) \ar[r] & H^{m-n+n\epsilon+j\sigma}_{C_2\times\Sigma _2}(U(\C)) \ar[r] & H^{m+j-n+n\sigma}_{\Sigma _2}(U(\C))\ar[r]&  . 
}
\]

\end{proof}

\begin{proposition} \label{EG} 
Let $X$ be a smooth real variety with $C_2$-action and $A$ a finite abelian group. Then the map
$$
H^{a+p\sigma,b+q\sigma}_{C_2}(X \times \EG C_2, \ul{A})\rightarrow 
H^{a-b +(p-q)\sigma +b\epsilon+q\sigma\otimes\epsilon}_{C_{2}\times \Sigma_2}(X(\C) \times 
{\rm E} _{\Sigma _2}C_2, \ul{A})
$$
is 
\begin{enumerate}
\item[(i)] an isomorphism if $a\leq b-q$ and $a+p\leq b+q$, and 
\item[(ii)] a monomorphism if $a \leq b-q+1$ and $a+p\leq b+q+1$. 
\end{enumerate}

\end{proposition}
\begin{proof}  
It suffices to consider the case of a quasi-projective $X$ by the same argument as in the second half of \aref{prop:free}.
It also suffices to consider the case $A=\Z/2$ by the same argument as in \aref{prop:free}, 
using an obvious variant of \aref{prop:mtwo}.   

Using \aref{cor:period} and \aref{prop:topperiod}, we can replace $(a+p\sigma, b+q\sigma)$ by $(a+2q+(p-2q)\sigma, b+q)$ and $a-b+(p-q)\sigma + b\epsilon + q\sigma\otimes\epsilon$ by 
$(a-b-q) +(p-2q)\sigma+(b+q)\epsilon$. For simplicity, we reindex, replacing $a+2q+(p-q)\sigma$ by $m+i\sigma$ and $b+q$ by $b$. The hypothesis on $(a+p\sigma, b+q\sigma)$ implies that  $(m+i\sigma,b)$ satisfies the hypothesis of the previous lemma \ref {lem:freeinpR}. 

Considering  the comparison of Milnor exact sequences, as in \aref{prop:free}, we see from lemma \ref{lem:freeinpR} that
\begin{equation}\label{eqn:c2s2map}
H^{m+i\sigma,b}_{C_{2}}(U_{n}, \ul{\Z/2}) \to H^{m-b+i\sigma+b\epsilon}_{C_{2}\times C _2}(U_{n}(\C), \ul{\Z/2})
\end{equation}
is an isomorphism for $m+i\leq b$ and $m\leq b$ and a monomorphism for $m+i=b+1$ and $m\leq b+1$. Here $U_{n}:=X\times(\A(n\sigma)\setminus  \{0\})$. Replacing back $m$ by $a+2q$ and $i$ for $p-2q$ and $b$ for $b+q$ we get the result.
 
\end{proof}  
    
We also have the following:
\begin{proposition} \label{first} 
Let $X$ be a smooth real variety with $C_2$-action and $A$ a finite abelian group. The map
$$
\rH^{a+p\sigma,b+q\sigma} _{C _2}(X _+\wedge \wt\EG C_2,\ul{A})\to   \rH^{a-b+b\epsilon}_{C_2\times \Sigma_2}(X(\C) _+\wedge \wt{\rm E}_{\Sigma _2}C _2, \ul{A})
$$ 
is an isomorphism for $a\leq b$ and a monomorphism for $a=b+1$. 
\end{proposition}
\begin{proof}

In the case of a free $C _2$-action on $X$ the above map is an isomorphism for all the indexes because the groups are zero. 
Consider the comparison of long exact sequences in cohomology associated to the motivic isotropy cofiber sequence
$$
X _+\wedge \EG C _2\rightarrow X _+\rightarrow X _+\wedge \wt{\EG} C _2.
$$
Suppose that $X$ has  trivial action. Using \aref{cor:degshift} and \aref{prop:tildeRperiod}, we may assume that $p=0$ and $q=0$. Consider the map
$$
\rH^{a,b} _{C _2}(X _+\wedge \wt{\EG} C_2,\ul{A})\rightarrow \rH^{a-b+b\epsilon}_{C_2\times \Sigma _2}(X(\C) _+\wedge \wt{\rm E} _{\Sigma _2}{C _2}, \ul{A}).
$$
 	  
Because $X$ has a trivial action, we have that the middle map in the long exact sequence given by the motivic isotropy cofiber sequence is identified with
$$
H^{a,b} _{\mcal{M}}(X,\ul{A})  \rightarrow H^{a-b,b}_{C_2}(X(\mathbb{C}), \ul{A}).
$$
This is an isomorphism if $a\leq b$ and a monomorphism if $a=b+1$ by the Beilinson-Lichtenbaum conjecture  \eqref{eqn:blreal}.
 	
The other significant map in the diagram is 
$$
\rH^{a,b} _{C _2} (X \times \EG C _2,\ul{A})
\rightarrow 
\rH^{a-b+b\epsilon} _{C _2\times \Sigma _2} (X(\C)\times E _{\Sigma_2}{C _2},\ul{A})
$$ 
which according to Proposition \ref{EG} is an isomorphism for any $a\leq b$ and monomorphism for any $a=b+1$. This confirms the isomorphism in the statement when the action is trivial. 
 	
The isomorphism in the general case follows as in Proposition \ref{prop:isotrop} by considering the equivariant cofiber sequence for the real variety $X$
$$
X\setminus X^{C _2}\rightarrow X\rightarrow Th(\mathcal{N}),
$$
where $\mathcal{N}$ is the normal bundle of the inclusion $X^{C _2}\subset X$.  
\end{proof}

\begin{theorem}  \label{equivBLconjecture}
Let $X$ be a smooth real variety with $C_2$-action and $A$ an abelian group. Then the map
$$
H^{a+p\sigma,b+q\sigma}_{C_2}(X, \ul{A})\to 
H^{a-b+(p-q)\sigma + b\epsilon +q\sigma\otimes\epsilon}_{C_{2}\times \Sigma_2}(X(\C) , \ul{A})
$$
is   
\begin{enumerate}
\item[(i)] an isomorphism if both $a+p\leq b+q$ and $a\leq min\{b-q,b\}$,
\item[(ii)] an injection if both $a+p\leq b+q+1$ and $a\leq min\{b-q,b\}+1$.
\end{enumerate}
\end{theorem}
\begin{proof}
This follows by comparing the long exact sequences induced by the motivic isotropy cofiber sequence $X_+\wedge\EG C_{2\,+} \to X_+ \to X_+\wedge \wt\EG C_2$ together with \aref{EG} 
and \aref{first}.
\end{proof}

When $p=q=0$ and $X$ a smooth real variety with trivial $C_2$-action the above theorem reduces to the version of the Beilinson-Lichtenbaum conjecture for a real variety (see \ref{eqn:blreal}) 
established in \cite{HV}. (Of course this was used to prove the theorem.)

\appendix
\section{Equivariant motivic homotopy theory}\label{app:EMHT}
Unstable equivariant motivic homotopy theory was first considered  by Voevodsky \cite{Deligne:V}.  
A stable version was considered by Hu-Kriz-Ormsby \cite{HKO} as part of their work on the completion problem in Hermitian $K$-theory. 
General foundations and model structures are constructed in \cite{HKO:EMHT} and representability results for equivariant algebraic $K$-theory are also established. A general framework for stable equivariant motivic homotopy theory emphasizing the six functor formalism is introduced in \cite{Hoyois:six}.
Alternate versions of a homotopy theory for smooth $G$-schemes are studied in \cite{Herrmann} and \cite{CJ};
however, 
theories of interest, such as equivariant algebraic $K$-theory, are not representable in the homotopy categories constructed there.

The main results in this paper rely on a Betti or topological realization functor, 
which is constructed in this appendix, relating equivariant motivic homotopy and classical equivariant homotopy theory. 
We begin by giving a brief but self-contained construction of a model for the unstable and stable equivariant motivic homotopy categories, 
which is geared towards the construction of the Betti realization functor. 
We also record the details of the construction of several comparison functors between equivariant and nonequivariant homotopy categories in this setting. Finally in the last sections of this appendix we verify that the topological realization of the Bredon motivic cohomology spectrum is the topological Bredon cohomology spectrum.

\subsection{Unstable equivariant motivic homotopy theory}
To keep exposition streamlined, 
we restrict attention to the case of a finite group $G$ over a field $k$. Furthermore we always assume that the order of $G$ is invertible in $k$.
A \emph{motivic $G$-space} over $k$ is defined to be a presheaf of simplicial sets on $G\Sm/k$. 
We write $G\MS(k)$ and $G\bMS(k)$ respectively for the categories of motivic $G$-spaces and pointed motivic $G$-spaces over $k$. 
We are primarily interested in the stable equivariant motivic homotopy category, 
and so we only treat the unstable model structure for pointed motivic $G$-spaces. The category $G\bMS(k)$ is a symmetric monoidal category via 
the pointwise smash product $(F\wedge G)(X) := F(X)\wedge G(X)$ of motivic $G$-spaces. 

We make use of an equivariant version of the ``closed flasque model structure'' introduced in \cite{PPR:KGL}, a variation on the flasque model structure of \cite{I:flasque}. 
This model structure is particularly well-suited for topological realization. 

Let $\mcal{Z}=\{Z_{i}\hookrightarrow  X\}$ be a finite collection of closed immersions ($\emptyset \to X$ is allowed) in $G\Sm/k$. 
Define 
$$
\cup \mcal{Z} : = \mathrm{coeq}( \coprod_{r,r'} Z_{r}\times_{X} Z_{r'} \rightrightarrows \coprod_{r} Z_{r}),
$$ 
where the coequalizer is computed in $\MS(k)$ (i.e., $\cup Z$ is the categorical union of the $Z_{r}\subseteq X$). 
Write $i_{\mcal{Z}}:\cup \mcal{Z} \to X$ for the resulting monomorphism. 

The pushout product $i\ \Box \ j$ of two maps $i:A\to X$ and $j:B\to Y$ is defined to be the map $A\wedge Y \coprod_{A\wedge B} X\wedge B \to X\wedge Y$.
Define two sets of maps:
\begin{enumerate}
\item  
$I^c$ is the set of maps of the form $(i_{\mcal{Z}})_+\ \Box \ g_+$, 
where $\mcal{Z}$ is a finite set of closed  equivariant immersions in $G\Sm/k$ and $g:\partial\Delta^n \to \Delta^n$, $n\geq 0$ is the standard inclusion.
\item 
$J^c$ is the set of maps of the form  $(i_{\mcal{Z}})_+\  \Box \ g_+$, 
where $\mcal{Z}$ is a finite set of closed  equivariant immersions in $G\Sm/k$ and $g: \Lambda^{n,k}\to\Delta^n$, $n\geq 1$, $0\leq i \leq n$ is the standard inclusion.
\end{enumerate}

\begin{definition}
\label{definition:globalclosedflasque}
Let $f:F\to G$ be a map of pointed motivic $G$-spaces.
\begin{enumerate}
 \item Say that $f$ is a \emph{schemewise weak equivalence} provided $f:F(U)\to G(U)$ is a weak equivalence of simplicial sets for all $U$ in $G\sm/k$. 
\item Say that $f$ is a \emph{closed flasque fibration} if it has the right lifting property with respect to $J^c$. 
\item Say that $f$ is a \emph{closed flasque cofibration} if it has the left lifting property with respect to acyclic closed schemewise fibrations. 
\end{enumerate}
\end{definition}

The schemewise weak equivalences, closed flasque cofibrations, and closed flasque fibrations define the \emph{global closed flasque model structure} on $G\bMS(k)$, 
see e.g., \cite[Theorem 3.8]{HKO:EMHT}.

\begin{proposition}
The global closed flasque model structure is a simplicial, proper, cellular,  monoidal model structure on $G\bMS(k)$. The sets $I^c$ and $J^c$ are respectively generating cofibrations 
and acyclic cofibrations.
\end{proposition}

For a distinguished equivariant Nisnevich square $Q$ as in (\ref{eqn:edist}), write $P_{Q}$ for the pushout of $A\leftarrow B \rightarrow Y$ in $G\MS(k)$.
\begin{definition}
\label{definition:localmotivicclosedflasque}
\begin{enumerate}
\item 
The \emph{closed flasque (equivariant Nisnevich) local model structure} on $G\bMS(k)$ is the left Bousfield localization of the global closed flasque model structure 
at the set of maps $\{(P_{Q})_+\to X_+\}$, where $Q$ ranges over the set of  distinguished equivariant Nisnevich squares. 
The associated homotopy category is denoted $\bGH^{Nis}(k)$.
\item 
The \emph{closed flasque motivic model structure} on $G\bMS(k)$ is the left Bousfield localization of the closed flasque local model structure at the set of projections 
$(X\times\A^{1})_{+}\to X_+$ for all $X$ in $G\Sm/k$. For brevity, we refer to the weak equivalences and fibrations of this model structure as \emph{motivic weak equivalences} and \emph{motivic fibrations}.
The associated homotopy category is the unstable equivariant motivic homotopy category and is denoted $\bGH(k)$. 
\end{enumerate}
\end{definition}

\begin{theorem}
\label{theorem:localmotivicclosedflasque}
The closed flasque local and closed flasque motivic model structures are simplicial, proper, cellular, monoidal model structures on $G\bMS(k)$. 
The identity functor from the  projective model structures to the closed flasque model structures is a left Quillen equivalence.
In particular, 
the homotopy category $\bGH(k)$ coincides with the one defined in \cite{HKO:EMHT} (and hence also with the one defined by Voevodsky in \cite{Deligne:V}). 
\end{theorem}
\begin{proof}
The global closed flasque model structure is left proper (in fact proper) as well as cellular and so by \cite[Theorem 4.1.1]{Hirschhorn} the left Bousfield localization of this 
model structure at a set of maps exists (and is again left proper and cellular). 
This implies that the above model structures exist, are simplicial, left proper, and cellular. 
Right properness follows from the fact that the local projective and motivic model structures are proper,
see \cite[Theorem 4.3]{HKO:EMHT}.

Every projective cofibration is a closed flasque cofibration and 
the weak equivalences in the global model structures coincide so that the identity is a Quillen equivalence between the global model structures, 
and hence a Quillen equivalence on the localized model structures.

To show that these are symmetric monoidal model structures we need to check that if $f$ and $g$ are cofibrations, 
then the pushout product of $f\ \Box \ g$ is also a cofibration and that it is a weak equivalence if either $f$ or $g$ is an acyclic cofibration. 
It suffices to assume that $f$ and $g$ are generating cofibrations. 
Note that taking smash products preserve motivic weak equivalences since every object is injective cofibrant. 
The pushout product of maps of the form $(i_{\mcal{Z}})_+\to X_+$ are again of the same form. 
The pushout product axiom in simplicial sets therefore implies that the pushout product of closed flasque cofibrations is again a cofibration. 
If one of $f$ or $g$ is a weak equivalence, 
so is the pushout product because the smash product preserves equivariant motivic equivalences. 
\end{proof}

\begin{remark}
It follows from the definitions that $F$ is motivic fibrant on $G\sm/k$ if and only if the following three conditions hold.
\begin{enumerate}
\item 
$F$ is fibrant in the global closed flasque model structure.
\item 
$F$ is \emph{equivariant Nisnevich excisive}, 
i.e., 
for any distinguished square in $G\sm/k$ the square
$$
\xymatrix{
F(X) \ar[r]\ar[d] & F(Y) \ar[d]\\
F(A) \ar[r] & F(B)
}
$$
is homotopy cartesian.
\item 
$F$ is \emph{$\A^1$-invariant}, 
i.e., 
$F(X\times \A^{1})\to F(X)$ is a weak equivalence for all $X$.
\end{enumerate}
\end{remark}

\subsection{Hypercohomology and  motivic homotopy}\label{sub:hyper}
Let $\mcal{F}$ be a presheaf of simplicial abelian groups on $G\sm/k$. 
Write $\mcal{N}\mcal{F}$ for the associated presheaf of normalized cochain complexes.
Forgetting the group structure, we view $\mcal{F}$ as a pointed motivic $G$-space, with $0$ as basepoint. 
There is a natural isomorphism (see e.g., \cite[Proposition 2.1.26]{MV:A1}) of homotopy classes of maps in $\bGH^{Nis}(k)$ and sheaf cohomology groups
$$
[S^{n}\wedge X_{+}, \mcal{F}]_{\bGH^{Nis}(k)} \iso  H^{-n}_{GNis}(X,(\mcal{N}\mcal{F})_{GNis}).
$$

Moreover, 
if $\mcal{F}$ is equivariant Nisnevich excisive then both of these groups agree with the homotopy group $\pi_{n}\mcal{F}(X)$.

In \cite{HVO:cancellation} we introduced an equivariant generalization of Voevodsky's machinery of presheaves with transfers. 
A presheaf with equivariant transfers is an additive presheaf on the category $G\Cor_{k}$, 
which has the same objects as $G\Sm/k$ and whose maps are given by $ \Cor_{k}(X,Y)^{G}$.
From the viewpoint of motivic homotopy theory, 
a fundamental feature of the transfer structure is that it allows one to construct a ``small'' motivic fibrant replacement. 

Write $\mcal{L}_{GNis}$ for a fibrant replacement functor in the equivariant Nisnevich local model structure of Theorem \ref{theorem:localmotivicclosedflasque}.
If $F$ is a presheaf of abelian groups recall that $C_{*}F(X)$ is the simplicial abelian group $F(X\times\Delta^{\bullet}_{k})$. 
By the exponent $\exp(G)$ of a group $G$ we mean the least common multiple of the orders of elements of the group.

\begin{theorem}\label{thm:rcoh}
Suppose that $G$ is abelian group, 
$|G|$ is coprime to ${\rm char}(k)$, 
and $k$ contains a primitive $\exp(G)$th-root of unity.  
Let $\mcal{F}$ be a presheaf with equivariant transfers. 
Then $\mcal{F}\to \mcal{L}_{GNis}C_{*}\mcal{F}$ is a motivic fibrant replacement functor. 
In particular, 
we have a natural isomorphism 
$$
[S^{n}\wedge X_{+},\mcal{F}]_{\bGH(k)} \iso  H^{-n}_{GNis}(X,\mcal{N}C_{*}\mcal{F}).
$$
\end{theorem}
\begin{proof}
By definition $\mcal{L}C_*\mcal{F}$ is closed flasque fibrant and equivariant Nisnevich excisive. 
It remains to see that $\mcal{L}_{GNis}C_{*}\mcal{F}$ is $\A^1$-invariant.  
We have natural isomorphisms
$$
\pi_{n}\mcal{L}_{GNis}C_*\mcal{F}(X) 
\iso 
[S^{n}\wedge X_+, C_*\mcal{F}]_{\bGH^{Nis}(k)} 
\iso 
H^{-n}_{GNis}(X,\mcal{N}C_*\mcal{F}).
$$
The $\A^1$-invariance of this presheaf follows from \cite[Theorem 8.12]{HVO:cancellation} together with an application of the spectral sequence 
$$
H^{p}_{GNis}(X, \mcal{H}^{q})\Longrightarrow H^{p+q}_{GNis}(X, \mcal{N}C_*\mcal{F}),
$$ 
where $\mcal{H}^{q}$ is the sheafification of the presheaf $U\mapsto H^{q}\mcal{N}C_*\mcal{F}(U)$. 
\end{proof}

\subsection{Adjunctions of motivic spaces}\label{sub:adjunction}
Next we record some motivic analogues of familiar adjunctions in topology relating $G$-spaces and ordinary spaces.
We begin with the adjunctions
\begin{align*}
i:=(-)^{triv}:\Sm/k & \rightleftarrows G\Sm/k: (-)^{G}=:\phi,
\\ 
\epsilon:=G\times - :\Sm/k & \rightleftarrows G\Sm/k: (-)^{e}=:\rho.
\end{align*}
Here, $(-)^{e}$ simply forgets the action, the underlying scheme of $X^{triv}$ is $X$ and it is considered as a $G$-scheme with trivial action, and 
$X^{G}$ is the fixed points scheme, which is smooth since $|G|$ is invertible in $k$, see e.g., \cite[Proposition 3.4]{Edixhoven}).
We obtain several adjoint pairs of functors on based motivic spaces
\begin{align*}
i^*:\spc_{\bullet}(k) & \rightleftarrows G\spc_{\bullet}(k): i_*,
\\ 
i_*=\phi^*: G\spc_{\bullet}(k) & \rightleftarrows \spc_{\bullet}(k): \phi_*,
\\
\epsilon^*:\spc_{\bullet}(k) & \rightleftarrows G\spc_{\bullet}(k): \epsilon_*,
\\ 
\epsilon_*=\rho^*: G\spc_{\bullet}(k) & \rightleftarrows 
\spc_{\bullet}(k): \rho_*.
\end{align*}

\begin{proposition}
 Each of the pairs  $(i^*,i_*)$, $(\phi^*,\phi_*)$, $(\epsilon^*,\epsilon_*)$, and $(\rho^*, \rho_*)$ are Quillen adjoints on the closed flasque motivic model structures. 
\end{proposition}
\begin{proof}
Each of the functors $i$, $\phi$, $\epsilon$, and $\rho$ commute with fiber products and preserve closed immersions.  It follows that $i^*$, $\phi^*$, $\epsilon^*$, and $\rho^*$ preserve generating cofibrations and generating acyclic cofibrations for the global closed flasque model structure. These are thus left Quillen adjoints on the global closed flasque model structure.

Both $i$ and $\epsilon$ send distinguished Nisnevich squares to distinguished equivariant Nisnevich squares. 
The functor $\rho$ sends distinguished equivariant Nisnevich squares to distinguished Nisnevich squares 
and by \cite[Corollary 3.2.6]{Herrmann},  
$\phi$ does as well. 
Moreover,  $i$, $\phi$, $\epsilon$, and $\rho$ all
send maps of the form $X\times\A^{1}\to X$ to maps of the same form. It follows by the universal property of Bousfield localization that the functors $i^*$, $\phi^*$, $\epsilon^*$, and $\rho^*$ are also left Quillen functors on the closed flasque motivic model structure.

\end{proof}

We may now define motivic analogues of the classical change of groups functors.
\begin{definition}
Define the
\begin{enumerate}
\item \emph{trivial action} functor by $(-)^{triv}:=  \L i^*$,
\item \emph{$G$-fixed points} functor by  $ (-)^{G} := \R i_*$,
\item \emph{induced motivic $G$-space} functor by $G_+\wedge - := \L \epsilon^*$, 
\item \emph{coinduced motivic $G$-space} functor by $F(G_+, -):= \R\rho_{*}$,
\item \emph{underlying motivic space} functor by $(-)^{e}:= \R \epsilon_*$.
\end{enumerate}
\end{definition} 
Note that  $i_*= \phi^*$ and $\epsilon_*=\rho^*$ are both Quillen left and Quillen right adjoints. 
We thus have natural motivic equivalences $\R i_*=(-)^{G} \wkeq \L i_*$ and $\R \epsilon_*=(-)^{e}\wkeq \L \epsilon_*$.

In summary, we have adjunctions
\begin{align*}
(-)^{triv}:\HH(k) & \rightleftarrows \bGH(k): (-)^{G}, \\
G_+\wedge - :\HH(k) & \rightleftarrows \bGH(k) :  (-)^{e}, \\
(-)^{e} :\bGH(k) & \rightleftarrows \HH(k) :  F(G_+, -).
\end{align*}

We note that $(G_+\wedge X)^{e} \wkeq \coprod_{|G|}X$. 
Indeed, we have $\epsilon_*\epsilon^*(X) = \coprod_{|G|}X$ since
this formula holds for smooth schemes and both sides commute with colimits. If $X^{cof}$ is a cofibrant replacement for $X$, then  we have the weak equivalences $(G_+\wedge X)^{e} = \R\epsilon_*\L\epsilon^*(X) \wkeq \L\epsilon_*\L\epsilon^*(X) \wkeq
 \epsilon_*\epsilon^*(X^{cof}) = \coprod_{|G|}X^{cof}$, which establishes the claim. 
The inclusion $X^{e}\to (G_+\wedge X)^{e}$ at the summand corresponding to $e\in G$ induces the map $i:G_+\wedge X^{e}\to G_+\wedge X$.

\begin{lemma}\label{lem:trick}
The map $i: G_+\wedge X^{e} \xrightarrow{\iso} G_{+}\wedge X$ is an isomorphism in  $\bGH(k)$.
\end{lemma}
\begin{proof}
We have that $G_+\wedge (-)^{e} \iso \L\epsilon^*\L\epsilon_*(-) = \L\epsilon^*\L\rho^*(-)$. 
Note that these commute with homotopy colimits, 
as does $G_{+}\wedge -$. 
It thus suffices to assume that $X$ is representable. 
If $Y$ is a $G$-scheme, 
then the map $G\times Y^e \to G\times Y$ given by $(g,y)\mapsto (g,gy)$ yields the desired equivariant isomorphism.
\end{proof}

\subsection{Stable equivariant motivic homotopy theory}
\label{subsection:Stableequivariantmotivichomotopytheory}
Stable model structures on motivic spectra are constructed in \cite{Jardine:motspt}, yield the stable motivic homotopy category. 
A stable equivariant motivic homotopy category was first constructed in \cite{HKO}.
In this paper we will work with model structures for equivariant motivic spectra by using Hovey's machinery \cite[Section 8]{Hovey:spt}.

Let $V$ be a representation. 
The associated motivic representation sphere is the quotient 
$T^{V} := \P(V\oplus 1)/\P(V)$, where the quotient is taken in the category of presheaves. It is naturally an object of $G\spc_{\bullet}(k)$. We also write $T^{ nV}:= (T^{V})^{\wedge n}$. Since $\P(V)\to \P(V\oplus 1)$ is a closed flasque cofibration, 
$T^{V}$ is closed flasque cofibrant motivic $G$-space.  
We write $\rG$ for the regular representation.

Let $K$ be a pointed motivic $G$-space. 
We write $\ihom(K,-)$ for the right adjoint of $K\wedge -$. 
In particular,
$\Sigma_{T^{\rG}}F = T^{\rG}\wedge F$ and $\Omega_{T^{\rG}}F = \ihom(T^{\rG},F)$. 

\begin{definition}\label{def:spectra}
A symmetric $K$-spectrum  is a sequence $E=(E_{0}, E_{1},\ldots)$ consisting of objects $E_n$ of $G\spc_{\bullet}(k)$ together with the following data
\begin{enumerate}
\item[(i)] a $\Sigma_n$-action  on $E_n$, 
\item[(ii)] $\Sigma_n$-equivariant maps $\sigma_n:E_n\wedge K \to E_{n+1}$,
\end{enumerate}
where the structure maps $\sigma_n$ are required to satisfy the condition that the composites $E_{n}\wedge K^{\wedge p} \to E_{n+p}$ are $\Sigma_n\times\Sigma_p$-equivariant for all $n,p\geq 0$. 
\end{definition}
A map $E\to F$ of symmetric $K$-spectra is a collection of $\Sigma_n$-equivariant maps $E_n\to F_n$ which are compatible with the structure maps. 
Write $G\sspt_{K}(k)$ for the category of symmetric $K$-spectra in $G\spc_{\bullet}(k)$. 
Now suppose that $K$ is a closed flasque cofibrant based motivic $G$-space.
The category $G\spc_{\bullet}(k)$ equipped with the closed flasque motivic model structure is a left proper, cellular, simplicial, symmetric monoidal model category, 
and so we can use \cite[Definition 8.7]{Hovey:spt} to define a stable model structure on the category of symmetric $K$-spectra.
It is again a symmetric monoidal model category by \cite[Theorem 8.11]{Hovey:spt}. 
If $K'$ is another closed flasque cofibrant based motivic $G$-space, 
we write $G\sspt_{K,K'}(k)$ for the category of symmetric $(K,K')$-bispectra (i.e., symmetric $K'$-spectra in $G\sspt_{K}(k)$). 
This is again equipped with the stable model structure of \cite[Section 8]{Hovey:spt}. 
  
A symmetric $K$-spectrum $\mcal{E}$ is fibrant if and only if it is an $\Omega_K$-spectrum,
i.e., 
$E_i$ is motivic fibrant and $\mcal{E}_i\to \Omega_{K}E_{i+1}$ is a motivic weak equivalence for all $i$.
By \aref{thm:motivicrep} the Bredon motivic cohomology spectrum $\MA$ is an $\Omega_{T^{\rho_{C_2}}}$-spectrum.

The adjunctions of \aref{sub:adjunction} stabilize.

\begin{proposition}\label{prop:Gstadj}
The adjoint pairs of \aref{sub:adjunction} induce adjoint pairs 
\begin{align*}
(-)^{triv}:\SH(k) & \rightleftarrows \SH_{G}(k): (-)^{G}, \\
G_+\wedge - : \SH(k)&\rightleftarrows \SH_{G}(k): (-)^{e}, \\
(-)^{e}: \SH_{G}(k)&\rightleftarrows \SH(k): F(G_+,-).
\end{align*}
The functors $(-)^{triv}$,  and $(-)^{e}$ are symmetric monoidal and $(-)^{G}$ is lax monoidal.
\end{proposition}
\begin{proof}
First we note that we have an equivalence, as tensor triangulated categories
$$
\SH_{G}(k) \wkeq {\rm Ho}(G\sspt_{T,T^{\rho_{G}}}(k))
$$
where the category $G\sspt_{T,T^{\rho_{G}}}(k)$ of symmetric $(T,T^{\rho_G})$-bispectra equipped with the stable model structure as above. 
Indeed, the model category $G\sspt_{T,T^{\rho_{G}}}(k)$ is isomorphic to 
$G\sspt_{T^{\rho_{G}}, T}(k)$ and the endofunctor $-\wedge T$ on $G\sspt_{T^{\rho_{G}}}(k)$ is a Quillen equivalence, 
since $T\wedge T^{\tilde{\rho}_{G}}= T^{\rho_{G}}$ (where $\tilde{\rho}_{G}$ is the reduced regular representation). 
By \cite[Theorem 9.1]{Hovey:spt}, the stabilization functor (which is symmetric monoidal)
$G\sspt_{T^{\rho_{G}}}(k)\to G\sspt_{T^{\rho_{G}},T}(k)$ is therefore a left Quillen equivalence.

The monoidal functor $i^*:\spc_{\bullet}(k)\to G\spc_{\bullet}(k)$ makes $G\spc_{\bullet}(k)$ into a $\spc_{\bullet}(k)$-model category in the sense of \cite[Definition 4.2.18]{Hovey:model} 
and the Quillen functor $i^*$ is a $\spc_{\bullet}(k)$-module functor (see \cite[Definition 4.1.7]{Hovey:model}). 
By \cite[Theorem 9.3]{Hovey:spt}, 
the Quillen pairs $(i^*, i_*)$  extends to a Quillen pair on spectra. Combined with the stablization adjunction, we have the composition of Quillen functors
$$
\xymatrix{
\sspt_{T}(k) 
\ar@<+.7ex>[r]^{i^*}
& 
G\sspt_{T}(k) 
\ar@<+.7ex>[r]^{\Sigma^{\infty}_{T^{\rho_{G}}}}
\ar@<+.7ex>[l]^{i_*}
&
G\sspt_{T,T^{\rho_{G}}}(k).
\ar@<+.7ex>[l]^{\Omega^{\infty}_{T^{\rho_G}}}
}
$$

The pair $((-)^{triv}, (-)^G)$ is the induced adjunction on homotopy categories, i.e., $(-)^{triv} = \L(\Sigma^{\infty}_{T^{\rho_{G}}}\circ i^*)$ and 
$(-)^G = \R(i_*\circ \Omega^{\infty}_{T^{\rho_G})}$. 

The other two adjunctions are attained as follows. 
We use $\sspt_{T^{|G|}}(k)$ as our model for $\SH(k)$. 
Note for $X$ in $\spc_{\bullet}(k)$ and $Y$ in $G\spc_{\bullet}(k)$ we have a natural isomorphism $\epsilon^*(X\wedge \rho^*(Y)) \iso \epsilon^*(X)\wedge Y$. 
Indeed, this holds when $X$ is in $\sm k$ and $Y$ is in $G\sm/k$ and both sides commute with colimits.
Since $\rho^*(T^{\rho_G}) = T^{|G|}$, we have $\epsilon^*(X \wedge T^{|G|}) = \epsilon^*(X)\wedge T^{\rho_{G}}$. 
We also have $\rho^*(Y\wedge T^{\rho_{G}}) = \rho^*(Y)\wedge T^{|G|}$. By \cite[Lemma 4.1]{HO:galois}, the adjunctions 
$(\epsilon^*,\epsilon_*)$ and $(\rho^*, \rho_*)$ extend to Quillen adjunctions on stable model structures
\begin{align*}
\epsilon^*:\sspt_{T^{|G|}}(k) & \rightleftarrows G\sspt_{T^{\rho_{G}}}(k):\epsilon_*, 
\\
\rho^*:G\sspt_{T^{\rho_{G}}}(k) & \rightleftarrows \sspt_{T^{|G|}}(k):\rho_*.
\end{align*}
We have $\epsilon_*=\rho^*:G\spc_{\bullet}(k)\to \spc_{\bullet}(k)$. 
The prolongations to functors on spectra are levelwise isomorphic and it is straightforward to verify that they are in fact isomorphic as spectra. 
In other words, 
$\epsilon_*= \rho^*:G\sspt_{T^{\rho_{G}}}(k) \to \sspt_{T^{|G|}}(k)$. In particular, $\R\epsilon_*\wkeq \L\rho^*$. The adjunctions $(G_+\wedge -, (-)^{e})$ and 
$((-)^{e}, F(G_+,-))$ on the homotopy categories are thus obtained as $G_+\wedge - = \L\epsilon^*$, $F(G_+,-) = \R\rho_*$, and $(-)^{e} = \L\rho^*$.

Since $i^*$ and $\rho^*$ are symmetric monoidal functors, so are 
$(-)^{triv}$ and $(-)^e$. Since $(-)^G$ is right adjoint to a symmetric monoidal functor, it is lax monoidal.
\end{proof}

\subsection{Topological realization over \texorpdfstring{$\C$}{C} (unstable)}\label{sub:top}
If $X$ is a complex variety, 
we consider $X(\C)$ as a topological space with the Euclidean topology. 
If $X$ has a $G$-action, 
then $X(\C)$ also has a $G$-action and $X\mapsto X(\C)$ defines a functor $G\sm/\C\to G\topp$.
The topological realization functor $\RRe_{\C}:\GMS(\C)\to G\topp_{\bullet}$ is  defined by the Kan extension
$$
\RRe_{\C}(F) 
=
\colim_{(X\times \Delta^{n})_+\to F}( X(\C) \times \Delta^{n}_{top})_{+}
$$   
where $\Delta^{n}_{top}$ is the standard topological $n$-simplex, 
considered with trivial action. 
It has a right adjoint $K\mapsto \msing_{\C}(K)$, 
defined by $\msing_{\C}(K)(X) = \shom_{G}(X(\C), K)$ (where $\shom_{G}(-,-)$ is the simplicial set of continuous equivariant maps).
Equip $G\topp_{\bullet}$ with the model structure where a map $X\to Y$ is a weak equivalence or a fibration if $X^{H}\to Y^{H}$ is a weak equivalence or a fibration 
for all subgroups of $G$, see e.g., \cite[Theorem III.1.8]{MM:EOS}. The resulting homotopy category $\bGH$ is the classical unstable equivariant homotopy category. 
For the corresponding model structure on $G$-simplicial sets, 
see e.g.,  
\cite[\S9.2]{DRO:enrichedfunctors}. 

\begin{proposition}
\label{proposition:ReQuillen}
The adjoint pair
$$
\RRe_{\C}:\GMS(\C) \rightleftarrows G\topp_{\bullet}: \msing_{\C}
$$
is a Quillen adjunction and  
$\RRe_{\C}$ is a symmetric monoidal functor. 

\end{proposition}
\begin{proof}
The argument is similar to that given in \cite[Theorem A.23]{PPR:KGL}. 
First we show that $\RRe_{\C}$ is a left adjoint on the closed flasque model structure. For any finite collection of closed immersions 
$\{Z_{i} \to X\}$ in $G\sm/\C$, $\RRe_{\C}((i_{\mcal{Z}})_{+})$ is  the inclusion of an equivariant subcomplex, 
and hence an equivariant cofibration. It follows that 
$\RRe_{\C}$ sends $I^{c}$ to cofibrations in $G\topp_{\bullet}$ and $J^{c}$ to weak homotopy equivalences in $G\topp_{\bullet}$. This implies that 
$\msing_{\C}$ preserves trivial fibrations as well as fibrations between fibrant objects. By Dugger's lemma \cite[Corollary A2]{dugger:simp}, $\RRe_\C$
is thus a left Quillen functor on the closed flasque global model structure.

Next we claim that $\RRe_{\C}$ sends equivariant distinguished squares to homotopy pushouts.  By \cite[Corollary 2.13]{Herrmann}, the $H$-fixed points of an equivariant distinguished square is a distinguished  Nisnevich square in $\Sm/k$.  A square in $G\topp_{\bullet}$ is a homotopy pushout if and only if it is so on all fixed points. Since $X^{H}(\C) = X(\C)^{H}$, it suffices to show that the square obtained by taking complex points of a distinguished Nisnevich square is a homotopy pushout square in topological spaces. This follows from \cite[Theorem 5.2]{DI:hyp}, establishing the claim. As $\A^{1}(\C)$ is equivariantly contractible, 
$\RRe_{\C}(X_{+})\to \RRe_{\C}((X\times \A^{1})_{+})$ is an equivariant homotopy equivalence. It follows by the universal property of Bousfield localization that $\RRe_{\C}$ is a left Quillen functor on the closed flasque motivic model structure. 

That $\RRe_{\C}$ is symmetric monoidal is a simple consequence of the fact that there is a natural equivariant homeomorphism $(X\times Y)(\C) \iso X(\C)\times Y(\C)$.
\end{proof}

\begin{proposition}\label{prop:commrel}
 The squares of left adjoints commutes up to natural isomorphisms
 $$
 \xymatrix{
 \HH(\C) \ar[r]^-{G_+\wedge -} \ar[d]_-{\RRe_{\C}} & \bGH(\C)\ar[d]^-{\RRe_{\C}} \\ 
 \HH\ar[r]^-{G_+\wedge -}&  \bGH,
 }
\;
\xymatrix{
 \HH(\C) \ar[r]^-{(-)^{triv}} \ar[d]_-{\RRe_{\C}} & \bGH(\C)\ar[d]^-{\RRe_{\C}} \\ 
 \HH\ar[r]^-{(-)^{triv}}&  \bGH,
 }
\;
 \xymatrix{
 \bGH(\C) \ar[r]^-{(-)^{e}} \ar[d]_-{\RRe_{\C}} & \HH(\C)\ar[d]^-{\RRe_{\C}} \\ 
 \bGH\ar[r]^-{(-)^{e}}&  \HH.
 }
 $$
\end{proposition}
\begin{proof}
Using the notation from \aref{sub:adjunction}, we have a natural isomorphism $\RRe_{\C}\circ\epsilon^*(-)\iso G_+\wedge \RRe_{\C}(-)$ of functors $\spc_{\bullet}(\C)\to G\topp$. Indeed, this is clear on representable motivic spaces and since these functors commute with colimits, this suffices. Similarly, we have natural isomorphisms $\RRe_{\C}\circ i^*(-)\iso (\RRe_{\C}(-))^{triv}$ and $\RRe_{\C}\circ \rho^*(-) \iso (\RRe_{\C}(-))^{e}$. These isomorphisms imply the isomorphisms of derived functors on the homotopy categories. 
\end{proof}

\subsection{Topological realization over \texorpdfstring{$\C$}{C} (stable)}
\label{subsection:topologicalrealization(stable)}
Now we turn our attention to a stable realization functor. 
For any real orthogonal representation $V$ there is a representation sphere $S^{V}=V_+$, 
where $(-)_+$ denotes the one-point compactification. 
Since $\C[G]=\R[G]\oplus\R[G]$ as real representations, 
we have $\RRe_{\C}(T^{\rho_{G}}) = S^{2\rho_{G}}$. 

If $\mathsf{E}$ is a motivic $G$-spectrum, 
define the topological $S^{2\rho_{G}}$-spectrum $\RRe_{\C}\mathsf{E}$ by $(\RRe_{\C}\mathsf{E})_{i} = \RRe_{\C}\mathsf{E}_i$ with structure maps 
$$
\RRe_{\C}\mathsf{E}_i\wedge S^{2\rho_{G}} = \RRe_{\C}(\mathsf{E}_{i}\wedge T^{\rho_G}) \to \RRe_{\C}\mathsf{E}_{i+1}.
$$
The functor $\msing_{\C}$ extends as well to a functor on $S^{2\rho_{G}}$-spectra.
In fact,  
we obtain an adjoint pair of functors
$$
\RRe_{\C}:\sspt_{T^{\rho_G}}(\C) \rightleftarrows \sspt_{S^{2\rho_{G}}}(G\topp_{\bullet}):\msing_{\C}.
$$

\begin{theorem}
\label{theorem:stableBettirealization}
The adjoint pairs
\begin{align*}
\RRe_{\C}:\sspt_{T^{\rho_G}}(\C) & \rightleftarrows \sspt_{S^{2\rho_{G}}}(G\topp_{\bullet}):\msing_{\C},  \\
\RRe_{\C}:\sspt_{T^{\rho_G}}(\C)^{fp} &\rightleftarrows \sspt_{S^{2\rho_{G}}}(G\topp_{\bullet}):\msing_{\C}
\end{align*}
are both Quillen adjoint pairs.
Moreover, 
$\RRe_{\C}$ is symmetric monoidal.
\end{theorem}
\begin{proof}
This is straightforward from Proposition \ref{proposition:ReQuillen},
cf.~\cite[Theorem A.45]{PPR:KGL}. 
\end{proof}

\begin{proposition}\label{prop:stablecommrel}
The squares of left adjoints commute up to natural isomorphisms
$$
\xymatrix{
\SH(\C) \ar[r]^-{G_+\wedge -} \ar[d]_-{\RRe_{\C}} & \SH_G(\C)\ar[d]^-{\RRe_{\C}} \\ 
\SH\ar[r]^-{G_+\wedge -}&  \SH_G,
}
\;
\xymatrix{
\SH(\C) \ar[r]^-{(-)^{triv}} \ar[d]_-{\RRe_{\C}} & \SH_G(\C)\ar[d]^-{\RRe_{\C}} \\ 
\SH\ar[r]^-{(-)^{triv}}&  \SH_G,
}
\;
\xymatrix{
\SH_G(\C) \ar[r]^-{(-)^{e}} \ar[d]_-{\RRe_{\C}} & \SH(\C)\ar[d]^-{\RRe_{\C}} \\ 
\SH_G\ar[r]^-{(-)^{e}}&  \SH.
}
$$
\end{proposition}
\begin{proof}
Straightforward using the unstable result in Proposition \ref{prop:commrel}.
\end{proof}

\subsection{Topological realization over  \texorpdfstring{$\R$}{R}}\label{sub:topR}

Write $\Sigma_2={\rm Gal}(\C/\R)$. 
If $X$ is a real variety with $G$-action,
then $X(\C)$ is a $(G\times\Sigma_2)$-space, 
where the $\Sigma_2$-action is via complex conjugation. 
We thus have a functor $G\Sm/\R\to (G\times\Sigma_2)\topp$ which induces the topological realization functor
$\RRe_{\C,\,\Sigma_2}:\GMS(\C)\to (G\times\Sigma_2)\topp_{\bullet}$,
defined by the Kan extension
$$
\RRe_{\C,\,\Sigma_2}(F) 
=
\colim_{(X\times \Delta^{n})_+\to F}( X(\C) \times \Delta^{n}_{top})_{+}.
$$   
Its right adjoint is defined by $\msing_{\C,\,\Sigma_2}(K)(X) = \shom_{G\times\Sigma_2}(X(\C),K)$, 
where $K$ is a $G\times\Sigma_2$-space and $X$ is a smooth real $G$-variety.

\begin{proposition}\label{prop:unstableR}
The adjoint pair
$$
\RRe_{\C,\,\Sigma_2}:\GMS(\R) \rightleftarrows (G\times\Sigma_2)\topp_{\bullet}: \msing_{\C,\Sigma_2}
$$
is a Quillen adjunction and $\RRe_{\C,\Sigma_2}$ is a symmetric monoidal functor. 
\end{proposition}
\begin{proof}
The argument is similar to \aref{proposition:ReQuillen}.
\end{proof}

Now we turn our attention to a stable realization functor. 
Since $\C[G]=\R[G\times \Sigma_2]$ as real representations, 
we have $\RRe_{\C}(T^{\rho_{G}}) = S^{\rho_{G\times\Sigma_2}}$. 

If $\mathsf{E}$ is a motivic $G$-spectrum, 
define the topological $S^{\rho_{G\times\Sigma_2}}$-spectrum $\RRe_{\C,\,\Sigma_2}\mathsf{E}$ by $(\RRe_{\C,\,\Sigma_2}\mathsf{E})_{i} = \RRe_{\C,\,\Sigma_2}\mathsf{E}_i$ with structure maps 
$$
\RRe_{\C,\,\Sigma_2}\mathsf{E}_i\wedge \RRe_{\C,\,\Sigma_2} (S^{\rho_{G\times\Sigma_2}}) = \RRe_{\C,\,\Sigma_2}(\mathsf{E}_{i}\wedge T^{\rho_G}) \to \RRe_{\C,\,\Sigma_2}\mathsf{E}_{i+1}.
$$
The functor $\msing_{\C}$ extends as well to a functor on $S^{\rho_{G\times\Sigma_2}}$-spectra, adjoint to $\RRe_{\C,\,\Sigma_2}$.

\begin{theorem}\label{thm:stableR}
The adjoint pair
\[
\RRe_{\C,\,\Sigma_2}:\sspt_{T^{\rho_G}}(\C) \rightleftarrows \sspt_{S^{\rho_{G\times\Sigma_2}}}((G\times\Sigma_2)\topp_{\bullet}):\msing_{\C,\,\Sigma_2}, 
\]
is a  Quillen adjoint pair.
Moreover, 
$\RRe_{\C,\,\Sigma_2}$ is symmetric monoidal.
\end{theorem}
\begin{proof}
This is straightforward from Proposition \ref{prop:unstableR},
cf.~\cite[Theorem A.45]{PPR:KGL}. 
\end{proof}

\subsection{Symmetric powers}
In order to analyze the topological realization of the Bredon motivic cohomology spectrum, we need to make precise that it is represented by symmetric powers. 
In this subsection we introduce and analyze symmetric powers for motivic $G$-spaces. 
This discussion parallels the treatment of symmetric powers in \cite{Deligne:V}, \cite{Voevodsky:motivicEM}, and 
\cite[Appendix]{Levine:motivicclassical}. 
Throughout this subsection, 
we assume that ${\rm char}(k) = 0$.

Write $G\Sm'/k$ for the category of smooth, 
quasi-projective $G$-schemes over $k$ and 
$G\spc'(k)$ and $G\spc'_{\bullet}(k)$ for the corresponding categories of simplicial presheaves and pointed simplicial presheaves. 
These can also be given a motivic model structure as above. The inclusion $\phi:G\sm'/k\subseteq G\sm/k$ induces an adjoint pair
$\phi^*:G\spc'(k) \rightleftarrows G\spc(k):\phi_*$ 
and similarly for based motivic spaces.

\begin{lemma}\label{lem:qpequiv}
The adjunction  $\phi^*:G\spc'(k) \rightleftarrows G\spc(k):\phi_*$ is a Quillen equivalence on motivic model structures. Similarly for based motivic $G$-spaces.
\end{lemma}
\begin{proof}
 This follows easily from the fact that smooth $G$-schemes are locally affine in the equivariant Nisnevich topology, see \aref{rem:localaffine}.
\end{proof}

Write $G\sch'/k$ for the category of reduced, quasi-projective $G$-schemes of finite type over $k$. 
Consider the functor $(-)^{\times n}:G\sch'/k \to (\Sigma_n\times G)\sch'/k$ which sends a $G$-scheme $X$ to the $n$-fold product $X^{\times n}$, 
where $G$ acts diagonally and $\Sigma_{n}$ acts by permuting the factors.
The composition of $(-)^{\times n}$ and the Yoneda embedding gives us a functor $G\sch'/k \to \spre((\Sigma_{n}\times G)\sch'/k)$ and we define $\Pi^{(n)}(-)$ to be its left Kan extension, 
yielding the functor
$$
\Pi^{(n)}:\spre(G\sch'/k) \to 
\spre((\Sigma_{n}\times G)\sch'/k),
$$
and similarly for pointed motivic spaces.

Let $N\trianglelefteq K$ be a normal subgroup of a group $K$ and write $\Gamma = K/N$ for the quotient group. 
If $X$ is a quasi-projective $K$-scheme over $k$ then a quotient scheme $X/N$ exists and the $K$-action on $X$ induces a $\Gamma$-action on the scheme $X/N$. 
We write 
$q_{K,N}:K\sch'/k \to \Gamma\sch/k$ for the quotient functor, i.e., $q_{K,N}(X) = X/N$. 
The functor $q_{K,N}$ induces an adjoint pair of functors
$$
(q_{K,N})^*:\spre(K\sch'/k)\rightleftarrows \spre(\Gamma\Sch/k):(q_{K,N})_{*}
$$ 
and similarly for based presheaves. 
For $\mcal{X}$ in $K\spc'(k)$, define $\mcal{X}/N:= (q_{K,N})^*(\mcal{X})$.
If $\mcal{X}$ is represented by a quasi-projective $G$-scheme $X$,
then $\mcal{X}/N$ is represented by $X/N$.

Let $\mathcal{C}$ be a category with finite coproducts and $F$ a presheaf of sets on $\mcal{C}$. 
Say that $F$ is $\emph{additive}$\footnote{In \cite{Voevodsky:rad} the term \emph{radditive} is used.} 
provided $F(\emptyset) = \ast$ and $F(X\coprod Y) = F(X)\times F(Y)$ for any $X,Y\in \mcal{C}$.
By \cite[Theorem 2.6]{sifted}, a presheaf of sets on $\mcal{C}$ is additive if and only if the opposite of its category of elements is sifted. In particular an additive presheaf is canonically a sifted colimit of representable presheaves. Recall that filtered colimits and reflexive coequalizers are sifted colimits.

Note that a presheaf on $G\sch'/k$ is additive if and only if it is sheaf for the topology generated by covers by $G$-connected components. 
The inclusion of the full subcategory of presheaves 
$i:{\rm Pre}_{\Sigma}(G\sch'/k)\subseteq {\rm Pre}(G\sch'/k)$ whose objects are additive presheaves  has a left adjoint 
$$
a_\#:{\rm Pre}(G\sch'/k)\to {\rm Pre}_{\Sigma}(G\sch'/k)
$$
given by sheafification (or as a special case of \cite[Lemma 3.8]{Voevodsky:rad}). Since we do not work in the category 
${\rm Pre}_{\Sigma}(G\sch'/k)$, we will also simply write $a_\#$ again for the composite $i a_\#$.

\begin{definition}
Define the \emph{$n$th symmetric product} $\Sym^{n}:\spre(G\sch'/k) \to \spre(G\sch'/k)$  by
$$
\Sym^{n}(\mcal{X}) := \Pi^{(n)}(a_\#\mcal{X})/\Sigma_{n}.
$$
\end{definition}

The inclusion of categories $i:G\sm'/k\subseteq G\sch'/k$ induces an adjoint pair of functors
$$
i^*:G\spc'(k) \rightleftarrows \spre(G\Sch'/k):i_{*}
$$
and similarly for pointed motivic spaces. 
Note that $i_*$ preserves colimits, 
as these are computed sectionwise. 
If $\mcal{X}$ is an object of $\spre(G\sch'/k)$, 
when no confusion should arise we again write $\Sym^n(\mcal{X})$ for $i_*(\Sym^n(\mcal{X}))$ in $G\spc'(k)$ and also for $\phi^*i_*(\Sym^n(\mcal{X}))$ in $G\spc(k)$.

If $\mcal{X}$ is represented by a scheme  $X$ in $G\sch'/k$, then $\Sym^n(\mcal{X})$ is 
 represented by the scheme-theoretic symmetric power $\Sym^n(X)$. Moreover, $\Sym^n(-)$ commutes with sifted colimits.

If $(\mcal{X},x)$ is an object of $\spre_{\bullet}(G\sch'/k)$, 
then $\Sym^{n}(\mcal{X})$ has a canonical basepoint given by $\Sym^{n}(x)$ 
and we write $\Sym^{n}_{\bullet}(\mcal{X})$ for the corresponding pointed motivic space in $G\spc_{\bullet}'(k)$. If $\mcal{X}$ is represented by a pointed scheme  $X$ in $G\sch'/k$, then $\Sym^n(\mcal{X})$ is 
 represented by the scheme-theoretic symmetric power $\Sym^n(X)$. 
Note that $\Sym^n_\bullet(-)$ commutes with sifted colimits. 

\begin{remark}
We avoid here any explicit discussion of the derived functors necessary to make this construction $\A^1$-invariant. However, consideration of the symmetric powers of a coproduct shows that any sensible definition can't both commute with all colimits and compute the correct value of quotient spaces. Some amount of ``derived-ness'' needs to be baked into the definition. This explains  the appearance of $a_\#$ in the definition.
\end{remark}

We have natural pairings
$
(X^{\times m})/\Sigma_m \times (X^{\times n})/\Sigma_{n} 
\to (X^{\times (m+n)})/\Sigma_{m+n}$
and 
$(X^{\times m})/\Sigma_m\times (Y^{\times n})/\Sigma_{n}\to (X\times Y)^{\times mn}/\Sigma_{mn}$,
for  $X,Y\in G\Sch'/k$. 
These induce  addition and multiplication maps, 
\begin{align*}
\Sym_{\bullet}^{m}(\mcal{X}) &\times \Sym_{\bullet}^{n}(\mcal{X}) \xrightarrow{\sigma_{m,n}} 
\Sym_{\bullet}^{m+n}(\mcal{X}) \;\;\text{and} \\ 
\Sym_{\bullet}^{m}(\mcal{X}) &\wedge \Sym_{\bullet}^{n}(\mcal{Y}) \xrightarrow{\mu_{m,n}}\Sym^{mn}_{\bullet}(\mcal{X}\wedge\mcal{Y}).
\end{align*}

Stabilization maps 
${\rm st}_{n+1}:=\sigma_{n,1}(-,\ast):\Sym_{\bullet}^{n}(\mcal{X})\to \Sym_{\bullet}^{n+1}(\mcal{X})$ 
are thus obtained by adding the basepoint. 
Define 
$$
\Sym_{\bullet}^{\infty}(\mcal{X}):=\colim_{n}(\mcal{X} \xrightarrow{{\rm st}_2} \Sym^2_{\bullet}(\mcal{X}) \xrightarrow{{\rm st}_3} \Sym^3_{\bullet}(\mcal{X})\xrightarrow{{\rm st}_4}\cdots )
$$

Recall that the group of finite correspondences $\Cor_{k}(Y,X)$ is the free abelian group generated by integral closed subschemes $Z\subseteq Y\times X$ 
such that $Z$ is finite over $Y$ and dominates an irreducible component of $Y$. 
When $G$ acts on $X$ and $Y$ there is an induced action on $\Cor_{k}(Y,X)$. The group of equivariant finite correspondences is defined to be $G\Cor_{k}(Y,X) := \Cor_{k}(Y,X)^{G}$.
We write $\Z_{tr,G}(X)$ for the presheaf of equivariant correspondences,
$\Z_{tr,G}(X)(Y):=G\Cor_{k}(Y,X)$. 
If $A\subseteq X$ is a closed invariant subscheme we define $\Z_{tr,G}(X/A):=\Z_{tr,G}(X)/\Z_{tr, G}(A)$, the group quotient.

The submonoid $\Cor_{k}^{eff}(Y,X)^{G}$ of effective equivariant correspondences consists of those equivariant correspondences $\sum n_{Z} Z$ such that all $n_{Z}\geq 0$. 
We write $\Z_{tr,G}^{eff}(X)$ for the corresponding presheaf of effective equivariant correspondences.
Let $A\subseteq X$ is a closed invariant subscheme and
define an equivalence relation $\sim$ on $\Z_{tr,G}^{eff}(X)(Y)$ by declaring $Z\sim Z'$ if $Z-Z'\in\Z_{tr,G}(A)(Y)$. 
Now define the presheaf $\Z_{tr,G}^{eff}(X/A):=\Z_{tr,G}^{eff}(X)/\sim$. 
We have that $\Z_{tr,G}(X) = (\Z_{tr,G}^{eff}(X))^{+}$ and $\Z_{tr,G}(X/A) = (\Z_{tr,G}^{eff}(X/A))^{+}$, 
where $(-)^{+}$ denotes group completion. 

Consider the subset $L_{n}(X)(Y)\subseteq \Cor_{k}^{eff}(Y,X)$ of effective correspondences of degree $n$. 
Write $L_{n}^{G}(X)(Y) = (L_{n}(X)(Y))^{G}$ for the subset of equivariant correspondences of degree $n$. Addition of cycles induces a pairing
\[
\sigma_{i,j}:L_{i}^{G}(X) \times L_{j}^{G}(X) \to L_{i+j}^{G}(X)
\] 
Now we consider a pointed $G$-scheme $(X,x)$. 
The presheaf $L_{n}^G(X)$ is pointed by $n(Y\times x)$ in $L_{n}^G(X)(Y)$. 
Let $A\subseteq X$ be a closed invariant subscheme containing $x\in X$. Adding the basepoint induces stabilization maps
\[
{\rm st} = \sigma(-, \ast):L_i^G(X) \to L_{i+1}^G(X).
\]
Define the presheaf $L_{n}^{G}(X/A)$ to be the coequalizer (in $G\spc(k)$)
\begin{equation}\label{eqn:Lcoprod}
\coprod_{j=0}^{n} L_{j}^{G}(A) \times L_{n-j}^{G}(X) \rightrightarrows L_n^G(X) \to L_n^G(X/A)
\end{equation}
where the top arrow is induced by the inclusion $A\subseteq X$ and addition of cycles and the bottom map is the composition 
\[
L_{j}^{G}(A) \times L_{n-j}^{G}(X)\xrightarrow{{\rm proj}} L_{n-j}^{G}(X) \xrightarrow{{\rm st}} L_n^G(X). 
\]
Note that the presheaf $L_n^G(X/A)$ has a canonical basepoint and the quotient map $L_n^G(X)\to L^G_n(X/A)$ is a map of based presheaves. 
The addition map induces an addition map $L_{i}^{G}(X/A)\times L_{j}^{G}(X/A) \to L_{i+j}^{G}(X/A)$ and in particular we have stabilization maps ${\rm st}:L_{i}^G(X/A) \to L_{i+1}^G(X/A)$ obtained by adding the basepoint.

\begin{proposition}
Let $X$ be a pointed, semi-normal, quasi-projective  $G$-scheme over $k$ and $A\subseteq X$ an invariant closed reduced subscheme containing the basepoint. 
There are isomorphisms $L_{n}^{G}(X)\iso \Sym^n_{\bullet}(X)$ in $G\spc_{\bullet}(k)$ which induce isomorphisms $L_{n}^{G}(X/A)\iso \Sym_{\bullet}^{n}(X/A)$.
\end{proposition}
\begin{proof}
There is a natural isomorphism 
$\psi:\Sym_{\bullet}^n(X)\to L_{n}(X)$ by \cite[Theorem  6.8]{SV:homology}, which is given as follows. 
Let $W_{n}\subseteq \Sym^n(X)\times X$ be the image of the closed subscheme $X^{n-1}\times \Delta_X\subseteq X^{n+1}$ under  $\pi_n\times \id_X$, where $\pi_n:X^n\to \Sym^n(X)$ is the quotient. 
If $f:Y\to \Sym^n(X)$ is a map, with $Y$ smooth, 
then $\psi(f)= (f\times \id_{X})^*(W_n)$ is defined to be the pullback cycle. 
This is an equivariant isomorphism. 
Since $\Sym^n(-)$ commutes with sifted colimits, we can compute 
$\Sym^n_{\bullet}(X/A)$ as the coequalizer (in $G\spc(k)$)
\begin{equation}\label{eqn:symcoeq}
\Sym^n(A\coprod X) \rightrightarrows \Sym^n(X) \to \Sym^n_{\bullet}(X/A)
\end{equation}
where the top arrow is induced by the inclusion $A\subseteq X$ and the identity on $X$ while the bottom arrow is induced by mapping $A$ to the basepoint and the identity on $X$. Since $\Sym^n(A\coprod X) \iso \coprod_{j=0}^{n} \Sym^j(A)\times  \Sym^{n-j}(X) $, comparing with \eqref{eqn:Lcoprod} yields the isomorphisms  $L_{n}^{G}(X/A)\iso \Sym_{\bullet}^{n}(X/A)$. 
\end{proof}

The maps $L_{n}^{G}(X)\to \Z_{tr,G}^{eff}(X)$ induce maps 
$L_{n}^{G}(X/A)\to \Z_{tr,G}^{eff}(X/A)$
which are compatible with the stabilization 
$L_{n}^{G}(X/A)\to 
L_{n+1}^{G}(X/A)$. We thus have an induced map $\colim_{n}L_{n}^{G}(X/A)\to \Z_{tr,G}^{eff}(X/A)$.
\begin{proposition}\label{prop:Liso} 
Let $(X,x)$ be a pointed smooth quasi-projective $G$-scheme over $k$ and $A\subseteq X$ a closed reduced subscheme which contains the basepoint $x\in X$. 
Then we have isomorphisms in $G\spc_{\bullet}(k)$
$$
\Sym^{\infty}_{\bullet}(X/A)\xleftarrow{\iso}\colim_{n}L_{n}^{G}(X/A) \xrightarrow{\iso} \Z_{tr,G}^{eff}(X/A).
$$
\end{proposition}
\begin{proof}
The left hand isomorphism follows from the previous proposition. The map $\colim_{n}L_{n}(X_+)\to \Z_{tr,G}^{eff}(X)$ is immediately seen to be an isomorphism. The map 
$f:\colim_{n}L_{n}^{G}(X/A) \to \Z_{tr,G}^{eff}(X/A)$ is surjective and we check injectivity. Let $[W],[V]\in L_n^{G}(X/A)(Y)$ and represent them by elements in $L_n^G(X)(Y)$ which we again write as $W,V$. We can write $W=W_A+W'$ and $V=V_A +V'$ uniquely as a sum of effective cycles with $W_A$, $V_A$ supported on $Y\times A$ and no component of the supports of $W'$, $V'$ contained in $Y\times A$. If $f([W])= f([V])$, then we have that $W'=V'$. It follows from the definition of $L_n^G(X/A)$ that $[W]=[V]$.
\end{proof}

\begin{proposition}\label{prop:effwk}
Let $X$ be a smooth quasi-projective $G$-scheme over $k$ and $V$ a representation which contains a copy of the trivial representation. 
Then the map $\Sym^{\infty}_{\bullet}(T^V\wedge X_+)
\to \Z_{tr,G}(T^{V}\wedge X_+)$ is an equivariant motivic weak equivalence.
\end{proposition}
\begin{proof}
If $E$ is a presheaf of sets, write $C_*E$ as usual for the presheaf of simplicial sets defined by 
$C_nE(Y) = E(Y\times \Delta^n_k)$. The map $E\to C_*E$ is an $\A^1$-weak equivalence.

Write $F_1=\Z_{tr,G}^{eff}(T^V\wedge X_+)$ and $F_2=\Z_{tr,G}(T^{V}\wedge X_+)$ and let $\phi:F_1\to F_2$ be the group completion. 
Under the isomorphism of \aref{prop:Liso}, the map in the statement of the proposition is identified with $\phi$.
It suffices to show that $\phi:C_*F_1\to C_*F_2$ is a motivic weak equivalence. 
In fact, we show that if $S$ is a point for the equivariant Nisnevich topology, then $\phi:C_*F_1(S)\to C_*F_2(S)$ is a weak equivalence of simplicial sets.  
We claim that $\pi_{0}C_*F_1(S) = 0$. 
Granted this, since $\phi$ is the group completion of a free commutative monoid, by \cite[Theorem Q4]{FM} the map $\phi$ is a homology equivalence. 
Since $\phi$ is a map between simple spaces, this implies that $\phi$ is in fact a homotopy equivalence.

It remains to see that $\pi_{0}C_*F_1(S) = 0$. 
We have that $S=G\times_{H}\spec(R)$ where $H\subseteq G$ is a subgroup and $R$ is a smooth, local Henselian ring with $H$-action, 
see \cite[Theorem 3.14]{HVO:cancellation}. 
Write $\rho_{H}:G\sch/k\to H\sch/k$ for the functor which restricts the action. 
Note that $\rho_{H}(\Sym^{\infty}_{\bullet}(Y)) = \Sym^{\infty}_{\bullet}(\rho_HY)$ for any based quasi-projective $G$-scheme $Y$ and so for any $H$-scheme $Z$
$$
\Hom_{G\Sch/k}(G\times_{H}Z, \Sym^{\infty}_{\bullet}(X))=
\Hom_{H\Sch/k}(Z, \Sym^{\infty}_{\bullet}(\rho_HX)).
$$
Thus, replacing $G$ with $H$ and $S$ with $\spec(R)$, we may assume that $S$ is a local Henselian $G$-scheme. 
	
Write 
$$
F'_1 = \Z_{tr,G}^{eff}\left((\P(V\oplus 1)/(\P(V\oplus 1)\setminus\P(1))\wedge X_+)\right).
$$
First note that since $\P(V)\subseteq \P(V\oplus 1)\setminus\P(1)$ has an equivariant section, that 
$$
\pi_0C_*F_1(S) \subseteq \pi_0C_*F_1'(S).
$$ 
(In fact, $C_*F_1(S)$ and $C_*F_1'(S)$ are weakly equivalent,  but the weaker statement is enough for our purposes here.)
Therefore, it suffices to show that $\pi_0C_*F_1'(S)=0$.
An element of $C_dF_1'(S)$ is represented by an effective invariant cycle $\sum n_iZ_i$ on
$\Delta^d_{S\times X} \times \P(V\oplus 1)$ such that $Z_i$ is finite over $S$ and has  nontrivial intersection with $\Delta^d_{S\times X}\times\P(1)$. 
	
Let $\mcal{Z}$ be an element of $C_0F_1'(S)$ and $\sum n_iZ_i\in \Z_{tr,G}^{eff}(\P(V\oplus 1)\times X)(S)$ an element representing $\mcal{Z}$, as in the previous paragraph. Since the $Z_i$ are finite over $S$, they are local. 
Since $Z_i$ has nontrivial intersection with $\P(1)_{S\times X}$ the closed points are supported in $\P(1)_{S\times X}$ and $Z_i$ are contained in the open 
$\A(V)_{S\times X}\subseteq \P(V\oplus 1)_{S\times X}$ (where $\P(1)_{S\times X} = 0_{S\times X}$ in $\A(V)_{S\times X}$). 
By assumption, $V$ contains a copy of the trivial representation so we can write $V = 1\oplus V'$. 
Consider the equivariant map $\phi:\A^1\times \A(V)_{S\times X} \to \A(V)_{S\times X} $ given by  $(t, (x_i),y)\mapsto ((x_i-t), y)$.
Let $\Phi$ be the cycle on $\A^1\times \A(V)_{S\times X}$ obtained by pull-back of $\mcal{Z}$ along $\phi$. 
Then $\Phi\in C_1F_1'(S)$ and has the property that $\Phi|_{0} = \mcal{Z}$ and $\Phi|_{1}$ is supported in $(\P(V\oplus 1)\setminus \P(1))_{S\times X}$ 
(the closed points lie in $\{1\}_{S\times X}\subseteq \A(V)_{S\times X}$). Thus $\mcal{Z} = 0$ in $\pi_{0}C_*F_1'(S)$.
\end{proof}

\subsection{Realization of Eilenberg-MacLane spectra}
Our next goal is to show that the topological realization, in both the complex and real case, of the Bredon motivic cohomology spectrum $\MA$ is the topological Bredon cohomology spectrum 
(for the constant Mackey functor $\ul{A}$). 
The construction of the motivic $G$-spectrum $\MZ$ is spelled out in detail in \aref{sub:stablerep} for the case $G=C_2$ and the construction for a general finite group $G$ is similar. 
In brief, we set $\MZ_{n} := \Z_{tr,G}(T^{n\rho_{G}})$ and structure maps are defined by $\MZ_n\wedge T^{\rho_{G}} \to \MZ_{n}\wedge \Z_{tr, G}(T^{\rho_{G}}) \to \MZ_{n+1}$, 
the first map being induced by the inclusion $T^{\rho_{G}}\to \Z_{tr,G}(T^{\rho_{G}})$ together with the addition-of-cycles map.

\begin{lemma}\label{lem:n1}
\hspace{2em}
	\begin{enumerate}
\item Let $V$ be a complex representation which contains a trivial summand. For any $X$ in $G\Sch/\C$ the natural map 
$$
\L\RRe_{\C}(T^V\wedge X_+)
\to S^{V(\C)}\wedge X(\C)_+
$$ 
is an equivariant equivalence in $G\topp_{\bullet}$. 
\item Let $V$ be a real representation which contains a trivial summand. 
For any $X$ in $G\Sch/\R$ the natural map 
$$
\L\RRe_{\C,\,\Sigma_2}(T^V\wedge X_+)
\to S^{V(\C)}\wedge X(\C)_+
$$ 
is an equivariant equivalence in $(G\times\Sigma_2)\topp_{\bullet}$. 
\end{enumerate}
\end{lemma}
\begin{proof}
	We treat the complex case, the real case is similar.

Over a field of characteristic zero there are equivariant resolutions of singularities, 
see e.g., 
\cite[Proposition 3.9.1]{Kollar:singularities}.
We may thus find a  simplicial scheme $X_{\bullet}\to X$ over $X$ such that $X^{H}_{\bullet}\to X^{H}$ is a proper $cdh$-hypercover for every subgroup $H\subseteq G$.  We have $|X_{\bullet}|^{H} = |X_{\bullet}^H|$ and since $V^H\neq 0$, it follows from  \cite[Theorem 4.2]{Voev:unstable} that the map
 $ (T^V\wedge |X_{\bullet}|_+)^H \to (T^V\wedge X_+)^H$ 
 is a motivic weak equivalence in $\bMS(\C)$.
Therefore the map $\L\RRe_{\C}(T^V\wedge |X_{\bullet}|_+)
\to S^{V(\C)}\wedge X(\C)_+$ 
is an equivariant weak equivalence in $G\topp_{\bullet}$.

Since each $X_{n}$ is smooth, 
$\L\RRe_{\C}(T^V\wedge |X_{\bullet}|_+) \wkeq 
S^{V(\C)}\wedge |X(\C)_{\bullet}|_+$.
To complete the proof it remains to show that $S^{V(\C)}\wedge |X(\C)_{\bullet}|_+\to S^{V(\C)}\wedge X(\C)_+$ is an equivariant weak equivalence. 
It suffices to show that this map is a weak  equivalence on all fixed points. Note that $X^{H}(\C) = X(\C)^{H}$.
Applying $(-)^{H}$ to the map above 
we obtain the map $S^{2|V^H|}\wedge |X^H(\C)_{\bullet}|_+ \to  S^{2|V^H|}\wedge X^H(\C)_+$.
Since $X^H(\C)_{\bullet}\to X^H(\C)$ is a proper hypercover, 
it  is a universal cohomological descent hypercover \cite[5.3.5]{Deligne:hodgeIII}. 
It follows that $H^{*}_{sing}(|X^H(\C)_{\bullet}|,A) \to H^{*}_{sing}(X^H(\C),A)$ is an isomorphism for all abelian groups $A$. 
It follows that  
$S^{2|V^H|}\wedge |X^H(\C)_{\bullet\,+}| \to S^{2|V^H|}\wedge X^H(\C)_{+}$ 
is a
homology isomorphism. Since $|V^H|\geq 1$, these are simply connected spaces and thus this homology isomorphism is a weak equivalence. 
\end{proof}

\begin{lemma}\label{lem:Nsmash}
	\begin{enumerate}
 \item Let $W$,$V$ be complex representations with $V$ containing a trivial summand and $X$ a smooth quasi-projective complex variety with $G$-action. Then for $n\geq 0$,
 $$
\L\RRe_{\C}(\Sigma_{T^V}\Sym_{\bullet}^{n}(\Sigma_{T^W}X_+))
\to 
\Sigma_{S^{V(\C)}}\Sym_{\bullet}^{n}(\Sigma_{S^{W(\C)}} X(\C)_+)
$$ 
is an equivariant weak equivalence in $G\topp_{\bullet}$.
\item  Let $W$,$V$ be real representations with $V$ containing a trivial summand and $X$ a smooth quasi-projective real variety with $G$-action. Then for $n\geq 0$,
$$
\L\RRe_{\C,\,\Sigma_2}(\Sigma_{T^V}\Sym_{\bullet}^{n}(\Sigma_{T^W}X_+))
\to 
\Sigma_{S^{V(\C)}}\Sym_{\bullet}^{n}(\Sigma_{S^{W(\C)}} X(\C)_+)
$$ 
is an equivariant weak equivalence in $(G\times\Sigma_2)\topp_{\bullet}$.
\end{enumerate}
\end{lemma}
\begin{proof}
We treat the complex case, the real case is similar. 
 From  \eqref{eqn:symcoeq}, we see that $\L\RRe_{\C,\,\Sigma_2}(\Sigma_{T^V}\Sym_{\bullet}^{n}(\Sigma_{T^W}X_+))$
is the homotopy coequalizer of
\[
\L\RRe_{\C,\,\Sigma_2}(\Sigma_{T^V}\Sym^n(A\coprod Y)) \rightrightarrows \L\RRe_{\C,\,\Sigma_2}(\Sigma_{T^V}\Sym^n(Y)). 
\]
The result now follows from \aref{lem:n1}.

\end{proof}

\begin{corollary}\label{cor:symrealize}
\hspace{2em}
\begin{enumerate}
\item Let $X$ be a smooth quasi-projective complex variety with $G$-action. 
For any $n,m\geq 0$ the natural map 
$$
\L\RRe_{\C}(\Sigma^{\infty}_{T^{\rho_G}}\Sym_{\bullet}^{n}(\Sigma^{m}_{T^{\rho_{G}}}X_+))
\to 
\Sigma^{\infty}_{S^{2\rho_{G}}}\Sym_{\bullet}^{n}(\Sigma^{m}_{S^{2\rho_{G}}}X(\C)_+)
$$ 
is a stable equivariant equivalence.
\item Let $X$ be a smooth quasi-projective real variety with $G$-action. For any
$n,m\geq 0$ the natural map 
$$
\L\RRe_{\C,\Sigma_2}(\Sigma^{\infty}_{T^{\rho_G}}\Sym_{\bullet}^{n}(\Sigma^{m}_{T^{\rho_{G}}}X_+))
\to 
\Sigma^{\infty}_{S^{2\rho_{G}}}\Sym_{\bullet}^{n}(\Sigma^{m}_{S^{2\rho_{G}}}X(\C)_+)
$$ 
is a stable equivariant equivalence.
\end{enumerate}
\end{corollary}

Define
$$
(\Sigma^{\infty}_{T^{\rho_{G}}}X)_{tr}^{eff} = ( \Sym^{\infty}_{\bullet}(X_+), \Sym_{\bullet}^{\infty}( \Sigma_{T^{\rho_{G}}}X_+ ), \Sym^{\infty}_{\bullet}(\Sigma^{2}_{T^{\rho_{G}}}X_+),\ldots)
$$
with obvious structure maps.

\begin{theorem}\label{thm:realbred}
There is an isomorphism $\RRe_{\C}(\M{A}) \cong {\rm H}\ul{A}$ in $\SH_{G}$, for any abelian group $A$. 
Similarly, there is an isomorphism  $\RRe_{\C,\Sigma_2}(\M{A}) \cong {\rm H}\ul{A}$ in $\SH_{G\times\Sigma_2}$.
\end{theorem} 
\begin{proof}
We treat the complex case. The real case is similar.
	
Since $\MA = \MZ \wedge \MM A$ and ${\rm H}\ul{A} = {\rm H}\ul{\Z}\wedge \MM A$, 
where $\MM A$ is a Moore spectrum for $A$, 
and $\L\RRe_{\C}(\MM A) = \MM A$, 
it suffices to establish the result for $A=\Z$.
The map $(\Sigma^{\infty}_{T^{\rho_{G}}}(S^{0}))_{tr}^{eff} \to \MZ$ is a stable equivalence by \aref{prop:effwk}.
 
It follows from \cite[Proposition 3.7]{DS:equiDT} that the spectrum $\{\Z S^{2n\rho_{G}}\}_{n\geq 0}$ is a $S^{2\rho_{G}}$-spectrum model for ${\rm H}\ul{\Z}$.
It follows from \cite[Proposition A.6]{Dug:kr} that the natural map 
$(\Sigma^{\infty}_{S^{2\rho_{G}}}(S^{0}))_{tr}^{eff}:=\{\Sym_{\bullet}^{\infty}(\Sigma_{S^{2n\rho_{G}}}S^0)\} \to \{\Z S^{2n\rho_{G}}\}_{n\geq 0}$ is an equivariant stable equivalence. 
It thus suffices to see that $\L\RRe_{\C}(\Sigma^{\infty}_{T^{\rho_{G}}}(S^{0}))_{tr}^{eff}\to (\Sigma^{\infty}_{S^{2\rho_{G}}}(S^{0}))_{tr}^{eff}$ is a stable equivalence.
 
We have the natural isomorphism $\colim_{n}(\Sigma^{\infty}_{T^{\rho_G}}E_{n})[-n]\iso E$ in $\SH_{k}(G)$, 
where $D[n]$ is the shifted $T^{\rho_{G}}$-spectrum given by $(D[n])_{i} = D_{i-n}$. 
Similarly, 
we have the natural isomorphism $\colim_{n}(\Sigma^{\infty}_{S^{2\rho_{G}}}F_{n})[-n]\iso F$ in $\SH_{G}$. 
Since $\L\RRe_{\C}$ preserves homotopy colimits and shifts, 
the result follows from \aref{cor:symrealize}.
\end{proof}

{\bf Acknowledgments}
We thank the anonymous referee for a careful reading, spotting a mistake in the appendix, and many helpful comments which have improved the exposition. 
The authors gratefully acknowledge support from the RCN project Special Geometries, no. 239015 and the RCN Frontier Research Group Project no. 250399 ``Motivic Hopf equations." 
Heller was supported by NSF Grant DMS-1710966.
{\O}stv{\ae}r was supported by a Friedrich Wilhelm Bessel Research Award from the Humboldt Foundation.
\providecommand{\bysame}{\leavevmode\hbox to3em{\hrulefill}\thinspace}
\providecommand{\MR}{\relax\ifhmode\unskip\space\fi MR }
\providecommand{\MRhref}[2]{%
  \href{http://www.ams.org/mathscinet-getitem?mr=#1}{#2}
}
\providecommand{\href}[2]{#2}

\vspace{0.2in}

\end{document}